\documentclass[12pt]{amsart}
\usepackage{amssymb}
\usepackage{amscd}
\usepackage{graphics}
\usepackage{amsmath}
\usepackage{epsfig, graphics}
\usepackage[pdftex,colorlinks]{hyperref}

\newtheorem{prop}{Proposition}[section]

\theoremstyle{definition}
\newtheorem{Ack}[prop]{Acknowledgments}

\newtheorem{theorem}[prop]{Theorem}
\newtheorem{proposition}[prop]{Proposition}
\newtheorem{lemma}[prop]{Lemma}
\newtheorem{corollary}[prop]{Corollary}
\newtheorem{remark}[prop]{Remark}

\begin{document}
\author{Emmanuel Breuillard}
\address{Emmanuel Breuillard, Universit\'{e} Paris Sud, Orsay, France}
\email{emmanuel.breuillard@math.u-psud.fr}



\title[A height gap theorem for finite subsets of $GL_{d}(\overline{\Bbb{Q}})$]
{A height gap theorem for finite subsets of $GL_{d}(\overline{\Bbb{Q}})$
and non amenable subgroups}

\begin{abstract}
We introduce a conjugation invariant normalized height $\widehat{h}(F)$
on finite subsets of matrices $F$ in $GL_{d}(\overline{\Bbb{Q}})$  and describe its properties. In particular,
we prove an analogue of the Lehmer problem for this height by showing that
$\widehat{h}(F)>\varepsilon $ whenever $F$ generates a non-virtually solvable
subgroup of $GL_{d}(\overline{\Bbb{Q}}),$ where $\varepsilon =\varepsilon (d)>0$
is an absolute constant. This can be seen as a global adelic analog of the classical
Margulis Lemma from hyperbolic geometry. As an application we prove a uniform version of the
classical Burnside-Schur theorem on torsion linear groups. In a companion paper we will apply these results
to prove a strong uniform version of the Tits alternative.
\end{abstract}
\maketitle%

\setcounter{tocdepth}{1}
\tableofcontents%

\section{Introduction}

According to the Lehmer conjecture, the absolute Weil height times the
degree of an algebraic number $x\in \overline{\mathbb{Q}}$ which is not a
root of unity ought to be bounded below by an absolute constant. Various
generalizations and extensions of this problem have been recently studied by
a variety of authors in particular in the setting of abelian varieties (e.g.
\cite{SUZ},\cite{Zhang0}) and also in connection with the dynamics of
iterated polynomial maps (e.g. \cite{FL}, \cite{Cham}, \cite{BakRum}, \cite%
{Pi}). In the present paper, we will introduce yet another height function $%
\widehat{h}(F)$ which is well suited to the study of the geometric and
arithmetic behavior of power sets $F^{n}=F\cdot ...\cdot F$ for $n\in
\mathbb{N}$, where $F$ is a finite subset of $GL_{d}(\overline{\mathbb{Q}}).$
We will investigate its properties, in particular describe when it might
become small and then prove a statement analogous to the Lehmer conjecture
in this setting. In fact, we will prove that if $\mathbb{G}$ is the Zariski
closure of the subgroup generated by $F$, then $\widehat{h}(F)$ is always
bounded away from zero by a positive constant $\varepsilon =\varepsilon
(d)>0 $ unless the connected component of the identity $\mathbb{G}^{0}$ is
solvable. While if $\mathbb{G}^{0}$ is solvable, proving a lower bound on $%
\widehat{h}(F)$ boils down to the original Lehmer conjecture. Before we
explain our motivations for studying this object, and present the main
results of the paper, let us first define it.

\bigskip

\textit{Definitions.}

Let $d\geq 1$ be an integer, $\overline{\mathbb{Q}}$ be the field of
algebraic numbers, and $K\leq \overline{\mathbb{Q}}$ a number field. We let $%
V_{K}$ be the set of equivalence classes of absolute values on $K$ and $%
n_{v}=[K_{v}:\mathbb{Q}_{p}]$ the degree of the completion $K_{v}$ of $K$
over the closure $\mathbb{Q}_{p}$ of $\mathbb{Q}$ in $K_{v}.$ We normalise
the absolute value $|\cdot |_{v}$ on $K_{v}$ so that its restriction to $%
\mathbb{Q}_{p}$ is the standard absolute value, i.e. $|p|_{v}=\frac{1}{p}.$
To any finite subset $F$ of square matrices in $M_{d}(K)$ we associate the
following \textbf{height}
\begin{equation}
h(F)=\frac{1}{[K:\mathbb{Q}]}\sum_{v\in V_{K}}n_{v}\log ^{+}||F||_{v}
\label{heightdefintro}
\end{equation}%
where $\log ^{+}=\max \{0,\log \}$ and $||F||_{v}=\max \{||f||_{v},f\in F\}.$
Here $||f||_{v}$ is the operator norm on $M_{d}(K_{v})$ associated to the
\textit{standard norm} on $K_{v}^{d}.$ We define the standard norm for $x\in
K_{v}^{d}$ to be the sup norm $||x||_{v}=\max_{1\leq i\leq d}|x_{i}|_{v}$ if
$v$ is ultrametric and the Euclidean norm $||x||_{v}=\sqrt{%
\sum_{i=1}^{d}|x_{i}|_{v}^{2}}$ otherwise. If $d=1,$ this notion coincides
with the (absolute, logarithmic) Weil height of an algebraic number (see
e.g. \cite{Bom}).

We can now define the \textbf{normalized height} $\widehat{h}(F)$ as
\begin{equation*}
\widehat{h}(F)=\lim_{n\rightarrow +\infty }\frac{1}{n}h(F^{n})
\end{equation*}%
This limit exists by subadditivity. Unlike $h(F)$, $\widehat{h}(F)$ is
independent of the choice of basis of $K_{v}^{d}$ used to define the norms $%
||x||_{v}.$

Another way to describe $\widehat{h}(F)$ is in terms of spectral radius (see
Section \ref{heightdef} below) ; for instance if $F=\{A\}$ is a singleton,
then $\widehat{h}(F)=h([1,\lambda _{1},...,\lambda _{d}])$, where $(\lambda
_{1},...,\lambda _{d})$ are the eigenvalues of $A$ and $h([1,\lambda
_{1},...,\lambda _{d}])$ the standard Weil height of the point $[1,\lambda
_{1},...,\lambda _{d}]$ in the projective space $\mathbb{P}^{d}(\overline{%
\mathbb{Q}})$ as defined in \cite[\S 1.5.]{Bom}. This connection was first
described by V. Talamanca in \cite{Tala}, where a closely related definition
of the height and normalized height of a single matrix is given (see Remark %
\ref{tala} below).

The normalized height is an invariant of the diagonal action by conjugation
of $GL_{d}$ on $GL_{d}^{k},$ where $k=Card(F)$, and it is a measure of the
combined \textit{spectral radius} of $F$ (i.e. the rate of exponential
growth of $||F^{n}||_{v}$) at all places $v,$ where $v$ varies among all
possible equivalence classes of non trivial absolute values on the number
field generated by the matrix coefficients of $F$.

\bigskip

\textit{Basic properties and height gap.}

Here are a few sample properties which are satisfied by the normalized
height. We have : $\widehat{h}(F^{n})=n\cdot \widehat{h}(F)$ for $n\in
\mathbb{N}$. A finite set $F$ satisfies $\widehat{h}(F)=0$ if and only if $F$
generates a quasi-unipotent subgroup, i.e. a group all of whose elements
have only roots of unity as eigenvalues (Proposition \ref{zeroheight}).
Moreover, the following holds:

\begin{proposition}
\label{minh}There is a constant $C=C(d)>0$ such that if $F$ is a finite
subset of $GL_{d}(\overline{\mathbb{Q}})$ generating a subgroup whose
Zariski closure is semisimple, then%
\begin{equation*}
\widehat{h}(F)\leq \inf_{g\in GL_{d}(\overline{\mathbb{Q}})}h(gFg^{-1})\leq
C\cdot \widehat{h}(F)
\end{equation*}
\end{proposition}

In other words, $F$ can always be conjugated back in a good position where
its height is comparable to its normalized height. Also $\widehat{h}$ has the Northcott property (cf. [7]) in the sense that a subset of $GL_d(\overline{\mathbb{Q}})$ whose cardinality and normalized height are bounded, which generates a subgroup with semisimple Zariski closure, and which has all its matrix coefficients of bounded degree over $\mathbb{Q}$, must belong to a bounded finite family of conjugacy classes of such sets.

The main result of this paper establishes the existence of a uniform gap for
the normalized height of subsets $F$ generating a non amenable subgroup of $%
GL_{d}(\overline{\mathbb{Q}}).$ We have:

\begin{theorem}
\label{main0}There is a constant $\varepsilon =\varepsilon (d)>0$ such that
if $F$ is a finite subset of $GL_{d}(\overline{\mathbb{Q}})$ generating a
non amenable subgroup that acts strongly irreducibly, then $\widehat{h}%
(F)>\varepsilon .$
\end{theorem}

The constant $\varepsilon (d)$ can be made explicit in principle, although
we make no attempt here to give a lower bound (see Remark \ref{effec}).

Recall that, as follows for instance from the Tits alternative (\cite{Tits}%
), amenable subgroups of $GL_{d}(\overline{\mathbb{Q}})$ are precisely the
virtually solvable subgroups, i.e. those subgroups which contain a solvable
subgroup of finite index.

Note that if $d=1,$ then $\widehat{h}$ coincides with the classical Weil
height of a non-zero algebraic number. Of course $GL_{1}(\overline{\mathbb{Q}%
})$ is solvable (it is a torus) and no uniform lower bound on the height can
exist there. However the Lehmer conjecture states that one ought to have $%
h(x)\geq \frac{c}{\deg (x)}$ for some absolute constant $c>0$ whenever $x$
is not a root of unity. We refer the reader to \cite{Sm} for a recent survey
on this conjecture (see also \cite{Bom}) and to \cite{AmDv}, \cite{AmoD} and
for recent progress. Theorem \ref{main0} can thus be seen as a positive
solution to a Lehmer type problem in semisimple algebraic groups as opposed to tori.

As it turns out, for each integer $k\geq 2,$ the set of $k$-tuples $F$ in $%
GL_{d}(\overline{\mathbb{Q}})$ which generate a virtually solvable subgroup
forms a closed algebraic subvariety of $GL_{d}(\overline{\mathbb{Q}})^{k}.$
Therefore Theorem \ref{main0} implies that the set of points with small
normalized height in $GL_{d}(\overline{\mathbb{Q}})^{k}$ is not
Zariski-dense. This is reminiscent of the Bogomolov conjecture proved by
Ullmo and Zhang (see \cite{Ull}, \cite{Zhang0}, \cite{SUZ}), which asserts
that, given an abelian variety, the set of points with small N\'{e}ron-Tate
height on an algebraic subvariety which is not a finite union of torsion
cosets of abelian subvarieties is not Zariski-dense. In fact the toric
version of the Bogomolov conjecture, proved by Zhang in \cite{Zhang}, will
be a key ingredient of the proof of Theorem \ref{main0}.

\begin{remark}
A competing definition of the normalized height $\widehat{h}(F)$ consists in
replacing $\log ^{+}$ by $\log $ in $(\ref{heightdefintro})$. The two
definitions coincide if $F\subset SL_{d},$ but may differ otherwise. However
the difference is minor and we found it more convenient to work with $\log
^{+},$ because all terms are then non negative, although many results, such
as Theorem \ref{main0}, also hold for this other definition of the height
(see the discussion in Remark \ref{loglog}).
\end{remark}

\bigskip

\textit{Motivation and consequences.}

In \cite{BreSol} we established a connection between the Lehmer conjecture
and the uniform exponential growth problem for linear solvable groups. More
precisely, we showed that proving uniform exponential growth over all
solvable subgroups of $GL_{2}(\mathbb{C})$, that is showing the existence of
an absolute constant $c>0$ such that $\lim_{n\rightarrow +\infty }|F^{n}|^{%
\frac{1}{n}}>c$ whenever $F$ generates a solvable non virtually nilpotent
subgroup of $GL_{2}(\mathbb{C})$ would imply the Lehmer conjecture.

We have not settled the issue of whether or whether not the Lehmer conjecture is in fact equivalent to the uniform exponential growth of solvable subgroups of $GL_2(\mathbb{C})$. However in our companion paper \cite{B}, we make use of Theorem \ref{main0} (height gap theorem) and Proposition \ref{minh} above to establish the following strengthening of the classical Tits alternative, which among other things implies the existence of a constant $c = c(d) > 0$ such that $lim_{n \rightarrow +\infty} |F^n|^{\frac{1}{n}} >c$ whenever $F$ generates a nonvirtually solvable subgroup of $GL_d(\mathbb{C}
)$.

\begin{theorem}[Uniform Tits alternative, \cite{B}] There is $N=N(d) \in \mathbb{N}$, such that if $K$ is a field and $F$ a finite symmetric subset of $GL_d(K)$ containing $1$ which generates a nonvirtually solvable subgroup, then $F^N$ contains two elements $a,b$ which generate a non-abelian free subgroup.
\end{theorem}

In the same vein, but much more straightforwardly, one obtains the following
corollary, which answers a question from \cite{BarCor} and is a
strengthening of a well known theorem of Burnside and Schur (see \cite{CurR}) asserting that finitely generated linear torsion groups are finite.

\begin{corollary}
(Effective Schur) There is an integer $N=N(d)\in \mathbb{N}$ such that if $K$
is a field and $F$ is a finite subset of $GL_{d}(K)$ which generates an
infinite subgroup, then $(F\cup F^{-1})^{N}$ contains an element of infinite
order.
\end{corollary}

The positive characteristic case of the above corollary is easy to prove,
while the characteristic zero case relies on our Theorem \ref{main0}.

The interpretation of $\widehat{h}(F)$ in terms of spectral radius allows us
to derive the following:

\begin{corollary}
\label{eig0}There are constants $N_{1}=N_{1}(d)\in \mathbb{N}$, $C=C(d)\in
\mathbb{N}$ such that if $F$ is any finite subset of $GL_{d}(\overline{%
\mathbb{Q}})$ containing $1$, there is some $a\in F^{N_{1}}$ and some
eigenvalue $\lambda $ of $a$ such that
\begin{equation*}
h(\lambda )\geq \frac{1}{|F|^{C}}\cdot \widehat{h}(F).
\end{equation*}
\end{corollary}

\begin{corollary}
\label{eig1}There are constants $N_{1}=N_{1}(d)\in \mathbb{N}$, $\varepsilon
=\varepsilon (d)>0$ such that if $F$ is any finite subset of $GL_{d}(%
\overline{\mathbb{Q}})$ containing $1$ and generating a non virtually
solvable subgroup, then we may find $a\in F^{N_{1}}$ and an eigenvalue $%
\lambda $ of $a$ such that $h(\lambda )>\varepsilon ,$ for some fixed $%
\varepsilon =\varepsilon (d)>0.$
\end{corollary}

This follows easily from Corollary \ref{eig0}, Theorem \ref{main0} and the
following fact that we prove along the way to the proof of Theorem \ref%
{main0} (see Proposition \ref{twoelements}):

\begin{proposition}
\label{2elem}Let $\mathbb{G}$ be a connected semisimple algebraic group over
an algebraically closed field of characteristic 0. There is a constant $%
c=c(d)\in \mathbb{N}$, where $d=\dim \mathbb{G}$, such that the following
holds. Let $F$ be a finite subset of $\mathbb{G}$ containing $1$ and
generating a Zariski-dense subgroup. Then $F^{c(d)}$ contains two elements $%
a $ and $b$ which generate a Zariski dense subgroup of $\mathbb{G}$.
\end{proposition}

N.B. This proposition also holds in positive characteristic, but the proof,
given in our companion paper \cite{B}, is more involved. See Remark \ref%
{poscar} for more on positive characteristic.

Corollaries \ref{eig0} and \ref{eig1} allow us to construct a short
(positive) word $w$ with letters in $F$ which has an eigenvalue of large
height. The length of the word is bounded by an absolute constant $%
N_{1}=N_{1}(d).$ This type of result is crucial in order to build the
so-called proximal elements which are needed in various situations, in
particular in the applications to the Tits alternative given in \cite{B}.

In the same vein we have:

\begin{corollary}
\label{corobv}There is a constant $N_{2}=N_{2}(d)\in \mathbb{N},$ such that
if $F$ is a finite subset of $GL_{d}(\mathbb{C})$ containing $1$ which generates a non
virtually solvable subgroup, then there is a matrix $w\in F^{N_{2}}$ with an
eigenvalue $\lambda $ such that : either there exists an ultrametric
absolute value $|\cdot |_{v}$ on $\mathbb{Q}(\lambda )$ such that $|\lambda
|_{v}>1,$ or there is a field homomorphism $\sigma :\mathbb{Q}(\lambda
)\hookrightarrow \mathbb{C}$ such that $|\sigma (\lambda )|\geq 2.$
\end{corollary}

In particular, if $\mathcal{O}$ is the ring of all algebraic integers, there
is an integer $N_{1}=N_{1}(d)\in \mathbb{N}$ such that if $F$ is a finite
set of $SL_{d}(\mathcal{O})$ containing $1,$ either $F$ generates a
virtually solvable subgroup, or there is an archimedean absolute value $v$
on $\overline{\mathbb{Q}}$ extending the canonical absolute value on $%
\mathbb{Q}$ and a matrix $f\in F^{N_{1}}$ with at least one eigenvalue of $v$%
-absolute value $\geq 2.$ Observe that this fails for arbitrary finite
subsets of $SL_{d}(\overline{\mathbb{Q}}).$ For instance $SL_{3}(\mathbb{Q}%
)\cap SO(3,\mathbb{R})$ is dense in $SO(3,\mathbb{R})$ and contains a
finitely generated dense subgroup.\newline

\bigskip

\textit{Geometric Interpretation and the Margulis Lemma.}

Theorem \ref{main0} has also the following geometric interpretation. Recall
that the classical Margulis Lemma (see \cite{Thu}) asserts that if $S=%
\mathbb{H}^{n}$ is the hyperbolic $n$-space, or more generally any real
symmetric space of non compact type endowed with its Riemannian metric $d$,
then there is a positive constant $\varepsilon =\varepsilon (S)>0$ such that
the following holds: suppose $F$ is a finite set of isometries of $S$ such
that $\max_{f\in F}d(f\cdot x,x)<\varepsilon $ for some point $x\in S$ and
suppose $F$ lies in a discrete subgroup of isometries of $S$, then $F$
generates a virtually nilpotent subgroup. This lemma has several important
consequences for the geometry and topology of hyperbolic manifolds and
locally symmetric spaces, such as the structure of cusps and the thick-thin
decomposition (\cite{Thu}), or lower bounds for the covolume of lattices in
semisimple Lie groups (see \cite{Wang}, \cite{KM}, \cite{Gel}).

What happens if one removes the discreteness assumption on the group
generated by $F$ and assumes instead that $F$ consists of elements which are
rational over some number field $K$ ? Of course the Margulis Lemma no longer
holds as such, in particular because $\varepsilon (S)$ tends to $0$ as $\dim
S$ tends to infinity. However Theorem \ref{main0} gives a kind of
substitute. As will be shown below (see Section \ref{heightdef}) the
normalized height $\widehat{h}(F)$ is always bounded above by the quantity $%
e(F),$ which we call minimal height, and which encodes, as a weighted sum
over all places $v\in V_{K},$ the minimal displacement of $F$ on each
symmetric space or Bruhat-Tits building $X_{v}$ associated to $%
SL_{d}(K_{v}). $ In particular the height gap $\widehat{h}(F)>\varepsilon $
obtained in Theorem \ref{main0} implies that there always is a natural space
$X_{v}$ (symmetric space or Bruhat-Tits building of $SL_{d}$) where $F$ acts
with a large displacement. More precisely:

\begin{corollary}
\label{geomint}Let $d\in \mathbb{N}$ and for a local field $k$ let us denote
by $X_{k}$ the symmetric space or Bruhat-Tits building of $PGL_{d}(k).$ We
let $d(\cdot ,\cdot )$ be a left invariant Riemannian metric on $X_{\mathbb{C}}$. There is a constant $\varepsilon =\varepsilon (d)>0$ with the following
property. Let $K$ be a number field and $F$ a finite subset of $SL_{d}(K)$
which generates a non virtually solvable subgroup $\Gamma $, then either for
some finite place $v$ of $K$, the subgroup $\Gamma $ acts (simplicially)
without global fixed point on the Bruhat-Tits building $X_{K_{v}},$ or for
some embedding $\sigma :K\hookrightarrow \mathbb{C}$
\begin{equation*}
\inf_{x\in X_{\mathbb{C}}}\max_{f\in F}d(\sigma (f)\cdot x,x)>\varepsilon .
\end{equation*}
\end{corollary}

The crucial point here of course is that $\varepsilon $ is independent of
the number field $K$. Thus Theorem \ref{main0} can be seen as a uniform
Margulis Lemma for all $S$-arithmetic lattices of a given Lie type. For
example, it is uniform over all $SL_{2}(\mathcal{O}_{K})$ where $K$ can vary
among all number fields, even though those groups can be lattices of
arbitrarily large rank.\newline

\bigskip

\textit{Outline of the proof of Theorem \ref{main0}.}

The first part of the proof consists in reducing to the situation when $F$
is a $2$-element set $F=\{A,B\},$ where $A$ and $B$ are two regular
semisimple elements in an absolutely almost simple algebraic group $\mathbb{G%
}$ of adjoint type and $F$ generates a Zariski-dense subgroup of $\mathbb{G}$%
. It is not hard to see that the existence of a gap for $\widehat{h}(F)$
when computed in the adjoint representation of $\mathbb{G}$ implies the
existence of a gap for $\widehat{h}(F)$ when computed in any finite
dimensional linear representation of $\mathbb{G}$. We thus reduce to the
adjoint representation of $\mathbb{G}$. The reduction from an arbitrary
finite set $F$ to a $2$-element set makes use of a lemma due to
Eskin-Mozes-Oh \cite{EMO} (\textquotedblleft escaping
subvarieties\textquotedblright\ Lemma \ref{Bezout}), which, given any non
trivial algebraic relation between pairs $\{x,y\}$ of elements in $\mathbb{G}
$, produces two short words in $\{x,y\}$ which no longer satisfy this
relation. This lemma is also used later on and is an essential tool here.

As we mentioned above, one may interpret $\widehat{h}(F)$ in terms of the
combined minimal displacement $e(F)$ of $F$ on all symmetric spaces and
Bruhat-Tits buildings that arise through the various completions of the
number field. The quantity $e(F)$ is defined as the weighted sum of the
logarithm of the minimal norms $E_{v}(F)=\inf \{||gFg^{-1}||_{v},g\in GL_{d}(%
\overline{K_{v}})\}.$ Crucial to this correspondence is a spectral radius
formula for sets of matrices (Lemma \ref{CompLem} below) which compares the
minimal displacement of $F$ (or equivalently $E_{v}(F)$) with the minimal
displacement of each individual matrix in the power set $F^{d^{2}}$ (or
equivalently its maximal eigenvalue)$.$ As a consequence, $\widehat{h}(F)$
is small if and only if $e(F)$ is small.

In the second part of the proof, we fix a place $v$ and work in $\mathbb{G}%
(K_{v}).$ Given $A,B$ in $\mathbb{G}(K_{v})$, with $A$ in a maximal torus $T$
of $\mathbb{G}(K_{v}),$ we obtain local estimates for the minimal
displacement of the action of $B$ restricted to the maximal flat associated
to $T.$ These estimates are obtained via the Iwasawa decomposition working
our way through all positive roots of $A$ starting from the maximal one. At
the end we get an upper bound for $\inf_{t_{v}\in T}||t_{v}Bt_{v}^{-1}||_{v}$
which involves $E_{v}(F)$ on the one hand and the gap $\left| 1-\alpha
(A)\right| _{v}$ between the roots of $\alpha (A)$ and $1$ on the other hand.

In the last part of the proof, we put all our local estimates together and
make crucial use of the product formula, so as to obtain an upper bound for
the weighted sum of all $\inf_{t_{v}\in T}\log ||t_{v}Bt_{v}^{-1}||_{v}$ in
terms of $e(F)$ and the average of the $\log |1-\alpha (A)|_{v}$ over all
archimedean places $v,$ for each root $\alpha $. When $e(F)$ is small this
upper bound becomes also small. Indeed, since the height of each $\alpha (A)$
is small, we can invoke Bilu's equidistribution theorem : the Galois
conjugates of $\alpha (A)$ equidistribute on the unit circle (\cite{Bilu}).
Hence the average of the $\log |1-\alpha (A)|_{v}$'s gives a negligible
contribution.

Finally, considering a suitably chosen regular map $f$ on $\mathbb{G}$ which
is invariant under conjugation by the elements of $T$ (a suitable matrix
coefficient of $B$ will do), we use the above upper bound to show that the
height of $f(B)$ as well as $f(B^{i})$ for larger and larger $i\in \mathbb{N}
$, becomes small when $e(F)$ is small. However, by a theorem of Zhang \cite%
{Zhang} on small points of algebraic tori, this must force a non trivial
algebraic relation between the $f(B^{i})$'s. Finally the Eskin-Mozes-Oh
lemma quoted above provides the desired contradiction, as we may have chosen
$F=\{A,B\}$ to avoid this relation to begin with.

The reader can also consult \cite{Bre}, where we gave the full details of
the proof in the special case of $GL_{2}.$

\bigskip

\textit{Outline of the paper.}

Section 2 is devoted to the definition of the normalized and minimal heights
and the derivation of their most basic properties. The main results of this
section are the spectral radius formula for several matrices (Lemma \ref%
{CompLem} below) and Proposition \ref{smallRE2}, which gives a lower bound
on the displacement of the power set $F^{n}.$ These facts will enable us to
compare the normalized height with the minimal height and to reinterpret the
normalized height in terms of adelic displacement.

In Section 3, we state our main results in full detail. Their proof occupies
the remainder of the paper. Section 4 gives the main reduction step from an
arbitrary subset of $GL_{d}(\overline{\mathbb{Q}})$ to a subset consisting
of two elements $F=\{A,B\}$ which generates a Zariski-dense subgroup of a
simple algebraic group $\mathbb{G}$. We also prove there the comparison
statement between different linear representations (Proposition \ref%
{minheight}). The geometric interpretation in terms of displacement is also
made precise at the end of Section 4.

In Section 5, we pick a Chevalley basis for the adjoint representation of $%
\mathbb{G}$ and we prove local estimates whose aim is to obtain good upper
bounds for the size of the matrix coefficients of a conjugate of $F=\{A,B\}$
which almost realizes the infimum $E_{v}(F)=\inf_{g\in GL_{d}(\overline{K_{v}%
})}||gFg^{-1}||_{v}$ in terms of $E_{v}(F)$ and the simple roots $\alpha
(A). $ These local bounds are then used and put together in Section 6 in
order to get a global bound on the height of matrix coefficients of $A$ and $%
B$ (Proposition \ref{heightbound}).

Section 7 is devoted to completing the proofs of the results stated in
Section 3. In particular, we make use of the global bound proved in Section
6 to prove Theorem \ref{main0} (height gap) and the local estimates of
Section 5 are used again to give a proof of Proposition \ref{minh} (good
position). Finally we also derive the corollaries stated in this
introduction.

\section{Minimal height and displacement\label{Sec2}}

\subsection{Local notions of minimal norm, spectral radius and minimal
displacement\label{locnot}}

Let $k$ be a local field of characteristic $0$. Let $\left\Vert \cdot
\right\Vert _{k}$ be the standard norm on $k^{d},$ that is the canonical
Euclidean (resp. Hermitian) norm if $k=\mathbb{R}$ (resp. $\mathbb{C}$) and
the sup norm ($\left\Vert x\right\Vert _{k}=\max_{i}|x_{i}|_{k}$) if $k$ is
non Archimedean. We will also denote by $\left\Vert \cdot \right\Vert _{k}$
the operator norm induced on the space of $d$ by $d$ matrices $M_{d}(k)$ by
the standard norm $\left\Vert \cdot \right\Vert _{k}$ on $k^{d}.$ Let $Q$ be
a bounded subset of matrices in $M_{d}(k)$. We set
\begin{equation*}
\left\Vert Q\right\Vert _{k}=\sup_{g\in Q}\left\Vert g\right\Vert _{k}
\end{equation*}%
and call it the \textit{norm of }$Q$. Let $\overline{k}$ be an algebraic
closure of $k.$ It is well known (see Lang's Algebra \cite[XII. 2.
Proposition 2.5. ]{Lang0}) that the absolute value on $k$ extends to a
unique absolute value on $\overline{k},$ hence the norm $\left\Vert \cdot
\right\Vert _{k}$ also extends in a natural way to $\overline{k}^{d}$ and to
$M_{d}(\overline{k}).$ This allows us to define the \textit{minimal norm} of
a bounded subset $Q$ of $M_{d}(k)$ as
\begin{equation*}
E_{k}(Q)=\inf_{x\in GL_{d}(\overline{k})}\left\Vert xQx^{-1}\right\Vert _{k}
\end{equation*}%
We will also need to consider the \textit{maximal eigenvalue of }$Q,$ namely
\begin{equation*}
\Lambda _{k}(Q)=\max \{|\lambda |_{k},\text{ }\lambda \in spec(q),q\in Q\}
\end{equation*}%
where $spec(q)$ denotes the set of eigenvalues (the spectrum) of $q$ in $%
\overline{k}.$ We also set $Q^{n}=Q\cdot ...\cdot Q$ to be the set of all
products of $n$ elements from $Q$. Finally, we introduce the \textit{%
spectral radius} of $Q,$ that is
\begin{equation*}
R_{k}(Q)=\lim_{n\rightarrow +\infty }\left\Vert Q^{n}\right\Vert _{k}^{\frac{%
1}{n}}
\end{equation*}%
in which the limit exists (and coincides with $\inf_{n\in \mathbb{N}%
}\left\Vert Q^{n}\right\Vert _{k}^{\frac{1}{n}}$) because the sequence $%
\{\left\Vert Q^{n}\right\Vert _{k}\}_{n}$ is sub-multiplicative.

These quantities are related to one another. The key property concerning
them is given in the following result, which, together with its corollary
below (Propositon \ref{properties}), we call \emph{"spectral radius formula
for several matrices"} because of its parallel with the classical spectral
radius formula relating the asymptotics of the powers of a matrix with its
maximal eigenvalue:

\begin{lemma}
\label{CompLem}\textbf{(Spectral Radius Formula for }$Q$\textbf{)} Let $Q$
be a bounded subset of $M_{d}(k)$.

$(a)$ if $k$ is non Archimedean, there is an integer $q\in \lbrack 1,d^{2}]$
such that $\Lambda _{k}(Q^{q})=E_{k}(Q)^{q}.$

$(b)$ if $k$ is Archimedean, there is a constant $c=c(d)\in (0,1)$
independent of $Q$ and an integer $q\in \lbrack 1,d^{2}]$ such that $\Lambda
_{k}(Q^{q})\geq c^{q}\cdot E_{k}(Q)^{q}.$\newline
\end{lemma}

\textbf{N.B. :} In the work of Eskin-Mozes-Oh \cite{EMO} a result of a
similar nature appears between the lines inside their argument (when they
consider almost algebras). A weaker version of this lemma (essentially part
(b)) was already used in \cite{uti}. The equality in part (a) is new and
will be crucial in our arguments.

\proof%
Let $K$ be a field. We make use of two well-known theorems. The first is a
theorem of Wedderburn (see Curtis-Reiner \cite{CurR} 27.27) that if an
algebra $A$ over $K$ has a linear basis over $K$ consisting of nilpotent
elements, then $A^{m}=0$ for some integer $m$. The second is a theorem of
Engel (see Jacobson \cite{Jab}) that if $A$ is a subset of $M_{d}(K) $ such
that $A^{m}=0$ for some integer $m,$ then $A$ can be simultaneously
conjugated in $GL_{d}(K)$ inside $N_{d}(K),$ the subalgebra of upper
triangular matrices with zeroes on and below the diagonal. Combined
together, these facts yield:

\begin{lemma}
\label{Wed+Eng}Let $K$ be a field. If $Q$ is any subset of $M_{d}(K)$ such
that $Q^{q}$ contains only nilpotent matrices for every $q,$ $1\leq q\leq
d^{2},$ then there is $g\in GL_{d}(K)$ such that $gQg^{-1}\subset N_{d}(K).$
\end{lemma}

\proof%
Since $\dim _{K}M_{d}(K)\leq d^{2},$ the $K$-algebra generated by $Q$ has a
linear basis made of elements in $\cup _{1\leq q\leq d^{2}}Q^{q}$. By
Wedderburn and Engel, the result follows.%
\endproof%

We first quickly prove $(b).$ We argue by contradiction. There is a sequence
$Q_{n}$ with $E_{k}(Q_{n})=1$ while $\max_{1\leq q\leq d^{2}}\Lambda
_{k}(Q_{n}^{q})^{\frac{1}{q}}$ tends to $0.$ Up to conjugating by some $%
g_{n}\in GL_{d}(\mathbb{C})$, we may assume that $||Q_{n}||_{\mathbb{C}}\leq
1+\frac{1}{n},$ and passing to a Hausdorff limit, we obtain a compact set $Q$
with $E_{\mathbb{C}}(Q)=||Q||_{\mathbb{C}}=1$, while $\max_{1\leq q\leq
d^{2}}\Lambda _{k}(Q^{q})^{\frac{1}{q}}=0.$ But this is a contradiction with
Lemma \ref{Wed+Eng} as $E_{\mathbb{C}}(C)=0$ for any bounded subset $C$ of $%
N_{d}(\mathbb{C}).$ This proves $(b).$

In order to prove $(a)$ we first show:

\begin{lemma}
\label{immeuble}(\textbf{small eigenvalues implies large fixed point set})
Let $d\in \mathbb{N}$. There exists an integer $N=N(d)\in \mathbb{N}$ with
the following property. Let $k$ be a non archimedean local field with
absolute value $|\cdot |_{k}$ and $\mathcal{O}_{k}$ its ring of integers.
Let $Q$ be a subset of $M_{d}(\mathcal{O}_{k})$ such that for each integer $%
q\in \lbrack 1,d^{2}]$ every element of $Q^{q}$ has all its eigenvalues of
absolute value at most $|\pi |_{k}^{N},$ where $\pi $ is a uniformizer for $%
\mathcal{O}_{k}.$ Then there is $g\in GL_{d}(k)$ such that $gQg^{-1}$
belongs to $\pi M_{d}(\mathcal{O}_{k}).$
\end{lemma}

\proof%
We argue by contradiction. This means that we have a sequence of local
fields $k_{n}$ and subsets $Q_{n}$ in $M_{d}(\mathcal{O}_{k_{n}})$ such that
$||gQ_{n}g^{-1}||_{k_{n}}\geq 1$ for all $g\in GL_{d}(k_{n})$ and all
eigenvalues of $Q_{n}^{q}$ have absolute value at most$\ |\pi
_{n}|_{k_{n}}^{n}$. Let us consider a non-principal ultrafilter $\mathcal{U}$
on $\mathbb{N}$ and form the ultraproduct ring $A=\prod_{\mathcal{U}}%
\mathcal{O}_{k_{n}}$. First let us decide that we have chosen the absolute
value $|\cdot |_{n}$ on $k_{n}$ in such a way that $|\pi _{n}|_{n}=\frac{1}{2%
}$ for every $n$ where $\pi _{n}$ is a fixed uniformizer in $\mathcal{O}%
_{k_{n}}.$ For every $x_{n}\in \mathcal{O}_{k_{n}}$ the quantity $%
|x_{n}|_{n} $ may only take values among $2^{-(\mathbb{N\cup \{\infty \})}}.$
It follows that for every $x\in A$ represented by $(x_{n})_{n\in \mathbb{N}%
}, $ the quantity $|x|:=\lim_{\mathcal{U}}|x_{n}|_{n}$, which is well
defined, may only take values in $2^{-(\mathbb{N\cup \{\infty \})}}.$
Moreover the defining properties of the absolute values $|\cdot |_{n}$ are
inherited by $|\cdot |$, that is $|xy|=|x|\cdot |y|$ and $|x+y|\leq \max
\{|x|,|y|\}$, except that there may be non zero elements $x\in A$ with $%
|x|=0.$ We will quotient these elements out. Let $I=\{x\in A,|x|=0\}.$ Then $%
I$ is clearly a prime ideal of $A$. We can now set $\mathcal{O}=A/I$, which
is a domain on which our absolute value $|\cdot |$ descends to a
well-defined absolute value, which we still denote by $|\cdot |.$ On $%
\mathcal{O}$ the absolute value $|\cdot |$ takes values in $2^{-(\mathbb{%
N\cup \{\infty \})}}$ and satisfies the standard axioms ($|xy|=|x|\cdot |y|$
; $|x+y|\leq \max \{|x|,|y|\}$; $|x|=0$ iff $x=0$) which make $\mathcal{O}$
a discrete valuation ring (see \cite[Chapter 9]{AM}) with uniformizer $\pi $
equal to the class of $(\pi _{n})_{n\in \mathbb{N}}$ in $A/I.$ Let $K$ be
the field of fractions of $\mathcal{O}$. It is a field with a non
archimedean absolute value and $\mathcal{O}=\{x\in K,|x|\leq 1\}.$ Let $Q$
be the class of $(Q_{n})_{n\in \mathbb{N}}$ in $M_{d}(\mathcal{O}).$ Then $%
Q^{q}$ is the class of $(Q_{n}^{q})_{n\in \mathbb{N}}$ for each $q$. But by
assumption $|a|_{n}\leq \frac{1}{2^{n}}$ for every non dominant coefficient $%
a$ of the characteristic polynomial of any matrix in $Q_{n}^{q}.$ It follows
that $Q^{q}$ is made of nilpotent matrices for each $q,$ $1\leq q\leq d^{2}.$
We may thus apply Lemma \ref{Wed+Eng} to $Q$ in $M_{d}(K).$ There is a
matrix $g\in GL_{d}(K)$ such that $gQg^{-1}\subset N_{d}(K).$ Write $g=\pi
^{-L}\overline{g}$ where $\overline{g}\in M_{d}(\mathcal{O}).$ There is $%
\widehat{g}\in M_{d}(\mathcal{O})$ such that $\overline{g}\widehat{g}=\det
\overline{g}$ which is the transpose of the matrix of minors. We thus have $%
\overline{g}Q\widehat{g}\subset N_{d}(\mathcal{O}).$ This means that there
is a function $f(n)$ going to $+\infty $ with $n$ such that $\overline{g}%
_{n}Q_{n}\widehat{g}_{n}\subset N_{d}(\mathcal{O}_{k_{n}})$ $\mod \pi%
_{n}^{f(n)}$ for most $n$'s (i.e. for a set of $n$'s belonging to $\mathcal{U%
}$). In particular for every $M\in \mathbb{N}$, for most $n$'s one may find
a matrix $h_{n}\in GL_{d}(k_{n})$ such that $h_{n}\overline{g}_{n}Q_{n}%
\widehat{g}_{n}h_{n}^{-1}\subset \pi _{n}^{M+1}M_{d}(\mathcal{O}_{k_{n}})$
(e.g. take $h_{n}$ diagonal with coefficients $\pi _{n}^{-i(M+1)},$ $%
i=1,...,d$). Finally note that $\det \overline{g}\in \mathcal{O}\backslash
\{0\}$ so that if $(\overline{g}_{n})_{n}$ is a representative of $\overline{%
g}$ in $M_{d}(A),$ there is $M\in \mathbb{N}$ such that $|\det \overline{g}%
_{n}|_{n}\geq 2^{-M}$ for most $n\in \mathbb{N}$. Hence $h_{n}\overline{g}%
_{n}Q_{n}\overline{g}_{n}^{-1}h_{n}^{-1}\subset \pi _{n}M_{d}(\mathcal{O}%
_{k_{n}})$ for most $n$'s, which is the desired contradiction.
\endproof%

We can now prove $(a).$ Let $\pi $ be a uniformizer for $k$ and let $\delta
\geq 0$ be such that $\max_{1\leq q\leq d^{2}}\Lambda _{k}(Q^{q})^{\frac{1}{q%
}}=|\pi |_{k}^{\delta }E_{k}(Q).$ Assume by contradiction that $\delta >0.$
Let $m\geq N(d)/\delta $. Let $k_{1}=k(\pi _{1})$ where $\pi _{1}^{m}=\pi $
and $F_{k_{1}}(Q)=\min_{x\in GL_{d}(k_{1})}\left\Vert xQx^{-1}\right\Vert
_{k_{1}}.$ Up to conjugating by $x\in GL_{d}(k_{1})$, we may assume that $%
F_{k_{1}}(Q)=||Q||_{k_{1}}\geq E_{k}(Q).$ Let $Q_{0}=\frac{Q}{q_{0}}$ for
some $q_{0}\in k_{1}$ such that $|q_{0}|_{k_{1}}=||Q||_{k_{1}}.$ Then
\begin{equation*}
\max_{1\leq q\leq d^{2}}\Lambda _{k}(Q_{0}^{q})^{\frac{1}{q}}\leq |\pi
_{1}|_{k_{1}}^{\delta m}\leq |\pi _{1}|_{k_{1}}^{N(d)}
\end{equation*}%
while $F_{k_{1}}(Q_{0})=1.$ But this obviously contradicts Lemma \ref%
{immeuble}. This ends the proof of $(a).$

\endproof%

\begin{remark}
In the proof we just gave of item $(a)$ in Lemma \ref{CompLem}, we used an
ultralimit argument in order to establish Lemma \ref{immeuble}. Passing to
ultralimits allowed us to obtain a set $Q$ made of genuinely nilpotent
(instead of almost nilpotent) matrices in the ultraproduct field $K$ and to
thereby be able to apply the theorems of Wedderburn and Engel in the field $%
K $ (i.e. Lemma \ref{Wed+Eng}). Without such a limiting object at our
disposal, we would have had to work much harder and prove an epsilon version
of the theorems of Wedderburn and Engel, where nilpotency is replaced by $%
\varepsilon $-nilpotency (see Remark \ref{effec} below). Of course the use
of ultralimits has the drawback that the constant $N(d)$ we get in Lemma \ref%
{immeuble} is non effective. However this non-effectiveness has no effect
for our purposes (and no effect on the effectivity of the height gap $%
\varepsilon (d)$ from Theorem \ref{main0}) because only the equality
obtained in Lemma \ref{CompLem} $(a)$ (and not the constant $N(d)$ of Lemma %
\ref{immeuble}) will be used later. See \cite{Bre} for an alternative
argument for $2$ by $2$ matrices.
\end{remark}

\begin{remark}
\label{effec}The proof of item $(b)$ in Lemma \ref{CompLem} was by
contradiction and gave no indication about how large $c$ is. This is in fact
the only place in this paper (and hence in the determination of the height
gap $\varepsilon (d)$ from Theorem \ref{main0}) where we have a constant
which is not explicitable in principle. However we can give another proof of
$(b)$ which is constructive and gives a lower bound of order $\exp
(-d^{d^{2}})$ for $c(d)$. We do not include this proof here because it is
much lengthier and requires to prove an approximate version of the theorems
of Wedderburn and Engel valid for a set of matrices $Q$ such that each $%
Q^{q} $ is made of $\varepsilon $-nilpotent matrices (i.e. matrices all of
whose eigenvalues have modulus $\leq \varepsilon $). Details can be found in
\cite{BreEffec}.
\end{remark}

\begin{remark}
Although we will not need this in the sequel, we observe in passing and also
to justify the title of Lemma \ref{immeuble} that it has the following
geometric interpretation in terms of the Bruhat-Tits building $\mathcal{BT}%
(GL_{d},k)$ of $GL_{d}(k).$ Let $S$ be a bounded subset of $GL_{d}(k).$ If
every element of $S^{q},$ $q\in \lbrack 1,d^{2}],$ fixes pointwise a ball of
radius $n$ in $\mathcal{BT}(GL_{d},k)$ for the combinatorial distance, then
there is a common ball of radius $\Omega _{d}(n)$ which is fixed pointwise
by all elements in $S.$ This statement does not follow directly from Lemmas %
\ref{Wed+Eng} and \ref{immeuble}, but from a simple modification of these
lemmas, where one considers the $k$-algebra generated by the $S^{q}-I_{d},$ $%
q\in \lbrack 1,d^{2}]$ in $M_{d}(k)$ in place of the one generated by the $%
Q^{q}$ as in the proof of Lemma \ref{immeuble}$.$
\end{remark}

With the spectral radius formula at our disposal, that is Lemma \ref{CompLem}%
, we can now understand the relationships between the various quantities at
hand, i.e. the minimal norm, spectral radius and maximal eigenvalue.

\begin{proposition}
\label{properties}Let $Q$ be a bounded subset of $M_{d}(k)$. We have

$(i)$ $\Lambda _{k}(Q)\leq R_{k}(Q)\leq E_{k}(Q)\leq \left\Vert Q\right\Vert
_{k},$ and $R_{k}(gQg^{-1})=R_{k}(Q)$ for any $g\in GL_{d}(\overline{k}),$

$(ii)$ $\Lambda _{k}(Q^{n})\geq \Lambda _{k}(Q)^{n}$, $E_{k}(Q^{n})\leq
E_{k}(Q)^{n}$ and $R_{k}(Q^{n})=R_{k}(Q)^{n}$ $\forall n\in \mathbb{N}$,

$(iii)$ $R_{k}(Q)=\lim_{n\rightarrow +\infty }E_{k}(Q^{n})^{\frac{1}{n}%
}=\inf_{n\in \mathbb{N}}E_{k}(Q^{n})^{\frac{1}{n}},$

$(iv)$ $R_{k}(Q)=\sup_{n\in \mathbb{N}}\Lambda _{k}(Q^{n})^{\frac{1}{n}},$

$(v)$ if $k$ is non Archimedean, $R_{k}(Q)=E_{k}(Q)$,

$(vi)$ if $k$ is Archimedean, $c\cdot E_{k}(Q)\leq R_{k}(Q)\leq E_{k}(Q)$,
where $c$ is the constant from Lemma \ref{CompLem} $(b)$.
\end{proposition}

\proof%
Items $(i)$ and $(ii)$ are clear from the definitions. Let us first show $%
(iii).$ We have $E_{k}(Q^{n})\leq ||Q^{n}||_{k}$ for every $n\in \mathbb{N},$
hence $\lim \sup E_{k}(Q^{n})^{\frac{1}{n}}\leq R_{k}(Q).$ On the other
hand, $R_{k}(Q)=R_{k}(gQg^{-1})\leq ||gQg^{-1}||_{k}$ for every $g\in GL_{d}(%
\overline{k}).$ Hence $R_{k}(Q)\leq E_{k}(Q)$ and for every $n\in \mathbb{N}$%
, $R_{k}(Q)^{n}=R_{k}(Q^{n})\leq E_{k}(Q^{n}),$ hence $R_{k}(Q)\leq \lim
\inf E_{k}(Q^{n})^{\frac{1}{n}}.$ So we have shown that $\lim E_{k}(Q^{n})^{%
\frac{1}{n}}$ exists and equals $R_{k}(Q).$ Furthermore, for every $n,p\in
\mathbb{N}$, $E_{k}(Q^{np})^{\frac{1}{np}}\leq E_{k}(Q^{p})^{\frac{1}{p}}$.
Letting $n$ tend to $+\infty $, we obtain $R_{k}(Q)\leq E_{k}(Q^{p})^{\frac{1%
}{p}}$. Hence $R_{k}(Q)=\inf_{n\in \mathbb{N}}E_{k}(Q^{n})^{\frac{1}{n}}$.

Now consider $(iv).$ It is clear that as $\Lambda _{k}(Q^{n})\leq
R_{k}(Q^{n})=R_{k}(Q)^{n},$ we have $\sup $ $\Lambda _{k}(Q^{n})^{\frac{1}{n}%
}\leq R_{k}(Q).$ On the other hand, given $n\in \mathbb{N}$, there is $0\leq
q\leq d^{2}$ from Lemma \ref{CompLem}, such that $\Lambda _{k}(Q^{qn})^{%
\frac{1}{qn}}\geq c^{\frac{1}{n}}\cdot E_{k}(Q^{n})^{\frac{1}{n}}$ (where $%
c=1$ if $k$ is non Archimedean) which forces $\sup \Lambda _{k}(Q^{n})^{%
\frac{1}{n}}\geq \lim \sup E_{k}(Q^{n})^{\frac{1}{n}}=R_{k}(Q).$

Now $(v).$ From $(iii)$ and $(iv)$ we clearly have for any $q\in \mathbb{N}$
$\Lambda _{k}(Q^{q})^{\frac{1}{q}}\leq R_{k}(Q)\leq E_{k}(Q).$ If $k$ is non
Archimedean, then this combined with Lemma \ref{CompLem} $(a)$ shows the
desired identity. If $k$ is Archimedean, then it gives $\Lambda
_{k}(Q^{q})\leq R_{k}(Q)^{q}$, which when combined with Lemma \ref{CompLem} $%
(b)$ gives $c\cdot E_{k}(Q)\leq R_{k}(Q).$
\endproof%

\begin{remark}
\label{RvsE}It can be shown that $R_{k}(Q)$ coincides with the infimum of $%
||Q||$ over all possible operator norms $||\cdot ||$ not necessarily assumed
to be operators norms of Euclidean or $\ell ^{\infty }$ norms (see \cite%
{BreEffec}). Observe however that when $k=\mathbb{R}$ or $\mathbb{C}$, then
we may have $R_{k}(Q)<E_{k}(Q).$ For instance, consider $Q=\{1,T,S\}\subset
SL_{2}(\mathbb{Z}),$ where $T$ and $S$ are the matrices corresponding to the
standard generators of $PGL_{2}(\mathbb{Z}),$ i.e. $T=\left(
\begin{array}{cc}
1 & 1 \\
0 & 1%
\end{array}%
\right) $ acts by translation by $1$ and $S=\left(
\begin{array}{cc}
0 & 1 \\
-1 & 0%
\end{array}%
\right) $ by inversion around the circle of radius $1$ in the upper-half
plane. Then it is easy to compute $E_{k}(Q)=\sqrt{2}=||tQt^{-1}||_{k}$ where
$t$ is the diagonal matrix $t=diag(\frac{1}{^{4}\sqrt{2}},^{4}\sqrt{2})$. On
the other hand, one can check that $||tQ^{2}t^{-1}||_{k}<2,$ and thus $%
R_{k}(Q)\leq E_{k}(Q^{2})^{\frac{1}{2}}<E_{k}(Q)$.
\end{remark}

Note that if $Q$ belongs to $SL_{d}(k),$ then $E_{k}(Q)\geq R_{k}(Q)\geq
\Lambda _{k}(Q)\geq 1$. The following Proposition explains what happens if
these quantities are close or equal to $1.$

\begin{prop}
\label{smallRE2}(\textbf{growth of displacement}) Suppose $k$ is Archimedean
(i.e. $k=\mathbb{R}$ or $\mathbb{C}$). Then for every $n\in \mathbb{N}$ and
every bounded subset $Q$ of $SL_{d}(k)$ containing $1,$ we have
\begin{equation}
E_{k}(Q^{n})\geq E_{k}(Q)^{\sqrt{\frac{n}{4d}}}  \label{gr}
\end{equation}%
And
\begin{equation}
\log R_{k}(Q)\geq c_{1}\cdot \log E_{k}(Q)\cdot \min \{1,\log E_{k}(Q)\}
\label{gr2}
\end{equation}%
where $c_{1}=c_{1}(d)>0$ is a positive constant.
\end{prop}

\proof%
We will use non-positive curvature of the symmetric space $\mathbb{X}_{k}%
\mathcal{\ }$associated to $SL_{d}(k).$ Let $d(\cdot ,\cdot )$ be the left
invariant Riemmanian metric on $\mathbb{X}_{k}$ normalized in such a way
that $d(ax_{0},x_{0})^{2}=\sum_{i}(\log |a_{i}|)^{2}$ if $x_{0}\in \mathbb{X}%
_{k}$ is the base point corresponding to $SO_{d}(\mathbb{R)}$ (resp. $SU_{d}(%
\mathbb{C})$) and $a$ is a diagonal matrix in $SL_{d}(k).$ We set $%
L_{k}(Q)=\inf_{x\in \mathbb{X}_{k}}\max_{q\in Q}d(q\cdot x,x).$ Observe that
$L_{k}\left( Q\right) \in \lbrack 1,\sqrt{d}]\log E_{k}(Q)$ (see also Lemma %
\ref{displacement} below).

Let $\ell _{n}:=L_{k}(Q^{n})$ and let $r_{n}$ be the infimum over $x\in
\mathbb{X}_{k}$ of the smallest radius of a closed ball containing $Q^{n}x.$
Note first that $r_{n}\leq \ell _{n}\leq 2r_{n}.$ Indeed if $\ell _{n}<t,$
then there is $x\in \mathbb{X}_{k}$ such that $d(qx,x)<t$ for all $q\in
Q^{n},$ i.e. $Q^{n}x$ lies in the ball of radius $t$ centered at $x,$ so $%
r_{n}\leq t$ and thus $r_{n}\leq \ell _{n}.$ Similarly if $r_{n}<t,$ then
there is $x\in \mathbb{X}_{k}$ such that $Q^{n}x$ is contained in a ball of
radius $t.$ In particular $d\left( y,z\right) \leq 2t$ for all $y,z\in
Q^{n}x $ and thus $d(qx,x)\leq 2t$ for all $q\in Q^{n},$ i.e. $\ell _{n}\leq
2t$, so $\ell _{n}\leq 2r_{n}.$

We now prove (\ref{gr}). Fix $\varepsilon >0$ and let $x,y\in \mathbb{X}_{k}$
be such that $Q^{n+1}x$ is contained in a ball of radius $%
r_{n+1}+\varepsilon $ around $y.$ Let $q\in Q$ be arbitrary. Since $Q$
contains $1,$ we have $Q^{n}x\subset Q^{n+1}x$, and $qQ^{n}x$ lies in the
two balls of radius $r_{n+1}+\varepsilon $ centered around $qy$ and around $%
y.$ By the CAT(0) inequality for the median, the intersection of the two
balls is contained in the ball $B$ of radius $t:=\sqrt{(r_{n+1}+\varepsilon
)^{2}-d(qy,y)^{2}/4}$ centered around the midpoint $m$ between $y$ and $qy.$
Translating by $q^{-1},$ we get that $Q^{n}x$ lies in the ball of radius $t$
centered at $q^{-1}m.$ In particular $r_{n}\leq t.$ This means $%
d(qy,y)^{2}\leq 4((r_{n+1}+\varepsilon )^{2}-r_{n}^{2}).$ Since $q\in Q$ and
$\varepsilon >0$ were arbitrary, we obtain $\ell _{1}^{2}\leq
4(r_{n+1}^{2}-r_{n}^{2})$. Summing over $n$, we get $n\ell _{1}^{2}\leq
4r_{n}^{2}\leq 4\ell _{n}^{2}$, hence (\ref{gr}).

For $(\ref{gr2}),$ note that for every $n,$ by Lemma \ref{CompLem} there is $%
q\leq d^{2}$ such that $\Lambda _{k}(Q^{qn})^{\frac{1}{q}}\geq
cE_{k}(Q^{n})\geq cE_{k}(Q)^{\sqrt{\frac{n}{4d}}}$ and hence $R_{k}(Q)\geq
c^{\frac{1}{n}}E_{k}(Q)^{\sqrt{\frac{1}{4dn}}}.$ Optimizing in $n$ we obtain
a constant $c_{1}=c_{1}(d)$ for which $(\ref{gr2})$ holds.
\endproof%

\begin{remark}
The above inequality (\ref{gr}) is interesting only when $E_{k}(Q)$ is
small. Indeed, a better estimate holds if $E_{k}(Q)>\frac{1}{c},$ where $c$
is the constant $c\in (0,1)$ obtained in Lemma \ref{CompLem} (b)
\begin{equation*}
E_{k}(Q^{n})\geq \max_{q\in \lbrack 1,d^{2}]}\Lambda _{k}(Q^{nq})^{\frac{1}{q%
}}\geq \max_{q\in \lbrack 1,d^{2}]}\Lambda _{k}(Q^{q})^{\frac{n}{q}}\geq
(cE_{k}(Q))^{n}.
\end{equation*}
\end{remark}

\begin{remark}
Observe that if $Q\subset SL_{d}(k)$, then adding the identity to $Q$ does
not modify our quantities. Namely if $Q_{1}=Q\cup \{Id\},$ then $%
E_{k}(Q_{1})=E_{k}(Q)$, $\Lambda _{k}(Q_{1})=\Lambda _{k}(Q)$ and also $%
R_{k}(Q_{1})=R_{k}(Q).$ For the last identity, note that for all $n\in
\mathbb{N}$, there is $m\leq n$ such that $\Lambda _{k}(Q_{1}^{n})=\Lambda
_{k}(Q^{m})\leq R_{k}(Q)^{m}\leq R_{k}(Q)^{n}$, since $R_{k}(Q)\geq 1$,
hence taking the supremum over $n$, $R_{k}(Q_{1})\leq R_{k}(Q),$ while the
converse inequality is clear.
\end{remark}

\subsection{Height, normalized height and minimal height\label{heightdef}}

Let $p$ be a prime number (abusing notation, we allow $p=\infty $). Fix an
algebraic closure $\overline{\mathbb{Q}_{p}}$ of the field of $p$-adic
numbers $\mathbb{Q}_{p}$ (if $p=\infty $, set $\mathbb{Q}_{p}=\mathbb{R}$).
We take the standard normalization of the absolute value on $\mathbb{Q}_{p}$
(i.e. $|p|_{p}=\frac{1}{p}$), while $|\cdot |_{\infty }$ is the standard
absolute value on $\mathbb{R}$. It admits a unique extension to $\overline{%
\mathbb{Q}_{p}}$, which we again denote by $|\cdot |_{p}$. Let $\overline{%
\mathbb{Q}}$ be the field of all algebraic numbers over $\mathbb{Q}$ and $K$
a number field. Let $V_{K}$ be the set of equivalence classes of valuations
on $K.$ For $v\in V_{K}$ let $K_{v}$ be the corresponding completion. For
each $v\in V_{K},$ $K_{v}$ is a finite extension of $\mathbb{Q}_{p}$ for
some prime $p$. We normalize the absolute value on $K_{v}$ to be the unique
one which extends the standard abolute value on $\mathbb{Q}_{p}$. Namely $%
|x|_{v}=|N_{K_{v}|\mathbb{Q}_{p}}(x)|_{p}^{\frac{1}{n_{v}}}$ where $%
n_{v}=[K_{v}:\mathbb{Q}_{p}]$. Equivalently $K_{v}$ has $n_{v}$ different
embeddings in $\overline{\mathbb{Q}_{p}}$ and each of them gives rise to the
same absolute value on $K_{v}.$ We identify $\overline{K_{v}}$, the
algebraic closure of $K_{v}$ with $\overline{\mathbb{Q}_{p}}$. Let $V_{f}$
be the set of finite places and $V_{\infty }$ the set of infinite places.

Let $d\in \mathbb{N}$ be an integer $d\geq 2$. For $v\in V_{K},$ in order
not to surcharge notation, we will use the subscript $v$ instead of $K_{v}$
in the quantities $E_{v}(F)=E_{K_{v}}(F)$, $\Lambda _{v}(F)=\Lambda
_{K_{v}}(F)$, etc.

Recall that if $x\in K$ then its height is by definition (see e.g. \cite{Bom}%
) the following quantity
\begin{equation*}
h(x)=\frac{1}{[K:\mathbb{Q}]}\sum_{v\in V_{K}}n_{v}\log ^{+}|x|_{v}
\end{equation*}%
It is well defined (i.e. independent of the choice of $K\ni x$). We will
make constant use of the following basic inequalities valid for every
algebraic numbers $x$ and $y$:\ $h(xy)\leq h(x)+h(y)$ and $h(x+y)\leq
h(x)+h(y)+\log 2.$

Let us similarly define the height of a matrix $f\in M_{d}(K)$ by
\begin{equation*}
h(f)=\frac{1}{[K:\mathbb{Q}]}\sum_{v\in V_{K}}n_{v}\log ^{+}||f||_{v},
\end{equation*}%
where $||f||_{v}$ is the operator norm of $f.$ We set the height of a finite
set $F$ of matrices in $M_{d}(K)$ to be
\begin{equation}
h(F)=\frac{1}{[K:\mathbb{Q}]}\sum_{v\in V_{K}}n_{v}\log ^{+}||F||_{v},
\label{heightdef2}
\end{equation}%
where $n_{v}=[K_{v}:\mathbb{Q}_{v}]$ and where $||F||_{v}=\max_{f\in
F}||f||_{v}.$ We also define the \textit{minimal height} of $F$ as:

\begin{equation}
e(F)=\frac{1}{[K:\mathbb{Q}]}\sum_{v\in V_{K}}n_{v}\log ^{+}E_{v}(F)
\label{abs}
\end{equation}%
and the \textit{normalized height} of $F$ as:
\begin{equation}
\widehat{h}(F)=\frac{1}{[K:\mathbb{Q}]}\sum_{v\in V_{K}}n_{v}\log
^{+}R_{v}(F)  \label{normhei}
\end{equation}%
For any height $\mathbf{h}$ (i.e. $h,e$ or $\widehat{h}$), we also set $%
\mathbf{h}=\mathbf{h}_{\infty }+\mathbf{h}_{f}$, where $\mathbf{h}_{\infty }$
is the infinite part of $\mathbf{h}$ (i.e. the part of the sum over the
infinite places of $K$) and $\mathbf{h}_{f}$ is the finite part of $\mathbf{h%
}$ (i.e. the part of the sum over the finite places of $K$). Note that these
heights are well defined independently of the number field $K$ such that $%
F\subseteq M_{d}(K).$ We also set $h_{v}(F)=\log ^{+}||F||_{v}$ (resp. $%
e_{v}(F)=\log ^{+}E_{v}(F),$ etc) so that $h=\frac{1}{[K:\mathbb{Q}]}%
\sum_{v\in V_{K}}n_{v}h_{v},$ etc.

\begin{remark}
\label{basechange}If we choose another basis of $\overline{\mathbb{Q}}^{d}$
the new height $h_{new}(F)$ differs only from the original height by a
bounded additive error. Indeed there are only finitely many places where the
new standard norm may differ from the original one. On the other hand $%
\widehat{h}(F)$ is independent of the choice of basis.
\end{remark}

The above terminology is justified by the following facts:

\begin{proposition}
\label{easyprop}For any finite set $F$ in $M_{d}(\overline{\mathbb{Q}})$, we
have:

(a) $\widehat{h}(F)=\lim_{n\rightarrow +\infty }\frac{1}{n}%
h(F^{n})=\inf_{n\in \mathbb{N}}\frac{1}{n}h(F^{n}),$

(b) $e_{f}(F)=\widehat{h}_{f}(F)$ and $e(F)+\log c\leq \widehat{h}(F)\leq
e(F)$ where $c$ is the constant in Lemma \ref{CompLem} $(b),$

(c) $\widehat{h}(F^{n})=n\cdot \widehat{h}(F)$ and $\widehat{h}(F\cup
\{Id\})=\widehat{h}(F),$

(d) $\widehat{h}(xFx^{-1})=\widehat{h}(F)$ if $x\in GL_{d}(\overline{\mathbb{%
Q}}).$
\end{proposition}

\proof%
Since $F$ is finite, there are only finitely many places $v$ such that $%
||F||_{v}>1.$ For each such place, $\frac{1}{n}\log
^{+}||F^{n}||_{v}\rightarrow \log ^{+}R_{v}(F),$ hence $\frac{1}{n}%
h(F^{n})\rightarrow \widehat{h}(F).$ By Prop. \ref{properties} $(vii)$ we
have $E_{v}(F)=R_{v}(F)$ if $v\in V_{f},$ hence $e_{f}(F)=\widehat{h}%
_{f}(F), $ while $c\cdot E_{v}(F)\leq R_{v}(F)\leq E_{v}(F)$ if $v\in
V_{\infty },$ hence $e_{\infty }(F)+\log c\leq \widehat{h}_{\infty }(F)\leq
e_{\infty }(F). $ Finally by Prop. \ref{properties} $(ii)$ $%
R_{v}(F^{n})=R_{v}(F)^{n}$ for every $n\in \mathbb{N}$ and every place $v$.
Hence $\widehat{h}(F^{n})=n\cdot \widehat{h}(F).$
\endproof%

We also record the following simple observation:

\begin{proposition}
\label{propbis}

(a) $e(xFx^{-1})=e(F)$ for all $F$ finite in $M_{d}(\overline{\mathbb{Q}})$
and $x\in GL_{d}(\overline{\mathbb{Q}}).$

(b) $e(F^{n})\leq n\cdot e(F),$

(c) If $\lambda $ is an eigenvalue of an element of $F,$ then $h(\lambda
)\leq \widehat{h}(F)\leq e(F),$

(d) If $F\subset GL_{d}(\overline{\mathbb{Q}})$ then $e(F\cup F^{-1})\leq
(d|F|+d-1)\cdot e(F)$ and $e(F\cup \{1\})=e(F).$ If $F$ is a subset of $%
SL_{d}(\overline{\mathbb{Q}}),$ then $e(F\cup F^{-1})\leq (d-1)\cdot e(F)$.
\end{proposition}

\proof%
The first three items are clear. For the last, observe that $||x^{-1}||_{v}=%
\frac{1}{|\det (x)|_{v}}||x||_{v}^{d-1}$ for any $x\in GL_{d}(K_{v})$ as can
be seen by expressing those norms in terms of the $KAK$ decomposition of $x$%
. Hence $||(F\cup F^{-1})||_{v}\leq ||F||_{v}^{d-1}\cdot \max \{\frac{1}{%
|\det (x)|_{v}},x\in F\cup \{1\}\}$ and $E_{v}(F\cup F^{-1})\leq
E_{v}(F)^{d-1}\cdot \max \{\frac{1}{|\det (x)|_{v}},x\in F\cup \{1\}\}.$ So $%
e(F\cup F^{-1})\leq (d-1)e(F)+\sum_{x\in F}h(\det (x)^{-1})$ i.e. $e(F\cup
F^{-1})\leq (d|F|+d-1)\cdot e(F).$%
\endproof%

We can also compare $e(F)$ and $\widehat{h}(F)$ when $\widehat{h}(F)$ is
small:

\begin{proposition}
\label{hecomp}For every $\varepsilon >0$ there is $\delta =\delta
(d,\varepsilon )>0$ such that if $F$ is a finite subset of $SL_{d}(\overline{%
\mathbb{Q}})$ containing $1$ with $\widehat{h}(F)<\delta $, then $%
e(F)<\varepsilon $. Moreover $\widehat{h}(F)=0$ iff $e(F)=0.$
\end{proposition}

This follows immediately from Proposition \ref{easyprop} b) and the
following proposition.

\begin{proposition}
\label{hecomp2}Let $c_{1}$ be the constant from Proposition \ref{smallRE2},
then
\begin{equation*}
\widehat{h}_{\infty }(F)\geq \frac{c_{1}}{4}\cdot e_{\infty }(F)\cdot \min
\{1,e_{\infty }(F)\}
\end{equation*}%
for any finite subset $F$ of $SL_{d}(\overline{\mathbb{Q}})$ containing $1.$
\end{proposition}

\proof%
From Proposition \ref{smallRE2}, $\widehat{h}_{v}(F)\geq c_{1}\cdot
e_{v}(F)\cdot \min \{1,e_{v}(F)\}$ for every $v\in V_{\infty }.$ We may
write $e_{\infty }(F)=\alpha e^{+}(F)+(1-\alpha )e^{-}(F)$ where $e^{+}$ is
the average of the $e_{v}$ greater than $1$ and $e^{-}$ the average of the $%
e_{v}$ smaller than $1$ (i.e. $e^{+}\sum_{v\in V_{\infty
},e_{v}>1}n_{v}=\sum_{v\in V_{\infty },e_{v}>1}n_{v}e_{v}$ and similarly for
$e^{-}$). Applying Cauchy-Schwarz, we have $\widehat{h}_{\infty }(F)\geq
c_{1}\cdot (\alpha e^{+}+(1-\alpha )(e^{-})^{2}).$ If $\alpha e^{+}(F)\geq
\frac{1}{2}e_{\infty }(F),$ then $\widehat{h}_{\infty }(F)\geq \frac{c_{1}}{2%
}e_{\infty }(F)$, and otherwise $(1-\alpha )e^{-}\geq \frac{e_{\infty }}{2}$%
, hence $\widehat{h}_{\infty }(F)\geq c_{1}(1-\alpha )(e^{-})^{2}\geq \frac{%
c_{1}}{4}e_{\infty }^{2}.$ At any case $\widehat{h}_{\infty }(F)\geq \frac{%
c_{1}}{4}\cdot e_{\infty }(F)\cdot \min \{1,e_{\infty }(F)\}.$%
\endproof%

In order to use the previous proposition inside $GL_{d},$ we shall need the
following:

\begin{proposition}
\label{Ad comp Cor}For every finite set $F$ in $GL_{d}(\overline{\mathbb{Q}}%
) $, then

$(i)$ $\widehat{h}(Ad(F))\leq d(|F|+1)\cdot \widehat{h}(F),$

$(ii)$ $e(Ad(F))\leq d(|F|+1)\cdot e(F)$ and

$(iii)$ $e(F)\leq e(Ad(F))+|F|\cdot \widehat{h}(F).$
\end{proposition}

\proof%
By Lemma \ref{Ad comp} below, $\log ||Ad(x)||_{v}\leq d\log
^{+}||x||_{v}+\log ^{+}|\det x^{-1}|_{v}$ for every place $v$ and $x\in
F^{n}.$ Thus $\log ||Ad(F^{n})||_{v}\leq d\log ^{+}||F^{n}||_{v}+n\max_{f\in
F}\log ^{+}|\det f^{-1}|_{v}.$ Letting $n$ go to infinity, we get $\log
R_{v}(Ad(F))\leq d\log ^{+}R_{v}(F)+\max_{f\in F}\log ^{+}|\det f^{-1}|_{v}.$
Summing over the places we obtain $\widehat{h}(Ad(F))\leq d\widehat{h}%
(F)+\sum_{f\in F}h(\det f^{-1})\leq d(1+|F|)\cdot \widehat{h}(F),$ where the
last inequality follows from Proposition \ref{propbis} (c). The other two
inequalities are proven in a similar way.
\endproof%

We used:

\begin{lemma}
\label{normcomp}\label{Ad comp}For every local field $k$ and every $x\in
GL_{d}(k)$, $\frac{1}{|\det (x)|_{k}^{1/d}}||x||_{k}\leq \left\Vert
Ad(x)\right\Vert _{k}\leq \frac{1}{|\det (x)|_{k}}\left\Vert x\right\Vert
_{k}^{d},$ where $Ad(x)\in GL(M_{d,d}(k)).$
\end{lemma}

\proof%
By the Cartan decomposition, we may assume that $x$ is diagonal $%
x=diag(a_{1},...,a_{d})$ with $|a_{1}|\geq ...\geq |a_{d}|.$ Then $%
\left\Vert x\right\Vert _{k}=|a_{1}|_{k}$ and $\left\Vert Ad(x)\right\Vert
_{k}=\frac{|a_{1}|_{k}}{|a_{d}|_{k}}.$ On the other hand $|\det
(x)|=|a_{1}\cdot ...\cdot a_{d}|$ hence $\frac{|a_{1}|_{k}}{|\det
(x)|_{k}^{1/d}}\leq |a_{1}|_{k}/|a_{d}|_{k}\leq \frac{|a_{1}|_{k}^{d}}{|\det
(x)|_{k}}.$ We are done.
\endproof%

\begin{corollary}
\label{Cor e h}Let $F$ be a finite subset of $GL_{d}(\overline{\mathbb{Q}}).$
Then $\widehat{h}(F)=0$ if and only if $e(F)=0.$
\end{corollary}

\proof%
By Proposition \ref{Ad comp Cor}, if $\widehat{h}(F)=0,$ then $\widehat{h}%
(Ad(F))=0.$ Since the elements of $Ad(F)$ have determinant $1,$ we may apply
Proposition \ref{hecomp} and obtain $e(Ad(F))=0.$ By the last inequality in
Proposition \ref{Ad comp Cor}, we get $e(F)=0.$ The converse is clear from
Proposition \ref{easyprop} (b).
\endproof%

\begin{remark}
\label{tala}In \cite{Tala} a variant of our height function $\widehat{h}$ is
studied in the case when $F$ is a single matrix. Namely setting $h_{0}(g):=%
\frac{1}{[K:\mathbb{Q}]}\sum_{v\in V_{K}}n_{v}\log ||g||_{v}$ for $g\in
M_{d}(\overline{\mathbb{Q}}),$ then it is shown in \cite{Tala} among other
things that if $g\in GL_{d}(\overline{\mathbb{Q}}),$ then $%
h_{0}(g)=\sup_{x\in \overline{\mathbb{Q}}^{d}\backslash
\{0\}}(h_{0}(gx)-h_{0}(x))$ and that $\lim_{n\rightarrow +\infty }\frac{1}{n}%
h_{0}(g^{n})=\frac{1}{[K:\mathbb{Q}]}\sum_{v\in V_{K}}n_{v}\log \Lambda
_{v}(g)$. The results of this section can be seen as a generalization of
\cite{Tala} to sets $F$ with more than one matrix.
\end{remark}

\section{Statement of the results\label{state}}

We state here our results. The main theorem is the following:

\begin{theorem}
\label{main}(Height gap) There exists a positive constant $\varepsilon
=\varepsilon (d)>0$ with the following property. Let $F$ be a finite subset
of $GL_{d}(\overline{\mathbb{Q}})$ generating a non virtually solvable
subgroup. Then $\widehat{h}(F)\geq \varepsilon .$
\end{theorem}

It is easy to characterize sets of zero normalized height :

\begin{proposition}
\label{zeroheight}(Height zero points) If $F$ is a finite subset of $GL_{d}(%
\overline{\mathbb{Q}})$, then $\widehat{h}(F)=0$ if and only if the group
generated by $F$ is virtually unipotent.
\end{proposition}

\proof%
If $\widehat{h}(F)=0,$ then $e(F)=0$ by Corollary \ref{Cor e h}. Now by
Proposition \ref{propbis}, $e(F\cup F^{-1})=0,$ hence $\widehat{h}((F\cup
F^{-1})^{n})=n\widehat{h}(F\cup F^{-1})=0$ for each $n\in \mathbb{N}$. Thus
every element from the group $\left\langle F\right\rangle $ generated by $F$
has only roots of unity as eigenvalues. However, according to Theorem 6.11
in \cite{Rag}, $\left\langle F\right\rangle $ has a finite index subgroup $%
\Gamma _{0}$ for which no element has a non-trivial root of unity as
eigenvalue. Therefore every element in $\Gamma _{0}$ must be unipotent, i.e.
$\Gamma _{0}$ is unipotent. Conversely, if $\left\langle F\right\rangle $ is
virtually unipotent, then every element in $\left\langle F\right\rangle $
has its eigenvalues among the roots of unity. In particular, as follows from
Proposition \ref{properties} $(iv)$, $R_{v}(F)=1$ for every place $v.$ Hence
$\widehat{h}(F)=0.$
\endproof%

The above results dealt with small values of the normalized height. The
following proposition says in substance that, provided $\left\langle
F\right\rangle $ has semisimple Zariski closure, the normalized height is
attained up to a constant by the height of some suitable conjugate of $F.$
We have

\begin{proposition}
\label{minheight}(Comparison between $h$ and $\widehat{h}$) If $\mathbb{G}$
is a semisimple algebraic group over $\overline{\mathbb{Q}}$ and $(\rho ,V)$
a finite dimensional linear representation of $\mathbb{G}$, then there is $%
C\geq 1$ and there is\ a choice of a basis on $V$ with associated height
function $h$ on $End(V),$ such that if $F$ is any finite subset of $\mathbb{G%
}(\overline{\mathbb{Q}})$ generating a Zariski-dense subgroup of $\mathbb{G}$%
, we have
\begin{equation*}
\widehat{h}(\rho (F))\leq e(\rho (F))\leq h(\rho (gFg^{-1}))\leq C\cdot
\widehat{h}(\rho (F))
\end{equation*}%
for some $g\in \mathbb{G}(\overline{\mathbb{Q}}).$
\end{proposition}

Recall from Remark \ref{basechange} that if we change the basis of $V$ the
associated height differs from the original one only by an additive
constant. This proposition subsumes Proposition \ref{minh} from the
introduction. It is important for the applications as it allows us to
conjugate $F$ back in the \textquotedblleft right position\textquotedblright
. Observe that by definition $e(F)$ is equal to the infimum of $h(gFg^{-1})$
when $g=(g_{v})_{v\in V_{K}}$ is allowed to vary among the full group of ad%
\`{e}les $GL_{d}(\mathbb{A}).$ This proposition shows that this infimum is
attained up to a multiplicative constant on principal ad\`{e}les, i.e. on $%
GL_{d}(\overline{\mathbb{Q}}).$ The condition that the Zariski closure of
the group generated by $F$ should be semisimple is important as easy
examples show that the result of the proposition can fail if for instance $F$
normalizes a unipotent subgroup.

The normalized height $\widehat{h}$ was defined for an arbitrary finite
subset of $GL_{d}(\overline{\mathbb{Q}})$. If $\mathbb{G}$ is an arbitrary
semisimple group, one can define the normalized height for $\mathbb{G}$ as
the one you obtain after taking some absolutely irreducible representation
of $\mathbb{G}$ which is non trivial on each factor of $\mathbb{G}$. The
following proposition shows that up to constants, this height is independent
of the choice of the representation.

\begin{proposition}
\label{welldefheight}(Invariance under change of representation) Let $%
\mathbb{G}$ be a semisimple algebraic group over $\overline{\mathbb{Q}}$ and
$(\rho _{i},V_{i})$ for $i=1,2$ be two finite dimensional linear
representations of $\mathbb{G}$ which are non trivial on each simple factor
of $\mathbb{G}$. Let $h_{i}$ be a height function on $End(V_{i})$ defined as
above by the choice of a basis in each $V_{i}.$ Then there are constants $%
C_{12},C_{12}^{\prime }\geq 1$ such that for any finite subset $F$ of $%
\mathbb{G}(\overline{\mathbb{Q}}),$ we have%
\begin{equation*}
\frac{1}{C_{12}}\cdot h_{2}(\rho _{2}(F))-C_{12}^{\prime }\leq h_{1}(\rho
_{1}(F))\leq C_{12}\cdot h_{2}(\rho _{2}(F))+C_{12}^{\prime }
\end{equation*}%
In particular%
\begin{equation*}
\frac{1}{C_{12}}\cdot \widehat{h_{2}}(\rho _{2}(F))\leq \widehat{h_{1}}(\rho
_{1}(F))\leq C_{12}\cdot \widehat{h_{2}}(\rho _{2}(F))
\end{equation*}%
Moreover the constant $C_{12}$ depends only on $\rho _{1}$ and $\rho _{2}$
and is independent of the choice of basis used to define $h_{1}$ and $h_{2}.$
\end{proposition}

Finally we record the following consequences:

\begin{corollary}
\label{eig} There are constants $\varepsilon =\varepsilon (d)$, $\kappa
=\kappa (d)\in \mathbb{N}$ and $C=C(d)\in \mathbb{N}$ such that if $F$ is
any finite subset of $GL_{d}(\overline{\mathbb{Q}})$ containing $1$, there
is some $a\in F^{\kappa }$ and some eigenvalue $\lambda $ of $a$ such that
\begin{equation*}
h(\lambda )\geq \frac{1}{|F|^{C}}\cdot \widehat{h}(F).
\end{equation*}
\end{corollary}

As a corollary of this and the height gap theorem we obtain an effective
version of Schur's classical result on torsion linear groups (see \cite%
{Schur}).

\begin{corollary}
\label{torsion}(Effective Schur: no large torsion balls) There is an integer
$N_{2}=N_{2}(d)\in \mathbb{N}$ such that if $K$ is a field and if $F$ is a
finite subset of $GL_{d}(K)$ containing $1$, then either it generates a
finite \ subgroup, or $(F\cup F^{-1})^{N_{2}(d)}$ contains an element of
infinite order. Furthermore if $F$ generates a non virtually nilpotent
subgroup, then we can find the element of infinite order already in $%
F^{N_{2}(d)}$.
\end{corollary}

The following example gives a situation showing that without the assumption
on $F$ in the last sentence of this corollary, the conclusion may fail.
Consider the subgroup of $GL_{2}(\mathbb{C})$ consisting of affine
transformations of the complex line. Then, for arbitrary $N\in \mathbb{N}$
one may find a finite (non-symmetric!) set $F$ containing the identity such
that the group generated by $F$ is infinite and virtually abelian, while $%
F^{N}$ consists solely of elements of finite order. For instance, take $%
F=\{id,a_{\omega },ta_{\omega }t^{-1}\}$ where $a_{\omega }=\left(
\begin{array}{cc}
\omega & 0 \\
0 & 1%
\end{array}%
\right) $ is multiplication by $\omega $ (a root of $1$ of order $N+1$) and $%
t=\left(
\begin{array}{cc}
1 & 1 \\
0 & 1%
\end{array}%
\right) $ is translation by $1,$ then the commutator $[a_{\omega
},ta_{\omega }t^{-1}]$ is $\neq 1$ if $N\geq 0$ and unipotent so of infinite
order, while $F^{N}$ is made of homotheties of ratio $\omega ^{k}$ with $%
1\leq k\leq N$ (i.e. elements of the form $\left(
\begin{array}{cc}
\omega ^{k} & \ast \\
0 & 1%
\end{array}%
\right) $), which are all torsion elements.

\begin{remark}
\label{poscar}In the entire paper we work over $\overline{\mathbb{Q}}.$
However a fair amount of what we do remains valid over global fields of
positive characteristic, i.e. over the algebraic closure of $\mathbb{F}%
_{p}(t).$ In particular the definition of the heights makes sense, except
that all places are non archimedean. Also all properties of Section \ref%
{Sec2} hold in positive characteristic as well, and they even become simpler
since all places are non archimedean and can thus be treated on an equal
footing, and $e(F)=\widehat{h}(F)$ always. Proposition \ref{welldefheight}
remains true for irreducible representations of $\mathbb{G}$. Moreover the
additive constant disappears. Also \ref{minheight} remains true for
irreducible representations. Same for Corollary \ref{eig}. This is key for
the applications to the Tits alternative in positive characteristic proved
in \cite{B}. The proof of these propositions is word by word the same as in
the $\overline{\mathbb{Q}}$ case, except for the proof of Proposition \ref%
{minheight} which needs some mild modification if the characteristic is $2$
or $3$ or if $\mathbb{G}$ is of type $A$ (see Remark \ref{poscharadjoint}).
Theorem \ref{main} however has no direct analog in positive characteristic
(nor does Zhang's theorem \ref{Zh}) : for a counter-example take $F_{n}$ to
be the two-element set in $SL_{2}$ consisting of an upper triangular and a
lower triangular unipotent matrix with coefficient $t^{\frac{1}{n}},$ then $%
F_{n}$ generates a Zariski-dense subgroup, but $\widehat{h}%
(F_{n})\rightarrow 0$. Nevertheless this is not a problem for the
applications to the Tits alternative, since all places being non archimedean
in positive characteristic, only the positivity of $\widehat{h}$ matters
there. See \cite{B} for more on positive characteristic.
\end{remark}

\begin{remark}
\label{loglog}Another possible definition of our height functions $h,%
\widehat{h}$ and $e$ consists in replacing the $\log ^{+}$ by $\log $ in $($%
\ref{heightdef2}), (\ref{abs}) and (\ref{normhei}). This new definition (let
us denote it by $h_{0}$ and $\widehat{h}_{0}$) is more adapted to $PGL_{d}$
while ours is more adapted to $SL_{d},$ but the differences are minor. First
of all, it is clear that the two notions coincide if $F\subset SL_{d},$
because each norm $||F||_{v}$ is then greater or equal to $1.$ Moreover, $%
h_{0}(F)\geq 0$ for all $F$ (from the product formula applied to any
eigenvalue of an element of $F,$ say). Also $h_{0}(\lambda F)=h_{0}(F)$ for
all $\lambda \in \overline{\mathbb{Q}}^{\times }$, and $h(F)=h_{0}(\rho (F))$
where $\rho $ is the obvious embedding of $GL_{d}$ inside $GL_{d+1}$ in the
upper left corner. Of course $\widehat{h}_{0}(F)\leq \widehat{h}(F).$

Moreover, Theorem \ref{main} also holds for $\widehat{h}_{0}.$ This follows
easily from Corollary \ref{eig1}. Indeed, let $F^{\prime }=\{f/(\det
f)^{1/d},f\in F\},$ then $\left\langle F^{\prime }\right\rangle $ is
virtually solvable if and only if $\left\langle F\right\rangle $ is. By
Corollary \ref{eig1} there is $g\in F^{\prime N_{1}(d)}$ and an eigenvalue $%
\lambda $ of $g$ such that $h(\lambda )>\varepsilon =\varepsilon (d)>0.$ But
$h(\lambda )\leq \widehat{h}_{0}(\{g\}),$ because $g\in SL_{d},$ and there
is $\mu \in \overline{\mathbb{Q}}^{\times }$ such that $\mu g\in
F^{N_{1}(d)}.$ So $h(\lambda )\leq \widehat{h}_{0}(\{g\})=\widehat{h}%
_{0}(\{\mu g\})\leq N_{1}(d)\widehat{h}_{0}(F).$ Hence the result.
\end{remark}

\section{Preliminary reductions\label{reduc}}

The main goal of this section is to establish Proposition \ref{adjoint}
below, which reduces the proof of Theorem \ref{main} to the case when $%
F=\{a,b\}$ is a finite set of two regular semisimple elements generating a
Zariski dense subgroup inside $\mathbb{G}(\overline{\mathbb{Q}}),$ where $%
\mathbb{G}$ is a Zariski-connected absolutely simple algebraic group of
adjoint type defined over $\overline{\mathbb{Q}}$, and where the underlying
vector space is the Lie algebra $\mathfrak{g}$ of $\mathbb{G}$ on which $%
\mathbb{G}$ acts via the adjoint representation, so that $\mathbb{G}\subset
SL(\mathfrak{g}).$

\subsection{Escape and reduction to a $2$-element set\label{reduc2}}

In this paragraph, we prove Proposition \ref{2elem} from the introduction in
the slightly stronger form given below in Proposition \ref{twoelements}. The
key ingredient there is a Lemma due to Eskin-Mozes-Oh about escaping from
algebraic subvarieties in bounded time.

First we recall some terminology. Let $\mathbb{G}$ be a connected semisimple
algebraic group over $\overline{\mathbb{Q}}.$ A semisimple group element $%
a\in \mathbb{G}(\overline{\mathbb{Q}})$ is said to be regular if $\ker
(Ad(a)-1)$ has the minimal possible dimension (namely equal to the absolute
rank of $\mathbb{G}$). For $A_{1}\in \mathbb{N},$ we will say that $a\in
\mathbb{G}(\overline{\mathbb{Q}})$ is $A_{1}$-regular if $\ker (Ad(a)-\omega
)$ has minimal possible dimension for every root of unity $\omega $ of order
at most $A_{1}$ (namely dimension $0$ if $\omega \neq 1$ and the absolute
rank if $\omega =1$). It is clear that the subset of $A_{1}$-regular
elements of $\mathbb{G}$ is a non-empty Zariski open subset of $\mathbb{G}$
consisting of semisimple elements.

If $Z$ is a proper Zariski closed subset of $\mathbb{G}$ invariant under
conjugation by a maximal torus $T,$ then we let $\widehat{Z}$ be the
Zariski-closure of $\{(gag^{-1},gbg^{-1})\in \mathbb{G}^{2}$ with $g\in
\mathbb{G}$, $a\in T$ and $b\in Z,$ or $a\in Z$ and $b\in T\}.$ It is a
proper algebraic subset of $\mathbb{G\times G}$ of dimension at most $2\dim
\mathbb{G}-1.$

\begin{proposition}
\label{twoelements}Let $\mathbb{G}$ be a connected semisimple algebraic
subgroup of $GL_{d}(\overline{\mathbb{Q}})$ with maximal torus $T$. Let $Z$
be a proper Zariski closed subset of $\mathbb{G}$ invariant under
conjugation by $T $. Then there is an integer $c=c(\mathbb{G},Z)>0$ such
that if $F$ is a finite subset of $\mathbb{G}(\overline{\mathbb{Q}})$
generating a Zariski-dense subgroup in $\mathbb{G}$, then $(F\cup
\{1\})^{c(d)}$ contains two elements $a$ and $b$ which are regular
semisimple, generate a Zariski dense subgroup of $\mathbb{G}$, and satisfy $%
(a,b)\notin \widehat{Z}.$ For any given integer $A_{1}\in \mathbb{N}$, by
allowing $c$ to depend also on $A_{1},$ i.e. $c=c(\mathbb{G},Z,A_{1})>0,$ we
may further assume that $a$ and $b$ are $A_{1}$-regular.
\end{proposition}

The key ingredient in this proposition is the following lemma. For an
algebraic variety $X$ we will denote by $m(X)$ the sum of the degree and the
dimension of each of its irreducible components.

\begin{lemma}
(Eskin-Mozes-Oh escape lemma \cite{EMO} Lemma 3.2)\label{Bezout} Given an
integer $m\geq 1$ there is $N=N(m)$ such that for any field $K$, any integer
$d\geq 1$, any $K$--algebraic subvariety $X$ in $GL_{d}(K)$ with $m(X)\leq m$
and any subset $F\subset GL_{d}(K)$ which contains the identity and
generates a subgroup which is not contained in $X(K)$, we have $%
F^{N}\nsubseteq X(K)$.
\end{lemma}

This result is a consequence of a generalized version of Bezout's theorem
about the intersection of finitely many algebraic subvarieties (see
Zannier's appendix in \cite{Sch}):

\begin{theorem}[Generalized Bezout Theorem]
\label{thm:Bezout} Let $K$ be a field, and let $Y_{1},\ldots ,Y_{p}$ be pure
dimensional algebraic subvarieties of $K^{n}$. Denote by $W_{1},\ldots
,W_{q} $ the irreducible components of $Y_{1}\cap \ldots \cap Y_{p}$. Then $%
\sum_{i=1}^{q}$deg$(W_{i})\leq \prod_{j=1}^{p}$deg$(Y_{j}).$
\end{theorem}

In order to apply the escape lemma to the proof of Proposition \ref%
{twoelements}, we need:

\begin{proposition}
\label{subvar}Let $\mathbb{G}$ be a connected semisimple algebraic group
over $\mathbb{C}$. There is a proper algebraic subvariety $X$ of $\mathbb{%
G\times G}$ such that any pair $(x,y)\notin X$ is made of regular semisimple
elements which generate a Zariski-dense subgroup of $\mathbb{G}$.
\end{proposition}

\proof%
Recall the well-known:

\begin{lemma}
\label{regular}The set $U$ of regular semisimple elements of $\mathbb{G}$ is
a non-empty Zariski-open subset of $\mathbb{G}$.
\end{lemma}

\proof%
The set $U$ coincides with the set of $g\in \mathbb{G}$ such that $\ker
(Ad(g)-1)$ is of minimal dimension. This is clearly a Zariski-open condition.%
\endproof%

We will make use of Jordan's theorem on finite subgroups of $GL_{d}(\mathbb{C%
})$ (see \cite{CurR}). Recall that according to this theorem, there is a
constant $C=C(d)\in \mathbb{N},$ such that if $\Gamma $ is a finite subgroup
of $GL_{d}(\mathbb{C}),$ then $\Gamma $ contains a abelian subgroup $A$ with
$[\Gamma :A]\leq C(d).$ As the kernel of the adjoint representation
coincides with the center of $\mathbb{G}$, it follows that the same bound
apply for all finite subgroups of $\mathbb{G}(\mathbb{C})$ as long as $\dim (%
\mathbb{G})\leq d.$ Let $V(\mathbb{G})$ be the proper Zariski-closed subset
of $\mathbb{G\times G}$ consisting of all couples $(x,y)$ such that $%
[x^{C!},y^{C!}]=1.$ By Jordan's theorem, if $(x,y)\notin V,$ then the
subgroup generated by $x$ and $y$ infinite.

Let $(\mathbb{G}_{i})_{1\leq i\leq k}$ be the $\mathbb{C}$-simple factors of
$\mathbb{G}$, together with their factor maps $\pi _{i}:\mathbb{G\rightarrow
G}_{i}.$ For convenience, let us denote $\mathbb{G}_{0}=\mathbb{G}$. Let $%
X_{i}$, for $0\leq i\leq k$, be the subset of $\mathbb{G}\times \mathbb{G}$
consisting of couples $(x,y)$ such that the $\mathbb{C}$-subalgebra of $End(%
\mathfrak{g}_{i})$ generated by $Ad(\pi _{i}(x))$ and $Ad(\pi _{i}(y))$ is
of strictly smaller dimension than the subalgebra generated by the full of $%
Ad(\mathbb{G}_{i}),$ where $\mathfrak{g}_{i}$ is the Lie algebra of $\mathbb{%
G}_{i}.$ This is a Zariski-closed subset of $\mathbb{G}\times \mathbb{G}.$
According to \cite{Bourb} VIII.2 ex.8, each $\mathfrak{g}_{i}$ is generated
by two elements. If follows that $X_{i}$ is a proper closed subvariety. Also
let $V_{i}$ be the set of couples $(x,y)\in \mathbb{G\times G}$ such that $%
(\pi _{i}(x),\pi _{i}(y))\in V(\mathbb{G}_{i}),$ where $V(\mathbb{G}_{i})$
is the proper closed subset defined above.

Finally, let $X$ be the proper closed subvariety $X=U^{c}\cup
\bigcup_{i}X_{i}\cup \bigcup_{i}V_{i}.$ Let us verify that $X$ satisfies the
conclusion of the proposition. Suppose $(x,y)\notin X.$ Then $(x,y)\in U$
and $x$,$y$ are regular semisimple. Let $\mathbb{H}$ be the Zariski closure
of the group generated by $x$ and $y.$ Let $\mathfrak{h}_{i}$ be the Lie
algebra of $\pi _{i}(\mathbb{H}),$ which is a Lie subalgebra of $\mathfrak{g}%
_{i}.$ As $\mathfrak{h}_{i}$ is invariant under $Ad(\pi _{i}(x))$ and $%
Ad(\pi _{i}(y)),$ it must be invariant $Ad(\mathbb{G}_{i})$, by the
assumption that $(x,y)\notin X_{i}.$ Therefore $\mathfrak{h}_{i}$ is an
ideal of $\mathfrak{g}_{i}.$ As $\mathfrak{g}_{i}$ is a simple Lie algebra,
either $\mathfrak{h}_{i}=\{0\}$ or $\mathfrak{h}_{i}=\mathfrak{g}_{i}.$ In
the former case, this means that $\pi _{i}(\mathbb{H})$ is finite. However,
by assumption $(\pi _{i}(x),\pi _{i}(y))\notin V(\mathbb{G}_{i}),$ this
means that the group generated by $\pi _{i}(x)$ and $\pi _{i}(y)$ is
infinite. So $\pi _{i}(\mathbb{H})$ is not finite, $\mathfrak{h}_{i}=%
\mathfrak{g}_{i}$ and $\pi _{i}(\mathbb{H})=\mathbb{G}_{i}.$

On the other hand, since $(x,y)\notin X_{0}$, the same argument shows that
the Lie algebra of $\mathbb{H}$ itself is an ideal in $\mathfrak{g.}$ Hence $%
\mathbb{H}^{\circ }$ is a normal subgroup of $\mathbb{G}$, hence is the
product of the simple factors of $\mathbb{G}$ contained in it. The fact that
$\pi _{i}(\mathbb{H})=\mathbb{G}_{i}$ for each $i$ forces $\mathbb{H=G}$.
\endproof%

\textit{Proof of Proposition \ref{twoelements}: }this is immediate by the
combination of Proposition \ref{subvar} and Lemma \ref{Bezout}.%
\endproof%

\subsection{Reduction to semisimple $\mathbb{G}\label{semireduc}$}

This paragraph is devoted to the proof of

\begin{proposition}
\label{semi}In order to prove Theorem \ref{main}, it is enough to prove the
following assertion. There is $\varepsilon =\varepsilon (d)>0$ such that :
if $\mathbb{G}\subseteq SL_{d}$ is a semisimple algebraic group over $%
\overline{\mathbb{Q}}$ acting irreductibly on $\overline{\mathbb{Q}}^{d}$,
and $F=\{Id,a,b\}$ is a subset of $\mathbb{G}$ generating a Zariski-dense
subgroup, then $e(F)>\varepsilon (d).$
\end{proposition}

The proof of this will rest mainly on the following proposition:

\begin{proposition}
\label{reducsemi}There are constants $C=C(d)>0$ and $m=m(d)\in \mathbb{N}$
such that if $F$ is a finite subset of $GL_{d}(\overline{\mathbb{Q}})$
containing $1$ and generating a non virtually solvable subgroup, there
exists a subset $F_{1}\subset F^{m}$, a connected semisimple algebraic group
$\mathbb{H}$ together with a faithful irreducible representation $(\rho
_{0},V_{0})$ of $\mathbb{H}$ with $\dim V_{0}\leq d$ and a homomorphism $\pi
:\Gamma _{0}\rightarrow \mathbb{H}(\overline{\mathbb{Q}})$, where $\Gamma
_{0}$ contains $F_{1}$ and has index at most $m$ in $\Gamma =\left\langle
F\right\rangle $, such that $\pi (\Gamma _{0})$ is Zariski dense in $\mathbb{%
H}$ and
\begin{equation*}
e(\rho _{0}\circ \pi (F_{1}))\leq C(d)\cdot e(F).
\end{equation*}
\end{proposition}

The proof of this proposition will occupy the rest of this subsection. At
the end we derive Proposition \ref{semi} from it.

We first analyse the local behavior at each place. Let $K$ be a number field
and $(e_{i})_{1\leq i\leq d}$ be the canonical basis of $V=K^{d}.$ Let $%
V=\bigoplus_{1\leq i\leq m}V_{i}$ be a direct sum decomposition adapted to
this basis, i.e. there are indices $j_{1}<...<j_{m}$ such that $%
V_{i}=span\{e_{j_{i}},...,e_{j_{i+1}-1}\}.$ Let $P$ be the group of block
upper triangular matrices determined by the corresponding flag, i.e. the
parabolic subgroup of $GL_{d}$ fixing the flag. Let $\rho :P\rightarrow
GL_{d}$ be the natural homomorphism that sends a matrix $A=(a_{ij})_{ij}\in
P $ to the matrix $\rho (A)=(a_{ij}^{\prime })_{ij}$ with $a_{ij}^{\prime
}=a_{ij}$ if $e_{i}$ and $e_{j}$ belong to the same $V_{k}$ and $%
a_{ij}^{\prime }=0$ otherwise.

\begin{lemma}
Let $v\in V_{K}$ be a place of $K.$ Let $F$ be a finite set in $%
GL_{d}(K)\cap P.$ Then
\begin{equation*}
E_{v}(\rho (F))=E_{v}(F)
\end{equation*}
\end{lemma}

\proof One needs first to observe that if $||\cdot ||$ is any standard norm
(i.e. a Euclidean norm associated to some basis of $k^{d}$ when $k$ is
archimedian, a sup-norm associated to some $\mathcal{O}_{k}$ lattice in $%
k^{d}$, say $R$, when $k$ is ultrametric) then $\left\Vert \rho
(x)\right\Vert _{v}\leq \left\Vert x\right\Vert _{v}$ for every $x\in P.$
This fact easily follows after we check that there is a direct sum
decomposition of $K_{v}^{d}$ as $\bigoplus_{1\leq i\leq m}W_{i}$ where the $%
W_{i}$'s are orthogonal (archimedean case) or give rise to a direct factor
decomposition $R=\bigoplus_{1\leq i\leq m}(W_{i}\cap R)$ (ultrametric case)
and for which $x$ remains block upper-triangular in any basis adapted to
this decomposition. From this we get the first half of the claimed relation,
i.e. $E_{v}(\rho (F))\leq \inf_{g\in GL_{d}(\overline{\mathbb{Q}}%
_{v})}\left\Vert g\rho (F)g^{-1}\right\Vert \leq \inf_{g\in GL_{d}(\overline{%
\mathbb{Q}}_{v})}\left\Vert gFg^{-1}\right\Vert _{v}=E_{v}(F).$

The second half follows from the remark that $\rho (F)$ can be approximated
uniformly by the $\delta F\delta ^{-1}$'s for some suitably chosen $\delta
\in \Delta (\overline{\mathbb{Q}}_{v}),$ where $\Delta $ is the group of
block scalar matrices associated with the $V_{i}$'s. Indeed we get
\begin{eqnarray*}
E_{v}(F) &=&\inf_{g\in GL_{d}(\overline{\mathbb{Q}}_{v})}\left\|
gFg^{-1}\right\| _{v}=\inf_{g\in GL_{d}(\overline{\mathbb{Q}}%
_{v})}\inf_{\delta \in \Delta (\overline{\mathbb{Q}}_{v})}\left\| g\delta
F\delta ^{-1}g^{-1}\right\| _{v} \\
&\leq &\inf_{g\in GL_{d}(\overline{\mathbb{Q}}_{v})}\left\| g\rho
(F)g^{-1}\right\| _{v}=E_{v}(\rho (F)).
\end{eqnarray*}

\endproof

This lemma gives that if $\overline{\mathbb{Q}}^{d}=\bigoplus_{1\leq i\leq
m}V_{i}$ is a direct sum decomposition associated to a composition series
for $\mathbb{G}$, then $e(\rho (F))=e(F)$. Moreover $\left\langle
F\right\rangle $ is virtually solvable if and only if $\rho (\left\langle
F\right\rangle )$ is virtually solvable and if and only if each $\rho
_{i}(\left\langle F\right\rangle )$ is virtually solvable, where $\rho _{i}$
is the induced action on $V_{i}.$ Hence there must be one $\rho _{i_{0}}$
for which $\rho _{i_{0}}(\left\langle F\right\rangle )$ is not virtually
solvable. Note that $e(\rho _{i_{0}}(F))\leq e(F).$

Let $\mathbb{H}_{0}$ be the Zariski closure of $\rho _{i_{0}}(F)$ in $%
GL(V_{i_{0}}).$ Note at this point that if we knew that $\mathbb{H}_{0}$ was
connected semisimple, we would be done.

Clearly, the connected component $\mathbb{H}_{0}^{\circ }$ is a reductive
group, since a non trivial unipotent radical would have a non trivial
pointwise fixed subspace :\ this subspace would then be globally invariant
under $\mathbb{H}_{0}$ and contradict the irreducibility of the action on $%
V_{i_{0}}.$

Let $W_{1}$ be a $\mathbb{H}_{0}^{\circ }$-irreducible subspace of minimal
dimension in $V_{i_{0}}.$ As $\mathbb{H}_{0}^{\circ }$ is normal in $\mathbb{%
H}_{0}$, and $\mathbb{H}_{0}$ acts irreducibly on $V_{i_{0}},$ we have a
direct sum decomposition $V_{i_{0}}=\bigoplus_{1\leq j\leq q}W_{j}$ into $%
\mathbb{H}_{0}^{\circ }$-irreducible subspaces where $\mathbb{H}_{0}/\mathbb{%
H}_{0}^{\circ }$ permutes transitively the $W_{j}$'s$.$ Since $\mathbb{H}%
_{0} $ is not virtually solvable, $\mathbb{H}_{0}^{\circ }$ is not solvable,
thus its image into $GL(W_{1})$ (say, all $W_{j}$ are isomorphic
representations of $\mathbb{H}_{0}^{\circ }$) is not solvable. Observe that,
since $\mathbb{H}_{0}^{\circ }$ is reductive and acts irreducibly on $W_{1}$%
, its center must act by homotheties (by Schur's lemma), hence the
semisimple part, say $\mathbb{S}$, of $\mathbb{H}_{0}^{\circ }$ also acts
irreducibly.

Let $\mathbb{H}_{1}$ be the stabilizer of $W_{1}$ in $\mathbb{H}_{0}$. Then $%
[\mathbb{H}_{0}:\mathbb{H}_{1}]\leq q\leq d.$ We now use:

\begin{lemma}
Suppose $\mathbb{L}$ is a linear algebraic group with $\mathbb{L}^{\circ }$
reductive. Let $\mathbb{S}$ be the semisimple part of $\mathbb{L}^{\circ }$ (%
$\mathbb{S}=[\mathbb{L}^{\circ },\mathbb{L}^{\circ }]$) and $\mathcal{Z}$ be
the centralizer of $\mathbb{S}$ in $\mathbb{L}$. Then $[\mathbb{L}:\mathcal{Z%
}\mathbb{S}]\leq c(d)$, where $c(d)$ is a constant depending only on $d=\dim
(\mathbb{L}).$
\end{lemma}

\proof%
The group $\mathbb{S}$ is normal in $\mathbb{L}$; let $\sigma :\mathbb{L}%
\rightarrow Aut(\mathbb{S})$ be the map given by conjugation. It induces $%
\overline{\sigma }:\mathbb{L}\rightarrow Out(\mathbb{S}).$ But $Out(\mathbb{S%
})$ is a finite group whose order depends only on the Dynkin diagram of $%
\mathbb{S}$, hence is bounded in terms of $d$ only (see \cite{Bor} 14.9).
Let $\mathbb{K}$ be the kernel of $\overline{\sigma }.$ Then $[\mathbb{L}:%
\mathbb{K}]\leq c(d)$ by the latter remark. On the other hand, by definition
of $\mathbb{K}$, $\mathbb{K}=\mathcal{Z}\mathbb{S}$.
\endproof%

We apply this lemma to $\mathbb{L=H}_{1}$. Since $\mathbb{S}$ acts
irreducibly on $W_{1},$ $\mathcal{Z}$ must act by homotheties (Schur's
lemma). Set $\mathbb{H}_{2}=\mathcal{Z}\mathbb{S}$. We have $\mathbb{H}%
_{0}^{\circ }\subset \mathbb{H}_{2}$ and $[\mathbb{H}_{0}:\mathbb{H}%
_{2}]\leq dc(d).$ Also $[\Gamma :\Gamma _{0}]\leq d$ where $\Gamma
_{0}=\Gamma \cap \mathbb{H}_{2}$ is Zariski dense in $\mathbb{H}_{2}.$ By
the (well-known) Lemma \ref{finiteindex} below, we may find a finite set $%
F_{0}$ in $(F\cup \{1\})^{2dc(d)-1}$ containing $1$ such that $\left\langle
F_{0}\right\rangle =\Gamma _{0}$. Moreover $e(F_{0})\leq e(F^{2dc(d)-1})\leq
(2dc(d)-1)e(F).$

\begin{lemma}
\label{finiteindex}Let $F$ be a finite subset of a group $\Gamma $
containing $1$. Assume that the elements of $F$ (together with their
inverses) generate $\Gamma .$ Let $\Gamma _{0}$ be a subgroup of index $k$
in $\Gamma .$ Then $F^{2k-1}$ contains a generating set of $\Gamma _{0}.$
\end{lemma}

\proof%
It is clear that $F^{k-1}$ contains a set of representatives for each left
coset in $\Gamma /\Gamma _{0},$ say $\{s_{1},...,s_{k}\}.$ Similarly, $%
(F^{-1})^{k-1}$ contains a set of representatives of the left cosets, say $%
\{u_{1},...,u_{k}\}.$ Consider all elements of $\Gamma _{0}$ of the form $%
s_{i}fu_{j}^{-1}$ for $i,j\in \lbrack 1,k]$ and $f\in F.$ They all belong to
$F^{2k-1}.$ It is straightforward to verify that, together with their
inverses, they generate $\Gamma _{0}$.%
\endproof%

In order to get rid of $\mathcal{Z}$, we now consider the action of $\mathbb{%
H}_{2}$ by conjugation on $End(W_{1}).$ The action factors through $\mathbb{S%
}$, hence the image is a connected semisimple algebraic subgroup of $%
GL(End(W_{1})),$ say $\mathbb{H}_{3}$. Moreover, we can bound the new height
in terms of the old one by making use of Proposition \ref{Ad comp Cor}
above. In particular if $F_{1}$ is any subset of $\mathbb{H}_{2}(\overline{%
\mathbb{Q}})$, then $e(Ad(F_{1}))\leq d(|F_{1}|+1)\cdot e(F_{1}).$

By Proposition \ref{twoelements} above (or Proposition \ref{2elem} from the
Introduction), we may find a pair $a,b$ in $F_{0}^{c_{2}(d)}$ (for some
constant $c_{2}(d)$) which generates modulo $\mathcal{Z}$ a Zariski dense
subgroup of $\mathbb{H}_{3}$. Let $F_{1}=\{1,a,b\}.$ Then $e(Ad(F_{1}))\leq
4d\cdot e(F_{1})\leq 4d\cdot c_{2}(d)\cdot e(F_{0})$ and $e(Ad(F_{1}))\leq
O_{d}(1)\cdot e(F)$ where $O_{d}(1)=8d^{2}c(d)c_{2}(d).$

Now the group $\left\langle Ad(F_{1})\right\rangle $ is Zariski dense in $%
\mathbb{H}_{3}$ and we may apply verbatim the beginning of the proof to this
group, to conclude that for some irreducible subrepresentation of $\mathbb{H}%
_{3}$ on $End(W_{1})$, say $(\rho ,\overline{W})$ we have $e(\rho
(Ad(F_{1})))\leq e(Ad(F_{1}))\leq O_{d}(1)\cdot e(F).$ Set $\mathbb{H}$ to
be the image of $\mathbb{H}_{3}$ in $GL(\overline{W}).$ Clearly $\Gamma _{0}$
acts on $\overline{W}$ with Zariski closure $\mathbb{H}$. Thus the proof of
Proposition \ref{reducsemi} is complete.

\begin{proof}[Proof of Proposition \protect\ref{semi}.]
In the setting of Theorem \ref{main} we first reduce to proving a gap for $%
e(F)$ instead of $\widehat{h}(F)$. This can indeed be achieved since, with
the notation of the last paragraph, $\widehat{h}(F)=\frac{1}{C_{d}}\widehat{h%
}(F^{C_{d}})\geq \frac{1}{C_{d}}\widehat{h}(F_{1})$ with $%
C_{d}=2dc(d)c_{2}(d).$ Moreover Proposition \ref{Ad comp Cor} also yields $%
\widehat{h}(Ad(F_{1}))\leq d(|F_{1}|+1)\cdot \widehat{h}(F_{1})\leq O_{d}(1)%
\widehat{h}(F).$ But $Ad(F_{1})$ lies in matrices with determinant $1,$ and
generates a non virtually solvable subgroup ; hence Proposition \ref{hecomp}
shows that $\widehat{h}(Ad(F_{1}))$ is bounded away from $0$ iff $%
e(Ad(F_{1}))$ is. But $e(\rho (Ad(F_{1})))\leq e(Ad(F_{1}))$ and $\rho
(Ad(F_{1}))$ generates a Zariski dense subgroup of the semisimple algebraic
group $\mathbb{H}$. Applying Proposition \ref{2elem} we are done.%
\endproof%
\bigskip
\end{proof}

\subsection{Comparison of heights under different representations\label%
{reducadj}\label{minidisp}}

In this paragraph we prove Proposition \ref{welldefheight} and we conclude
the reduction step of Theorem \ref{main} by proving Proposition \ref{adjoint}
below.

First let us recall some facts about representations of Chevalley groups.
Let $\mathbb{G}$ be a semisimple algebraic group over $\overline{\mathbb{Q}}$%
. The group $\mathbb{G}$ is a Chevalley group and comes with an associated $%
\mathbb{Z}$ structure. For general background on Chevalley groups we refer
the reader to Steinberg \cite{Stein} and to Bourbaki, Chapter 8 \cite{Bourb}%
. We let $\mathfrak{g}_{\mathbb{Z}}$ be a Chevalley order corresponding to $%
\mathbb{G}$ on the Lie algebra $\mathfrak{g}$ of $\mathbb{G}$ and $\mathfrak{%
a}$ the associated Cartan subalgebra in $\mathfrak{g.}$ Also let $%
(Y_{1},...,Y_{d})$ be a Chevalley basis of $\mathfrak{g}_{\mathbb{Z}}$ so
that the $Y_{i}$'s for $i\in \lbrack |\Phi ^{+}|+1,|\Phi ^{+}|+r]$ span the
admissible lattice $\mathfrak{g}_{\mathbb{Z}}\cap \mathfrak{a}$ of $%
\mathfrak{a}$ (here $\Phi ^{+}$ is the set of positive roots and $r$ the
absolute rank of $\mathbb{G}$). We denote by $T$ the maximal split torus of $%
\mathbb{G}$ corresponding to $\mathfrak{a}$ and by $\tau $ the Cartan
involution.

Given a local field $k,$ we define the \textquotedblleft Killing
norm\textquotedblright\ $||\cdot ||_{Kill,k}$ on $\mathfrak{g}_{k}$ to be
the one given by the Killing form $B_{\mathfrak{g}}$ when $k$ is archimedean
(i.e. $||X||_{Kill,k}=-B_{\mathfrak{g}}(X^{\tau },X)$) and the one arising
from the lattice $\mathfrak{g}_{\mathbb{Z}}\otimes \mathcal{O}_{k}=\mathfrak{%
g}_{\mathcal{O}_{k}}$ when $k$ is ultrametric (i.e. $||X||_{Kill,k}=%
\max_{i}|x_{i}|_{k}$ if $X=\sum x_{i}Y_{i}$). This allows us to define what
we will call the \textquotedblleft Killing height\textquotedblright\ $%
h_{Kill}(F)$ for $F\subseteq \mathbb{G}(\overline{\mathbb{Q}})$ by the usual
formula $(\ref{heightdef2})$ where we use the Killing norm at each place.

We denote by $K_{0}$ the stabilizer of $||\cdot ||_{Kill,k}.$ It is a
maximal compact subgroup of $\mathbb{G}(k).$ It is also a \textit{good}
maximal compact subgroup \ in the sense of \cite[3.3]{BT}, that is $K_{0}$
contains a copy of the Weyl group, so that $N_{K_{0}}(T(k))T(k)=N_{\mathbb{G}%
(k)}(T(k)).$

Let $V,\rho _{V}$ be a finite dimensional linear representation of $\mathbb{G%
}$ which is non trivial on each factor of $\mathbb{G}$. By Steinberg \cite%
{Stein} Section 2 Corollary 1, there exists an integer lattice, say $V_{%
\mathbb{Z}}$, of $V$ which is invariant under $\mathbb{G}(\mathbb{Z})$ and
which is spanned by a basis $(Y_{1},...,Y_{D})$ made of weight vectors for
the action of $T$. When $k$ is ultrametric $V_{\mathcal{O}_{k}}=V_{\mathbb{Z}%
}\otimes \mathcal{O}_{k}$ defines the following norm on $V_{k}=V_{\mathbb{Z}%
}\otimes k.$ We denote it by $||X||_{\rho _{V},k}:=\max_{i}|x_{i}|_{k}$ if $%
X=\sum_{i=1}^{D}x_{i}Y_{i}\in V_{k}.$ When $k$ is Archimedean, then there
exists a hermitian scalar product on $V_{k}$ which is invariant under $K_{0}$
and for which $\mathbb{G}(k)$ is stable under taking the adjoint (see \cite%
{Mos})$.$ We denote again by $||\cdot ||_{\rho _{V},k}$ the corresponding
hermitian norm. Together these norms define a height function $h_{\rho _{V}}$
on finite subsets of $End(V)$ defined as in $(\ref{heightdef2}).$ When $%
V,\rho _{V}$ is the adjoint representation, the just defined norms and
height coincide with the Killing norms and height.

\textit{Proof of Proposition \ref{welldefheight}.} By complete reducibility
(true in characteristic zero, in positive characteristic one has to assume
irreducibility to begin with), we may assume that both representations are
irreducible, with highest weight $\chi _{1}$ and $\chi _{2}$ respectively.
Let $W$ be the Weyl group of $\mathbb{G}$. If $g\in T,$ then $||\rho
_{i}(g)||_{\rho _{i},k}=\max_{w\in W}\left\vert \chi _{i}(w(g))\right\vert
_{k}.$ Since the root lattice is of finite index in the weight lattice,
there exists $n_{0}=n_{0}(\mathbb{G})\in \mathbb{N}$ such that $n_{0}\chi
_{i}$ is a linear combination $\sum_{\alpha \in \Pi }n_{\alpha }^{(i)}\alpha
$ with non-negative integer coefficients of the simple roots $\alpha \in \Pi
$ of $\mathbb{G}$. Since the inverse of the Cartan matrix of an irreducible
root system has no zero entry (see \cite{Bourb}), and since each $\rho _{i}$
is non trivial on each non trivial factor of $\mathbb{G}$, the coefficients $%
n_{\alpha }$ are non-zero. It follows that
\begin{eqnarray*}
||\rho _{1}(g)||_{\rho _{1},k}^{n_{0}} &\leq &\max_{\alpha \in \Pi
}\max_{w\in W}\left\vert \alpha (w(g))\right\vert _{k}^{M}\leq \max_{w\in
W}\left\vert \chi _{2}(w(g))\right\vert _{k}^{Mn_{0}} \\
&\leq &||\rho _{2}(g)||_{\rho _{2},k}^{Mn_{0}}
\end{eqnarray*}%
where $M=\max_{i=1,2,\alpha \in \Pi }n_{\alpha }^{(i)}$. Now the Cartan
decomposition implies that the above inequality holds for every $g\in
\mathbb{G}(k).$ It follows that $h_{\rho _{1}}\leq Mh_{\rho _{2}}.$ Finally,
if we considered instead the norm built from the basis $(Y_{1},...,Y_{D})$
of $V_{i}$ over $\mathbb{Z}$ defined above, then it would differ from $%
||\cdot ||_{\rho _{1},k}$ only at infinite places by a fixed multiplicative
constant, say $C_{i}.$ Let $h_{i}$ be the associated height. Then $|h_{\rho
_{i}}-h_{i}|\leq C_{i}$. Therefore $h_{1}\leq Mh_{2}+C_{1}+MC_{2}.$ Together
with Remark \ref{basechange} this ends the proof of Proposition \ref%
{welldefheight}.
\endproof%

We can now conclude this section of preliminary reductions by proving:

\begin{proposition}
\label{adjoint}In Theorem \ref{main}, we may assume that $F=\{Id,a,b\}$ is a
subset of $\mathbb{G}(\overline{\mathbb{Q}})$, where $\mathbb{G}$ is a
Zariski-connected absolutely simple algebraic group of adjoint type defined
over $\overline{\mathbb{Q}}$, viewed via the adjoint representation as an
algebraic subgroup of $SL(\mathfrak{g})$, where $\mathfrak{g}$ is the Lie
algebra of $\mathbb{G}$.
\end{proposition}

\textit{Proof:} According to Proposition \ref{semi}, when proving Theorem %
\ref{main}, we may assume that $F$ generates a Zariski-dense subgroup of a
semisimple algebraic group $\mathbb{G}$ acting irreducibly on $\overline{%
\mathbb{Q}}^{d}$. By Proposition \ref{welldefheight}, the normalized heights
of this representation of $\mathbb{G}$ and of the adjoint representation of $%
\mathbb{G}$ are comparable. Hence proving the gap for the first amounts to
proving the gap for the second. We may thus assume that $\mathbb{G}$ $=Ad(%
\mathbb{G})$ is acting via the adjoint representation on its Lie algebra $%
\mathfrak{g}.$ It remains to verify that we can reduce to a simple factor of
$\mathbb{G}$. Recall that $\mathbb{G}$ is the direct product of its simple
factors. As the representation space $\mathfrak{g}$ splits into the $\mathbb{%
G}$-invariant subspaces corresponding to the simple ideals $(\mathfrak{g}%
_{i})_{i}$ of $\mathfrak{g}$, and as $h(Ad(F))\geq h(Ad(F)_{|\mathfrak{g}%
_{i}})$ for each $i,$ it is enough to prove the theorem for one of the
simple factors. Finally by Proposition \ref{2elem}, we may assume that $F$
has three elements $\{Id,a,b\}$.
\endproof%

\subsection{Geometric interpretation and displacement on symmetric spaces
and Bruhat-Tits buildings\label{geominter}}

In this final paragraph of preliminary reductions, we give a geometric
interpretation of the minimal norm $E_{v}(F)$ and prove Lemma \ref{Mostow},
which will be key in the proof of the main theorem. We keep the notation of
the previous paragraph. Here again $\mathbb{G}$ is a Chevalley group and $k$
is a local field. We set $\mathcal{BT}(\mathbb{G},k)$ to be the Bruhat-Tits
building (resp. the symmetric space if $k$ is Archimedean) associated to $%
\mathbb{G}(k)$ as defined in \cite{BT}. We fix $V,\rho _{V}$ a finite
dimensional linear representation of $\mathbb{G}$ which is non trivial on
each factor of $\mathbb{G}$ as in \S \ref{reducadj} above. We let $x_{0}$ be
the base point of $\mathcal{BT}(SL_{V},k)$ corresponding to the stabilizer
of the norm $||\cdot ||_{\rho _{V},k}$ defined in \S \ref{reducadj}. The
maximal compact subgroup $K_{0}$ of $\mathbb{G}(k)$ defined in \S \ref%
{reducadj} coincides with the the stabilizer of $||\cdot ||_{\rho _{V},k}$
inside $\mathbb{G}(k).$

Let $\ell $ be a finite extension of $k.$ On $\mathcal{BT}(\mathbb{G},\ell )$
we define the distance $d$ to be the standard left invariant distance on $%
\mathcal{BT}(\mathbb{G},\ell )$ with the following normalization: if $a\in
A, $ then $d(a\cdot x_{0},x_{0})=\sqrt{\sum_{i=1}^{d}(\log |a_{i}|_{k})^{2}}%
, $ where $\log $ is the logarithm in base $|\pi _{\ell }^{-1}|_{k},$ with $%
\pi _{\ell }$ a uniformizer for $\mathcal{O}_{\ell }$ when $k$ is non
Archimedean, and the standard logarithm if $k$ is Archimedean. In this
normalization, the distance between adjacent vertices on $\mathcal{BT}(%
\mathbb{G},\ell )$ is of order $1$ and independent of $\ell $ (when $k$ is
non Archimedean).

Proposition \ref{caprace} below, which was communicated to us by P.E.
Caprace \cite{Cap}, shows that the symmetric space or building $\mathcal{BT}(%
\mathbb{G},k)\simeq \mathbb{G}(k)/K_{0}$ embeds isometrically in $\mathcal{BT%
}(SL_{V},k)$ as a closed and convex subspace via the orbit map $\mathbb{G}%
(k)/K_{0}\rightarrow \mathcal{BT}(SL_{V},k),$ $gK_{0}\mapsto g$. The short
proof given below makes use of the general theory of $CAT(0)$ spaces
(examples of which are the symmetric spaces and buildings $\mathcal{BT}%
(SL_{V},k)$ considered here). We refer the reader to the book by Bridson and
Haefliger \cite{BH} for background on $CAT(0)$ spaces. In particular, the
notion of a semisimple isometry of a $CAT(0)$ space is defined in \cite[II.6.%
]{BH}.

\begin{proposition}
\label{caprace}As above let $k$ be a local field and $\mathbb{G}$ a
semisimple $k$-split linear algebraic group, with Cartan decomposition $%
\mathbb{G}(k)=K_{0}T(k)K_{0}.$ Assume that $\mathbb{G}(k)$ acts properly by
isometries on a complete $CAT(0)$ space $X$ in such a way that semisimple
elements of $\mathbb{G}(k)$ act by semisimple isometries. Assume that $K_{0}$
fixes a point $p$ in $X$ which belongs to a flat $P$ stabilized by $T(k)$.
Then the map $gK_{0}\mapsto g\cdot p$ induces (up to renormalizing the
metric on $X$) a $\mathbb{G}(k)$-equivariant isometric embedding $f$ from $%
\mathcal{BT}(\mathbb{G},k)$ to $X$.
\end{proposition}

\proof%
Let $G=\mathbb{G}(k)$, $T=T(k)$ and $P_{0}$ the $T$-invariant flat in $%
\mathcal{BT}(\mathbb{G},k)$ containing the base point $p_{0}$ associated to $%
K_{0}.$ According to the Flat Torus Theorem (see \cite[II.7.]{BH}), there is
a unique minimal $T$-invariant flat containing $p$ and its dimension is $%
\dim T=r=rk(\mathbb{G})$. We may thus assume that $P$ is this minimal flat.
However, the normalizer $N_{G}(T)$ permutes the $T$-invariant flats and $%
N_{G}(T)$ is generated by $T$ and by $N_{G}(T)\cap K_{0}$. It follows that $%
N_{G}(T)$ stabilizes $P$. Hence $g\cdot p_{0}\mapsto g\cdot p$ induces an $%
N_{G}(T)$-equivariant map $f$ between $P_{0}$ and $P$.

Note first that it is enough to show that $f$ is a homothety from $P_{0}$ to
$P$. Indeed up to renormalizing the metric in $X,$ we may then assume that $%
f $ is an isometry from $P_{0}$ to $P,$ i.e. $d(a\cdot p,p)=d(a\cdot
p_{0},p_{0}).$ But then for any $g,h\in G$, $d(f(g\cdot p_{0}),f(h\cdot
p_{0}))=d(h^{-1}g\cdot p,p)=d(a\cdot p,p)=d(g\cdot p_{0},h\cdot p_{0})$ if $%
h^{-1}g=k_{1}ak_{2}$ is a Cartan decomposition of $h^{-1}g.$

The fact that $f:P_{0}\rightarrow P$ is a homothety follows from the
rigidity of Euclidean Coxeter group actions. Indeed $N_{G}(T)$ contains the
affine Weyl group as a co-compact subgroup which acts co-compactly by
isometries on both $P_{0}$ and $P.$ But any such action is isometric to the
standard Coxeter representation (cf. \cite{Bourb}).
\endproof%

\begin{remark}
This proposition is a special case of a theorem of Landvogt about
functoriality properties of Bruhat-Tits buildings (see \cite{Land}) in the
non-archimedean case and a theorem of Karpelevich and Mostow (see \cite{Mos}%
) in the form given by Eberlein in \cite[2.6.]{Eber} in the Archimedean case.
\end{remark}

The relation between the operator norm on $SL(V_{k})$ and the displacement
on $\mathcal{BT}(SL_{V},k)$ in given by the following well-known:

\begin{lemma}
\label{displacement}For any $f,g\in SL(V_{k})$ and $x=g^{-1}\cdot x_{0}\in
\mathcal{BT}(SL_{V},k),$ letting $\log $ be the logarithm in base $|\pi
_{k}^{-1}|_{k}$, we have
\begin{equation*}
\log \left\Vert gfg^{-1}\right\Vert _{\rho _{V},k}\leq d(f\cdot x,x)\leq
\sqrt{\dim V}\cdot \log \left\Vert gfg^{-1}\right\Vert _{\rho _{V},k}
\end{equation*}
\end{lemma}

\proof%
Since $d(\cdot ,\cdot )$ is left invariant, we may assume that $g=1$. Then
we may write $f=k_{1}ak_{2}$ the Cartan decomposition for $f$. Since the
norm is fixed by $K_{0}$ we can assume that $f=a$. Then the estimate is
obvious from the normalization we chose for $d(\cdot ,\cdot )$ above.%
\endproof%

A consequence of this lemma is that the logarithm of the minimal norm of a
finite set $F$ is comparable to the minimal displacement of $F$ on $\mathcal{%
BT}(SL_{V},k).$ As in \cite{uti}, 5.4.1., we will use a projection argument
and the fact that $\mathcal{BT}(SL_{V},k)$ is a $CAT(0)$ space in order to
show that the minimal displacement of $F$ is attained on $\mathcal{BT}(%
\mathbb{G},k).$ More precisely:

\begin{lemma}
\label{Mostow}For every finite set $F\in \mathbb{G(}k),$ we have
\begin{equation*}
E_{k}(\rho _{V}(F))\leq \inf_{g\in \mathbb{G}(\overline{k})}\left\Vert \rho
_{V}(gFg^{-1})\right\Vert _{\rho _{V},k}\leq E_{k}(\rho _{V}(F))^{\sqrt{\dim
V}}
\end{equation*}
\end{lemma}

\proof%
The left side of the inequalities is obvious from the definition of $%
E_{k}(\rho _{V}(F))$. For any $\varepsilon >0,$ one can find a finite
extension $\ell $ of $k$ such that $\inf_{g\in \mathbb{G}(\overline{\mathbb{Q%
}}_{v})}\left\Vert \rho _{V}(gFg^{-1})\right\Vert _{\rho _{V},k}\leq
\inf_{g\in \mathbb{G}(\ell )}\left\Vert \rho _{V}(gFg^{-1})\right\Vert
_{\rho _{V},k}+\varepsilon .$ By Lemma \ref{displacement}
\begin{equation}
\inf_{g\in \mathbb{G}(\ell )}\log \left\Vert \rho _{V}(gFg^{-1})\right\Vert
_{\rho _{V},k}\leq \inf_{g\in \mathbb{G}(\ell )}\max_{f\in
F}d(fgx_{0},gx_{0})\leq \inf_{x\in \mathcal{BT}(\mathbb{G},\ell )}\max_{f\in
F}d(fx,x)+c  \label{f1}
\end{equation}%
where the $\log $ is in base $|\pi _{\ell }^{-1}|_{k}$ and $c$ is the
maximal distance from any point in $\mathcal{BT}(\mathbb{G},\ell )$ to the
nearest point in the orbit $\mathbb{G}(\ell )\cdot x_{0}.$ Note that this
constant $c$ is independent of the choice of $\ell $. Since $\mathcal{BT}%
(SL_{V},\ell )$ is a $CAT(0)$ metric space and $\mathcal{BT}(\mathbb{G},\ell
)$ a closed convex subset, for every $x\in \mathcal{BT}(SL_{V},\ell ),$ one
can define the projection $p(x)$ of $x$ on $\mathcal{BT}(\mathbb{G},\ell )$
to be the (unique) point that realizes the distance from $x$ to $\mathcal{BT}%
(\mathbb{G},\ell ).$ The projection map is $1$-Lipschitz, hence $d(fx,x)\geq
d(fp(x),p(x))$ for any $x\in \mathcal{BT}(SL_{V},\ell ).$ Therefore
\begin{equation}
\inf_{x\in \mathcal{BT}(\mathbb{G},\ell )}\max_{f\in F}d(fx,x)=\inf_{x\in
\mathcal{BT}(SL_{V},\ell )}\max_{f\in F}d(fx,x)  \label{f2}
\end{equation}%
Combining (\ref{f1}) with (\ref{f2}) and Lemma \ref{displacement} we get
\begin{equation*}
\inf_{g\in \mathbb{G}(\overline{k})}\left\Vert \rho
_{V}(gFg^{-1})\right\Vert _{\rho _{V},k}\leq (|\pi _{\ell
}^{-1}|_{k})^{c}\inf_{g\in SL_{V}(\ell )}\left\Vert g\rho
_{V}(F)g^{-1}\right\Vert _{\rho _{V},k}^{\sqrt{\dim V}}+\varepsilon
\end{equation*}%
But $\ell $ can be taken arbitrarily large, so that $|\pi _{\ell }^{-1}|_{k}$
can be taken arbitrarily close to $1$, and since $c$ was independent of $%
\ell $ and $\varepsilon $ was arbitrary, we finally get the right hand side
of the desired inequality.%
\endproof%

\section{Local estimates on Chevalley groups\label{local}}

In this section, we work locally in a fixed local field, and prove several
crucial estimates relating the minimal norm $E_{k}(F)$ and the matrix
coefficients of the elements of $F$ in the adjoint representation. In the
next section, we will gather this local information at each place and put it
together to obtain global bounds.

\subsection{Notation\label{notat}}

Recall our notation. The group $\mathbb{G}$ is an absolutely simple
algebraic group of adjoint type defined over $\overline{\mathbb{Q}}$, viewed
via the adjoint representation as an algebraic subgroup of $GL(\mathfrak{g}%
), $ where $\mathfrak{g}$ is the Lie algebra of $\mathbb{G}$. We let $L$ be
a number field over which $\mathbb{G}$ splits. The set $F=\{Id,a,b\}$
consists of the identity and two semisimple regular elements of $\mathbb{G}(%
\overline{\mathbb{Q}})$ which generate a Zariski-dense subgroup of $\mathbb{G%
}$.

Let $T$ be the unique maximal torus of $\mathbb{G}$ containing $a.$ Let $%
\Phi =\Phi (\mathbb{G},T)$ be the set of roots of $\mathbb{G}$ with respect
to $T.$ Let $r$ be the absolute rank of $\mathbb{G}$. Let us also choose a
Borel subgroup $B $ of $\mathbb{G}$ containing $T,$ thus defining the set of
positive roots $\Phi ^{+}$ and a base $\Pi $ for $\Phi $. For $\alpha \in
\Phi $, let $\mathfrak{g}_{\alpha }$ be the root subspace corresponding to $%
\alpha $ and $\mathfrak{t}=\mathfrak{g}_{0}$ be the Lie algebra of $T,$ so
that we have the direct sum decomposition
\begin{equation}
\mathfrak{g}=\mathfrak{t}\oplus \bigoplus_{\alpha \in \Phi }\mathfrak{g}%
_{\alpha }  \label{rootdec}
\end{equation}

Let $(\alpha _{1},...,\alpha _{r})$ be an enumeration of the base associated
to the choice of $B$. The chosen enumeration of the elements of the base
induces a total order on the set of roots, namely two roots $\alpha =\sum
n_{i}\alpha _{i}$ and $\beta =\sum m_{i}\alpha _{i}$ satisfy $\alpha \geq
\beta $ iff $(n_{1},...,n_{r})\geq (m_{1},...,m_{r})$ for the canonical
lexicographical order on $r$-tuples. We may label the roots in decreasing
order, so that $\alpha _{1}>...>\alpha _{|\Phi ^{+}|}>0>\alpha _{|\Phi
^{+}|+r+1}>...>\alpha _{|\Phi |+r}$ is the full list of all roots. Note that
$d=\dim \mathfrak{g}=|\Phi |+r$ and that $\alpha _{|\Phi ^{+}|+r+i}=-\alpha
_{|\Phi ^{+}|+1-i}$ for $1\leq i\leq |\Phi ^{+}|.$ Also set $\alpha _{0}=0$
and $\alpha _{i}=0$ if $i\in I_{r}=[|\Phi ^{+}|+1,|\Phi ^{+}|+r].$ Finally,
for any root $\alpha $, let $i_{\alpha }$ be the index such that $\alpha
_{i_{\alpha }}=\alpha $.

For every $\alpha \in \Phi ^{+}\cup \{0\},$ let $\mathfrak{u}_{\alpha }$ be
the subspace of $\mathfrak{g}$ generated by the $\mathfrak{g}_{\beta }$'s
for all roots $\beta >\alpha .$

\begin{lemma}
For each $\alpha \in \Phi ^{+},$ $\mathfrak{u}_{\alpha }$ is an ideal in $%
\mathfrak{b}=\mathfrak{t}\oplus \bigoplus_{\alpha \in \Phi ^{+}}\mathfrak{g}%
_{\alpha }$. Moreover the sequence of $\mathfrak{u}_{\alpha }$'s for $\alpha
\in \Phi ^{+}$ is a decreasing (with $\alpha $) sequence of non-trivial
ideals in $\mathfrak{b}$ starting with $\mathfrak{u}_{0}=\bigoplus_{\alpha
\in \Phi ^{+}}\mathfrak{g}_{\alpha }$, each one being of codimension $1$
inside the previous one.
\end{lemma}

\proof%
We have $\mathfrak{u}_{\alpha }=\bigoplus_{\beta >\alpha }\mathfrak{g}%
_{\beta }.$ Moreover $[\mathfrak{g}_{\gamma },\mathfrak{g}_{\beta }]\leq
\mathfrak{g}_{\gamma +\beta }$ and $\gamma +\beta >\alpha $ for any $\gamma
\in \Phi ^{+}\cup \{0\},$ and so clearly $[\mathfrak{b},\mathfrak{u}_{\alpha
}]\leq \mathfrak{u}_{\alpha }.$ The second assertion follows from the fact
that each $\mathfrak{g}_{\alpha }$, $\alpha \in \Phi $, has dimension $1$.
\endproof%

We also denote by $U_{\alpha }$ the unipotent algebraic subgroup of $\mathbb{%
G}$ whose Lie algebra is $\mathfrak{u}_{\alpha },$ and by $U_{0}$ the
maximal unipotent subgroup, whose Lie algebra is $\mathfrak{u}_{0}$.
Furthermore, for each $\alpha \in \Phi ,$ we denote by $e_{\alpha }:\mathbb{G%
}_{a}\rightarrow \mathbb{G}$ the morphism of algebraic groups corresponding
to $X_{\alpha }\in \mathfrak{g}_{\alpha },$ i.e. $e_{\alpha }(t)=\exp
(tX_{\alpha }).$ Recall that $U_{\alpha }=\prod_{\beta >\alpha }e_{\beta }(%
\mathbb{G}_{a}),$ so any element in $U_{\alpha }$ can be written as a
product of $e_{\beta }(t_{\beta })$'s for $\beta >\alpha .$

Recall that since $\mathfrak{g}$ is a simple Lie algebra, it has a Chevalley
basis (canonical up to automorphisms of $\mathfrak{g}$) $\{H_{\alpha
},\alpha \in \Pi \}\cup \{X_{\alpha },\alpha \in \Phi \}$ with $H_{\alpha
}\in \mathfrak{t}$ and $X_{\alpha }\in \mathfrak{g}_{\alpha }$. Let $(\omega
_{\alpha })_{\alpha \in \Pi }$ be the basis of $\mathfrak{t}$ which is dual
to $\Pi .$ Equivalently $\beta (\omega _{\alpha })=\delta _{\alpha \beta }.$
Then $\{\omega _{\alpha },\alpha \in \Pi \}\cup \{X_{\alpha },\alpha \in
\Phi \}$ is also a basis of $\mathfrak{g}$ and defines a $\mathbb{Z}$%
-structure $\mathfrak{g}_{\mathbb{Z}}$ on $\mathfrak{g}$ with $[\mathfrak{g}%
_{\mathbb{Z}},\mathfrak{g}_{\mathbb{Z}}]\subset \mathfrak{g}_{\mathbb{Z}}$
(see \cite{Stein}). Hence for any field $k,$ we can define $\mathfrak{g}_{k}=%
\mathfrak{g}_{\mathbb{Z}}\otimes _{\mathbb{Z}}k.$ If $K$ is a number field
and $v$ a place of $K$ with corresponding embedding $\sigma
_{v}:K\rightarrow K_{v}$ where $K_{v}$ is the associated completion of $K,$
then we will use the notation $\mathfrak{g}_{v}$ to mean $\mathfrak{g}%
_{K_{v}}.$

Since the definition of $e(F)$ does not depend on the choice of the basis of
$\mathfrak{g}$ used to define the standard norm appearing in the quantities $%
E_{v}(F),$ we may as well fix the basis of $\mathfrak{g}$ to be the basis $%
\{\omega _{\alpha },\alpha \in \Pi \}\cup \{X_{\alpha },\alpha \in \Phi \}$,
which we denote $(Y_{1},...,Y_{d})$ with $Y_{i}=X_{\alpha _{i}}\in \mathfrak{%
g}_{\alpha _{i}}$ if $i\notin I_{r}=[|\Phi ^{+}|+1,|\Phi ^{+}|+r]$ and $%
Y_{i}\in \{\omega _{\alpha },\alpha \in \Pi \}$ if $i\in I_{r}.$

Let $B(X,Y)$ be the Killing form on $\mathfrak{g.}$ We have $%
B(Y_{i},Y_{j})\in \mathbb{Z}$ for all $i,j.$ The Chevalley involution is the
linear map $\tau :\mathfrak{g}\rightarrow \mathfrak{g}$ by $Y_{i}^{\tau
}=-Y_{i}$ for $i\in I_{r}$ and and $X_{\alpha }^{\tau }=-X_{-\alpha }$ for
each $\alpha \in \Phi $. Then $\tau $ is an automorphism of $\mathfrak{g}$
which perserves $\mathfrak{g}_{\mathbb{Z}}.$ We set $\phi (X,Y)=-B(X^{\tau
},Y).$

We now describe how to choose the norm $\left\| \cdot \right\| _{v}$ on $%
\mathfrak{g}_{v}.$ First consider the case when $v$ is Archimedean, i.e. $%
\overline{\mathbb{Q}_{v}}=\mathbb{C}$. We set $\left\langle X,Y\right\rangle
_{v}=\phi (X,\overline{Y}),$ and thus get a positive definite scalar product
on $\mathfrak{g}_{v}$ and a norm $\left\| \cdot \right\| _{v}$ on $\mathfrak{%
g}_{v}$. Let $\mathbf{K}_{v}=\{g\in \mathbb{G}(\mathbb{C}),g^{\tau }=g\}$,
where we denoted again by $\tau $ the automorphism of $\mathbb{G}(\mathbb{C}%
) $ induced by the Chevalley involution $\tau .$ Then $\mathbf{K}_{v}$ is a
maximal compact subgroup of $\mathbb{G}(\mathbb{C})$ and this group
coincides with the stabilizer of $\left\langle \cdot ,\cdot \right\rangle
_{v}$ in $\mathbb{G}(\mathbb{C}),$ which in turn coincides with $\{g\in
\mathbb{G}(\mathbb{C}),\left\| Ad(g)\right\| _{v}=1\}$ where the norm is the
operator norm associated to $\left\langle \cdot ,\cdot \right\rangle _{v}$.
Note that $(Y_{1},...,Y_{d})$ however is not orthogonal with respect to $%
\left\langle \cdot ,\cdot \right\rangle _{v}$ but the decomposition (\ref%
{rootdec}) is orthogonal$.$ Finally observe that according to the Iwasawa
decomposition we may write $\mathbb{G}(\mathbb{C})=\mathbf{K}_{v}U_{0}(%
\mathbb{C})T(\mathbb{C}).$

Suppose now that $v$ is non Archimedean. We let $\left\Vert \cdot
\right\Vert _{v}$ be the norm induced on $\mathfrak{g}_{v}$ by the basis $%
(Y_{1},...,Y_{d}),$ i.e. $\left\Vert \sum y_{i}Y_{i}\right\Vert
_{v}=\max_{1\leq i\leq d}|y_{i}|_{v}.$ Then we set $\mathbf{K}_{v}$ to be
the stabilizer in $\mathbb{G}(\overline{\mathbb{Q}}_{v})$ of $\mathfrak{g}_{%
\mathcal{O}_{v}}=\mathfrak{g}_{\mathbb{Z}}\otimes _{\mathbb{Z}}\mathcal{O}%
_{v},$ where $\mathcal{O}_{v}$ is the ring of integers in $\overline{\mathbb{%
Q}}_{v}.$ In this situation, the Iwasawa decomposition (see \cite{IwaMat})\
reads $\mathbb{G}(\overline{\mathbb{Q}}_{v})=\mathbf{K}_{v}U_{0}(\overline{%
\mathbb{Q}}_{v})T(\overline{\mathbb{Q}}_{v}).$ Recall (see \cite[\S 1, Lemma
6]{Stein}) that for any $n\in \mathbb{N}$ and any $\alpha \in \Phi ,$ $\frac{%
ad(X_{\alpha })^{n}}{n!}$ fixes $\mathfrak{g}_{\mathbb{Z}}$. Hence $%
\left\Vert \frac{ad(X_{\alpha })^{n}}{n!}\right\Vert _{v}\leq 1.$

Let $c_{v}=\sup_{\alpha \in \Phi }\frac{\left\| ad(X_{\alpha })\right\| _{v}%
}{\left\| X_{\alpha }\right\| _{v}}$ if $v$ is Archimedean and set $c_{v}=0$
if $v$ is non Archimedean. Then, for any place $v$ and $x\in \overline{%
\mathbb{Q}}_{v},$ the following holds
\begin{eqnarray}
\left\| Ad(e_{\alpha }(x))\right\| _{v} &=&\left\| 1+ad(xX_{\alpha })+\frac{%
ad(xX_{\alpha })^{2}}{2!}+...+\frac{ad(xX_{\alpha })^{d}}{d!}\right\| _{v}
\label{expbound} \\
&\leq &e^{c_{v}}\cdot \max \{1,\left\| xX_{\alpha }\right\| _{v}\}^{d}
\end{eqnarray}
for every $\alpha \in \Phi $, where $d=\dim \mathfrak{g}$.

Finally we observe that we have:

\begin{lemma}
\label{norm}Suppose $v$ is non Archimedean. Then, for each root $\alpha \in
\Phi ,$ the norm $|\alpha |_{v}:=\sup_{Y\in \mathfrak{t}_{v}\backslash \{0\}}%
\frac{|\alpha (Y)|_{v}}{\left\Vert Y\right\Vert _{v}}$ satisfies $|\alpha
|_{v}=1$.
\end{lemma}

\proof%
First, note that it obviously holds when $\alpha \in \Pi ,$ because $\alpha
(\omega _{\beta })=\delta _{\alpha \beta }.$ As every $\alpha \in \Phi $ is
a linear combination with integer coefficients of elements from $\Pi ,$ we
must have $|\alpha |_{v}\leq 1$. To show the opposite inequality, observe
that $\gcd (\alpha (\omega _{\beta }),\beta \in \Pi )=1$. Indeed, suppose
there were a prime number $p$ such that $p$ divides $\gcd (\alpha (\omega
_{\beta }),\beta \in \Pi ).$ Then $\alpha =p\alpha _{0}$ with $\alpha
_{0}=\sum_{i=1}^{r}n_{i}\alpha _{i}$ for some $n_{i}\in \mathbb{Z}$ and $\Pi
=\{\alpha _{1},...,\alpha _{r}\}.$ But since $\Phi $ is reduced, $\alpha $
belongs to some base of the root system say $\alpha =\alpha _{1}^{\prime
},...,\alpha _{r}^{\prime }$ (\cite{Bourb} VI.1.5)$.$ Since each $\alpha
_{i} $ is a linear combination with integer coefficients of some $\alpha
_{i}^{\prime }$ 's, we get that $\alpha _{0}\in \mathbb{Z}\alpha ,$ a
contradiction.%
\endproof%

Note that when $v$ is Archimedean, then $|\alpha |_{v}$ is independent of $v$
(it is the norm of $\alpha $ with respect to the canonical scalar product
induced on the real vector space spanned by the root system). We denote it
by $|\alpha |_{\infty }.$

\subsection{Some local estimates}

We work locally, fixing the place $v$. The aim of this subsection is to
record two estimates, namely Propositions (\ref{upper1}) and (\ref{upper2})
below.

Let now $(e_{i})_{1\leq i\leq d}$ be an orthonormal basis for $\mathfrak{g}_{%
\mathbb{C}}$ such that for each $1\leq i\leq d$, $e_{i}\in \mathfrak{g}%
_{\alpha _{i}}.$ Note that if $b\in Ad(B(\mathbb{C})\mathbb{)}$, then the
matrix of $b$ is upper-triangular in the basis $(e_{i})_{i}.$

\begin{lemma}
\label{triangular}Let $V$ be a complex vector space of dimension $n$ endowed
with a hermitian scalar product $\left\langle \cdot ,\cdot \right\rangle .$
Let $(e_{i})_{1\leq i\leq n}$ be an orthonormal basis of $V$ and assume that
$b\in SL(V)$ has an upper triangular matrix in this basis. Then
\begin{equation*}
\sum_{i<j}\left| \left\langle be_{i},e_{j}\right\rangle \right| ^{2}\leq
n\cdot \left( \left\| b\right\| ^{2}-1\right)
\end{equation*}
\end{lemma}

\proof%
Let $\lambda _{1}\geq ...\geq \lambda _{n}\geq 0$ be the eigenvalues of $%
b^{*}b.$ According to Cartan's $KAK$ decomposition, we have $\left\|
b\right\| ^{2}=\lambda _{1}.$ We have
\begin{equation*}
tr(b^{*}b)=\sum \lambda _{i}\leq n\cdot \lambda _{1}=n\cdot \left\|
b\right\| ^{2}
\end{equation*}
On the other hand,
\begin{equation*}
tr(b^{*}b)=\sum_{i,j}\left| \left\langle be_{i},e_{j}\right\rangle \right|
^{2}=\sum_{i<j}\left| \left\langle be_{i},e_{j}\right\rangle \right|
^{2}+\sum_{1\leq i\leq n}|\mu _{i}|^{2}
\end{equation*}
where $\mu _{1},...,\mu _{n}$ are the eigenvalues of $b.$ But $\frac{1}{n}%
\sum_{1\leq i\leq n}|\mu _{i}|^{2}\geq (\prod |\mu _{i}|)^{\frac{2}{n}}=1$
since $\det (b)=1.$ Hence
\begin{equation*}
n\cdot \left\| b\right\| ^{2}\geq tr(b^{*}b)\geq \sum_{i<j}\left|
\left\langle be_{i},e_{j}\right\rangle \right| ^{2}+n
\end{equation*}
\endproof%

\begin{lemma}
\label{onestep}Let $v$ be any place. Let $\alpha \in \Phi ^{+}$, $a\in T(%
\overline{\mathbb{Q}}_{v})$ regular$,$ $v_{\alpha }\in U_{\alpha }(\overline{%
\mathbb{Q}}_{v})$ and $n_{\alpha }=e_{\alpha }(x)$ for some $x\in \overline{%
\mathbb{Q}}_{v}$ and let $b=Ad(n_{\alpha }av_{\alpha }n_{\alpha }^{-1})$.
Then if $v$ is Archimedean
\begin{equation}
\left\Vert xX_{\alpha }\right\Vert _{v}\leq \frac{\sqrt{d\cdot (\left\Vert
b\right\Vert _{v}^{2}-1)}}{|1-\alpha (a)|_{v}|\alpha |_{v}}  \label{talpha}
\end{equation}%
while if $v$ is non Archimedean
\begin{equation}
\left\Vert xX_{\alpha }\right\Vert _{v}\leq \frac{\left\Vert b\right\Vert
_{v}}{|1-\alpha (a)|_{v}|\alpha |_{v}}  \label{talpha2}
\end{equation}%
where $|\alpha |_{v}$ is the norm of $\alpha $ viewed as linear form on $%
\mathfrak{t}_{v}$ as in Lemma \ref{norm}.
\end{lemma}

\proof%
First observe that if $m\in U_{\alpha }$ and $Y\in \mathfrak{t}_{v}$, then $%
Ad(m)Y\in Y+\mathfrak{u}_{\alpha },$ while if $m=e_{\alpha }(x)$ for some $%
x, $ then $Ad(m)Y=Y+x[X_{\alpha },Y]=Y-\alpha (Y)xX_{\alpha }.$ Let $Y\in
\mathfrak{t}_{v}$ be arbitrary. We have $n_{\alpha }av_{\alpha }n_{\alpha
}^{-1}=aa^{-1}n_{\alpha }an_{\alpha }^{-1}n_{\alpha }v_{\alpha }n_{\alpha
}^{-1}=a\cdot e_{\alpha }((\alpha (a^{-1})-1)x)\cdot n^{\prime \prime }$
where $n^{\prime \prime }\in U_{\alpha }.$ We then compute:
\begin{equation}
bY\in Y+x(1-\alpha (a))\alpha (Y)X_{\alpha }+\mathfrak{u}_{\alpha }
\label{sim}
\end{equation}%
Suppose first that $v$ is Archimedean
\begin{equation*}
\left\langle bY,X_{\alpha }\right\rangle _{v}=x(1-\alpha(a))\alpha
(Y)\left\Vert X_{\alpha }\right\Vert _{v}^{2}
\end{equation*}%
On the other hand $Y=\sum y_{i}e_{i}$ for some $y_{i}\in \overline{\mathbb{Q}%
}_{v}$ all zero except if $i\in I_{r}=[|\Phi ^{+}|+1,|\Phi ^{+}|+r]$ (recall
that we defined the vectors $e_{i}$'s in Lemma \ref{triangular} to be any
orthonormal basis for $\mathfrak{g}_{\mathbb{C}}$ such that for each $1\leq
i\leq d$, $e_{i}\in \mathfrak{g}_{\alpha _{i}}$)$.$ Using Cauchy-Schwarz, we
get:
\begin{equation*}
\left\vert \left\langle bY,X_{\alpha }\right\rangle _{v}\right\vert \leq
\left\Vert X_{\alpha }\right\Vert _{v}\left\Vert Y\right\Vert _{v}\sqrt{%
\sum_{i\in I_{r}}\left\vert \left\langle be_{i},e_{i_{\alpha }}\right\rangle
_{v}\right\vert ^{2}}
\end{equation*}%
But $b$ is upper-triangular in the basis $(e_{i})_{i}$ because $n_{\alpha
}av_{\alpha }n_{\alpha }^{-1}$ belongs to the Borel subgroup $B(\overline{%
\mathbb{Q}}_{v}).$ We are in a position to apply Lemma \ref{triangular},
which yields:
\begin{equation*}
\left\vert (1-\alpha (a))\alpha (Y)\right\vert _{v}\cdot \left\Vert
xX_{\alpha }\right\Vert _{v}\cdot \left\Vert X_{\alpha }\right\Vert _{v}\leq
\left\Vert X_{\alpha }\right\Vert _{v}\cdot \left\Vert Y\right\Vert
_{v}\cdot \sqrt{d\cdot (\left\Vert b\right\Vert _{v}^{2}-1)}
\end{equation*}%
As this is true for all $Y\in \mathfrak{t},$ we indeed obtain $(\ref{talpha}%
) $.

Now assume $v$ is non Archimedean, then (\ref{sim}) shows that
\begin{equation*}
\left\Vert x(1-\alpha (a))\alpha (Y)X_{\alpha }\right\Vert _{v}\leq
\left\Vert b\right\Vert _{v}||Y||_{v}
\end{equation*}%
which is what we wanted.
\endproof%

\begin{proposition}
\label{upper1}There are explicitely computable positive constants $%
(C_{i})_{1\leq i\leq 3}$ depending only on $d=\dim \mathfrak{g}$ and $%
p=|\Phi ^{+}|$ such that for any $a\in T(\overline{\mathbb{Q}}_{v})$ regular
and $u\in U_{0}(\overline{\mathbb{Q}}_{v}),$ we have
\begin{equation}
\left\Vert Ad(u)\right\Vert _{v}\leq C_{3}\cdot \left\Vert
Ad(uau^{-1})\right\Vert _{v}^{C_{1}}\cdot \left( \prod_{i=1}^{p}\max
\{1,L_{i}\}\right) ^{C_{2}}  \label{for}
\end{equation}%
where $L_{i}=\left( |1-\alpha _{i}(a)|_{v}\cdot |\alpha _{i}|_{v}\right)
^{-1}$. Moreover, if $v$ is non Archimedean, then (\ref{for}) holds with $%
C_{3}=1$.
\end{proposition}

\proof%
Recall that we may write $u=e_{\alpha _{p}}(x_{p})\cdot ...\cdot e_{\alpha
_{1}}(x_{1})$, where $p=|\Phi ^{+}|$ and $x_{i}\in \overline{\mathbb{Q}}_{v}$
for each $i.$ We want to apply Lemma \ref{onestep} recursively starting with
$\alpha =\alpha _{p}$ and going up to $\alpha _{1}.$ For each $\alpha \in
\Phi ^{+}$ let $u_{\alpha }=e_{\alpha _{i_{\alpha }-1}}(x_{i_{\alpha
}-1})\cdot ...\cdot e_{\alpha _{1}}(x_{1})$ and $n_{\alpha }=e_{\alpha
}(x_{\alpha }).$ For each $i\in \lbrack 1,p]$ we have $u_{\alpha
_{i+1}}au_{\alpha _{i+1}}^{-1}=n_{\alpha }u_{\alpha }au_{\alpha
}^{-1}n_{\alpha }^{-1}=n_{\alpha }av_{\alpha }n_{\alpha }^{-1},$ where $%
\alpha =\alpha _{i}$ $v_{\alpha }=a^{-1}u_{\alpha }au_{\alpha }^{-1}\in
U_{\alpha }.$

We set $b_{p+1}=Ad(uau^{-1})$ and $b_{i}=Ad(u_{\alpha _{i}}au_{\alpha
_{i}}^{-1})$. Lemma \ref{onestep} gives for each $i\in \lbrack 1,p]$,
\begin{equation*}
\left\Vert x_{\alpha _{i}}X_{\alpha _{i}}\right\Vert _{v}\leq f_{v}\cdot
L_{i}\cdot \left\Vert b_{i+1}\right\Vert _{v}
\end{equation*}%
where $f_{v}=\sqrt{d}$ if $v$ is Archimedean, and $f_{v}=1$ otherwise. Since
$b_{i+1}=Ad(n_{\alpha _{i}})b_{i}Ad(n_{\alpha _{i}}^{-1})$, we have
\begin{equation*}
\left\Vert b_{i}\right\Vert _{v}\leq \left\Vert b_{i+1}\right\Vert _{v}\cdot
e^{2c_{v}}\cdot \max \{1,\left\Vert x_{\alpha _{i}}X_{\alpha
_{i}}\right\Vert _{v}\}^{2d}
\end{equation*}%
where we have used (\ref{expbound}). Hence combining the last two lines:
\begin{equation}
\left\Vert b_{i}\right\Vert _{v}\leq \mu _{i}\cdot \left\Vert
b_{i+1}\right\Vert _{v}^{2d+1}  \label{recurse}
\end{equation}%
where $\mu _{i}=e^{2c_{v}}f_{v}^{2d}\max \{1,L_{i}\}^{2d}$.

On the other hand $\left\| Ad(u)\right\| _{v}\leq \prod_{\alpha \in \Phi
^{+}}\left\| Ad(e_{\alpha }(x_{\alpha }))\right\| _{v}$ and using (\ref%
{expbound}) again we obtain
\begin{eqnarray*}
\left\| Ad(u)\right\| _{v} &\leq &e^{pc_{v}}\cdot \left( \prod_{\alpha \in
\Phi ^{+}}\max \{1,\left\| x_{\alpha }X_{\alpha }\right\| _{v}\}\right) ^{d}
\\
&\leq &e^{pc_{v}}\cdot f_{v}^{dp}\cdot \left( \prod_{i=1}^{p}\max
\{1,L_{i}\}\right) ^{d}\cdot \left( \prod_{i=2}^{p+1}\left\| b_{i}\right\|
_{v}\right) ^{d}
\end{eqnarray*}
It remains to estimate the last term in the right hand side. Recursively
from (\ref{recurse}), we get
\begin{equation*}
\prod_{i=2}^{p+1}\left\| b_{i}\right\| _{v}\leq \left\| b\right\|
_{v}^{\sum_{k=0}^{p-1}(2d+1)^{k}}\cdot \prod_{i=2}^{p}\prod_{k=i}^{p}\mu
_{k}^{(2d+1)^{k-i}}
\end{equation*}
Hence we do indeed obtain a bound of the desired form.
\endproof%

The above proposition is useful to bound $\left\| Ad(u)\right\| _{v}$ when $%
\left\| Ad(uau^{-1})\right\| _{v}$ may be large. We now need an estimate
(only when $v$ is Archimedean) when this norm is small. Let $L_{i}$ be
defined as in the previous statement.

\begin{proposition}
\label{upper2}Suppose $v$ is Archimedean. Then there are positive constants $%
(D_{i})_{1\leq i\leq 3}$ depending only on $d=\dim \mathfrak{g}$ and $%
p=|\Phi ^{+}|$, such that for any $u\in U_{0}(\overline{\mathbb{Q}}_{v})$
and $a\in T(\overline{\mathbb{Q}}_{v})$ regular with $\log \left\Vert
Ad(uau^{-1})\right\Vert _{v}\leq 1,$ we have
\begin{equation*}
\log \left\Vert Ad(u)\right\Vert _{v}\leq D_{3}\cdot L_{v}^{D_{2}}\cdot
\left( \log \left\Vert Ad(uau^{-1})\right\Vert _{v}\right) ^{D_{1}}
\end{equation*}%
where $L_{v}=\prod_{i=1}^{p}\max \{1,L_{i}(a)_{v}\}.$
\end{proposition}

\proof%
In this proof, by a constant we mean a positive number that depends only on $%
d$ and $p$. Observe that there exists $\varepsilon _{1}>0$ such that $\sqrt{%
x^{2}-1}\leq 2\sqrt{\log x}$ as soon as $x\geq 1$ and $\log x\leq
\varepsilon _{1}.$ We keep the notations of the proof of the previous
proposition. Applying Lemma \ref{onestep}, we thus obtain that as soon as $%
\ell _{i+1}\leq \varepsilon _{1}$%
\begin{equation*}
\left\Vert x_{\alpha _{i}}X_{\alpha _{i}}\right\Vert _{v}\leq 2\sqrt{d}\cdot
L_{i}\cdot \sqrt{\ell _{i+1}}
\end{equation*}%
where we set $\ell _{i}=\log \left\Vert b_{i}\right\Vert _{v}$ for each $%
i\in \lbrack 1,p],$ and $\ell =\ell _{p+1}=\log \left\Vert
Ad(uau^{-1})\right\Vert _{v}.$

We may choose a smaller $\varepsilon _{1}$ so that
\begin{equation*}
\left\Vert Ad(e_{\alpha }(x))\right\Vert _{v}\leq 1+2c_{v}\left\Vert
xX_{\alpha }\right\Vert _{v}
\end{equation*}%
for each $\alpha \in \Phi ^{+}$ as soon as $|x|_{v}\leq \varepsilon _{1}$ as
we see from (\ref{expbound}). Hence if $\sqrt{\ell _{i+1}}\leq \frac{%
\varepsilon _{1}}{2\sqrt{d}\cdot L\cdot },$ then
\begin{equation*}
\left\Vert b_{i}\right\Vert _{v}\leq \left\Vert b_{i+1}\right\Vert _{v}\cdot
\left( 1+4\sqrt{d}\cdot L_{i}\cdot \sqrt{\ell _{i+1}}\right) ^{2}
\end{equation*}%
or
\begin{eqnarray*}
\ell _{i} &\leq &\ell _{i+1}+8\sqrt{d}\cdot L_{i}\cdot \sqrt{\ell _{i+1}} \\
&\leq &C\cdot L\cdot \sqrt{\ell _{i+1}}
\end{eqnarray*}%
for some constant $C$. Applying this recursively, we see that, as soon as $L$
is bigger than some constant, if $\ell \leq \frac{\varepsilon _{1}^{2^{p+1}}%
}{L^{3^{p+1}}}$ then, for each $i\in \lbrack 1,p],$ $\ell _{i}\leq \frac{%
\varepsilon _{1}^{2^{p+1}}}{L^{3^{p+1}}}$ and
\begin{equation*}
\ell _{i}\leq C^{\prime }\cdot L^{2}\cdot \ell ^{\frac{1}{2^{p+1-i}}}
\end{equation*}%
for each $i\in \lbrack 1,p]$ and some constant $C^{\prime }.$ On the other
hand $\left\Vert Ad(u)\right\Vert _{v}\leq \prod_{\alpha \in \Phi
^{+}}\left\Vert Ad(e_{\alpha }(x_{\alpha }))\right\Vert _{v}\leq
\prod_{\alpha \in \Phi ^{+}}e^{c_{v}\left\Vert x_{\alpha }X_{\alpha
}\right\Vert _{v}}$ and
\begin{eqnarray*}
\log \left\Vert Ad(u)\right\Vert _{v} &\leq &c_{v}\cdot
\sum_{i=1}^{p}\left\Vert x_{\alpha _{i}}X_{\alpha _{i}}\right\Vert _{v}\leq
C^{\prime \prime }\cdot L\cdot \sqrt{\sum_{2\leq i\leq p+1}\ell _{i}} \\
&\leq &C^{\prime \prime \prime }\cdot L^{2}\cdot \ell ^{\frac{1}{2^{p+1}}}
\end{eqnarray*}%
On the other hand the cruder bound obtained in Proposition \ref{upper1}
shows that without a condition on $\ell ,$
\begin{equation*}
\log \left\Vert Ad(u)\right\Vert _{v}\leq \log C_{3}+C_{1}\cdot \ell
+C_{2}\cdot \log L
\end{equation*}%
hence
\begin{equation*}
\log \left\Vert Ad(u)\right\Vert _{v}\leq C_{0}\cdot L
\end{equation*}%
for some constant $C_{0}$ if $\ell \leq 1$ and $L$ larger than some
constant. Take $D_{1}=\frac{1}{2^{p+1}},$ $D_{2}\geq 1+\left( \frac{3}{2}%
\right) ^{p+1}$ and $D_{3}\geq \max \{\frac{C_{0}}{\varepsilon _{1}}%
,C^{\prime \prime \prime }\}.$ Then if $\ell \geq \frac{\varepsilon
_{1}^{2^{p+1}}}{L^{3^{p+1}}},$ we have $D_{3}\cdot L^{D_{2}}\cdot \ell ^{%
\frac{1}{2^{p+1}}}\geq D_{3}\cdot L\cdot \varepsilon _{1}\geq C_{0}\cdot L.$
Therefore as soon as $\ell \leq 1$ and $L$ larger than some constant say $%
C_{4},$ we have
\begin{equation*}
\log \left\Vert Ad(u)\right\Vert _{v}\leq D_{3}\cdot L^{D_{2}}\cdot \ell
^{D_{1}}
\end{equation*}%
Hence up to changing $D_{3}$ into $D_{3}C_{4}^{D_{2}}$ if necessary, we
obtain the desired result.
\endproof%

\section{Global bounds on arithmetic heights\label{global}}

In this section we gather together all the local data obtained in the
previous section and sum it up to obtain a global bound (see Proposition \ref%
{heightbound} below) on the height of the matrix coefficients of the finite
set $F.$

Recall our set of notations from the last section (see \S \ref{notat}). $%
\mathbb{G}$ is a Chevalley group of adjoint type and $T$ a maximal torus. We
had fixed a total order on the set of all roots induced by an ordering of
the simple roots, that is $\Phi =\{\alpha _{1},...,\alpha _{|\Phi
^{+}|},\alpha _{|\Phi ^{+}|+r+1},...,\alpha _{d}\}$, where $r$ is the rank
of $\mathfrak{g}=Lie(\mathbb{G})$ and $I_{r}=[|\Phi ^{+}|,|\Phi ^{+}|+r]$.
The Lie algebra $\mathfrak{g}$ has a basis $(Y_{1},...,Y_{d})$ obtained from
a Chevalley basis of $\mathfrak{g}$, with $Y_{i}=X_{\alpha _{i}}$ if $%
i\notin I_{r}$ and $Y_{i}\in \{\omega _{\alpha },\alpha \in \Pi \}$ if $i\in
I_{r}.$ Also $\mathfrak{g}_{\mathbb{Z}}$ denotes the integer lattice
generated by the basis $(Y_{1},...,Y_{d}).$ Recall further that for $X,Y\in
\mathfrak{g}$ we had set $\phi (X,Y)=-B(X^{\tau },Y)$ where $B$ is the
Killing form and $\tau $ the Chevalley involution. Note that (\ref{rootdec})
is an orthogonal decomposition for the symmetric bilinear form $\phi $.

We will consider the elements $A=Ad(a)$ and $B=Ad(b)$ from $%
F=\{Id,a,b\}\subset \mathbb{G}(\overline{\mathbb{Q}})$ with $a\in T$ as
matrices in the basis $(Y_{1},...,Y_{d})$. Then $A$ is diagonal and $%
B=(b_{ij})_{ij}\in SL_{d}(\overline{\mathbb{Q}})$. Consider the regular
function on $\mathbb{G}$ given by $f(g)=g_{dd}$ in this basis. The root $%
\alpha \left( d\right) $ is the smallest root in the above ordering. It
coincides with the opposite of the highest root of $\Phi $ in the sense of
\cite[VI.1.8.]{Bourb}. Observe the following:

$-$ for every $t\in T$, we have $f(tgt^{-1})=f(g).$

$-$ for every $t\in T$, $f(t)=\alpha _{d}(t)$, hence $f$ is not constant.

$-$ $\phi (Ad(g)Y_{d},Y_{d})=f(g)\phi (Y_{d},Y_{d}),$

$-$ for every place $v$ we have $|f(g)|_{v}\leq ||Ad(g)||_{v}$.

Recall that we consider $\mathbb{G}$ as a subgroup of $SL(\mathfrak{g})$ and
thus define the heights $e$ and $\widehat{h}$ of finite subsets of $\mathbb{%
G(}\overline{\mathbb{Q}})$ with respect to the adjoint representation. The
goal of this section is to prove:

\begin{proposition}
\label{heightbound}For every $n\in \mathbb{N}$ and any $\alpha >0$ there is $%
\eta >0$ and $A_{1}>0$ such that if $F=\{Id,a,b\}$ is a subset of $\mathbb{G}%
(\overline{\mathbb{Q}})$ with $a\in T(\overline{\mathbb{Q}})$ such that $%
e(F)<\eta $ and $\deg (\alpha _{i}(a))>A_{1}$ for each positive root $\alpha
_{i},$ then we have for every $i\in \mathbb{N}$, $1\leq i\leq n$%
\begin{equation*}
h(f(b^{i}))<\alpha
\end{equation*}%
where $f$ is the function defined above.
\end{proposition}

This proposition is central to the proof of the main theorem of this paper,
that is Theorem \ref{main}. How to proceed from it to a complete proof of
Theorem \ref{main} will be explained in the next section. For the moment,
let us just say that given the assumptions of Theorem \ref{main}, if $%
\widehat{h}(F)$ is small, then by escape from subvarieties (see Proposition %
\ref{twoelements}) one may find many pairs $\{a,b\}$ in a bounded power of $%
F $ that satisfy the requirements of the above proposition, indeed one may
find such $\{a,b\}$ outside every given subvariety of $\mathbb{G\times G}$.
Applying Zhang's theorem (see \cite{Zhang} Theorem \ref{Zh} below), we will
see however that the height bounds imposed upon the $f(b^{i})$'s by the
above proposition will force $\{a,b\}$ to belong to a proper algebraic
variety of $\mathbb{G}$, thus contradicting the choice of $\{a,b\}.$

We now begin the proof of the above proposition. It will make use of the
local estimates obtained in the previous section as well as Bilu's
equidistribution theorem (see below). The proof will occupy the next two
subsections. First, we collect the local estimates and see what bounds they
give us. Then we use Bilu's theorem to show that the remainder terms give
only a small contribution to the height.

\subsection{Preliminary upper bounds\label{prelimupper}}

Recall that the $(C_{i})_{1\leq i\leq 3}$'s and $(D_{i})_{1\leq i\leq 3}$'s
are the constants obtained in Propositions \ref{upper1} and \ref{upper2}.
For $A\geq 1$ and $x\in \overline{\mathbb{Q}}$ we set
\begin{equation}
h_{\infty }^{A}(x):=\frac{1}{[K:\mathbb{Q}]}\sum_{v\in V_{\infty
},|x|_{v}\geq A}n_{v}\cdot \log ^{+}|x|_{v}  \label{hinfty}
\end{equation}
where the sum is limited to those $v\in V_{\infty }$ such that $|x|_{v}\geq
A $. In this paragraph, we prove the following.

\begin{proposition}
\label{heightbound0}There are positive constants $A_{0}$, $C_{4}$, and $%
D_{4} $ such that if $\varepsilon >0$ and $A\geq A_{0}$ are arbitrary, then
for any $j\in \mathbb{N}$ and any $a,b\in \mathbb{G}(K)$ two regular
semisimple $K$-rational elements for some number field $K$ with $a\in T(K)$,
\begin{equation}
\frac{h(f(b^{j}))}{j}\leq 4\frac{\log A}{\varepsilon }e(\{a,b%
\})+D_{4}A^{D_{2}}\varepsilon ^{D_{1}}+C_{4}\sum_{1\leq i\leq p}\left(
h_{f}(\delta _{i}^{-1})+h_{\infty }^{A}(\delta _{i}^{-1})\right)  \label{y}
\end{equation}%
where $\delta _{i}=1-\alpha _{i}(a)$ for each positive root $\alpha _{i}\in
\Phi ^{+}$ and $p=|\Phi ^{+}|$.
\end{proposition}

We set as before $F=\{a,b\}.$ For each place $v$ let $s_{v}>\log
(E_{v}^{Ad}(F))$ be some real number. According to Lemma \ref{Mostow}, there
exists $g_{v}\in \mathbb{G}(\overline{\mathbb{Q}}_{v})$ such that $%
\left\Vert Ad(g_{v}Fg_{v}^{-1})\right\Vert _{v}\leq e^{s_{v}}$. Since by the
Iwasawa decomposition we have $\mathbb{G}(\overline{\mathbb{Q}}_{v})=\mathbf{%
K}_{v}U_{0}(\overline{\mathbb{Q}}_{v})T(\overline{\mathbb{Q}}_{v}),$ and $%
\mathbf{K}_{v}$ stabilizes the norm, we may assume that $g_{v}\in U_{0}(%
\overline{\mathbb{Q}}_{v})T(\overline{\mathbb{Q}}_{v}),$ i.e. $%
g_{v}=u_{v}\cdot t_{v}.$ Since $t$ commutes with $a$ we get
\begin{eqnarray*}
\left\Vert Ad(u_{v}au_{v}^{-1})\right\Vert _{v} &\leq &e^{s_{v}} \\
\left\Vert Ad(u_{v}b^{t_{v}}u_{v}^{-1})\right\Vert _{v} &\leq &e^{s_{v}}
\end{eqnarray*}%
where $b^{t_{v}}=t_{v}bt_{v}^{-1}.$ Recall that $d=\dim \mathbb{G}$,

According to Proposition \ref{upper1} we have
\begin{eqnarray}
\left\Vert Ad(b^{t_{v}})\right\Vert _{v} &\leq &e^{s_{v}}\cdot \left\Vert
Ad(u_{v})\right\Vert _{v}^{d}  \label{Bbound} \\
&\leq &e^{s_{v}}\cdot C_{3}^{d}\cdot \left\Vert
Ad(u_{v}au_{v}^{-1})\right\Vert _{v}^{dC_{1}}\cdot \left(
\prod_{i=1}^{p}\max \{1,L_{i}(a)_{v}\}\right) ^{dC_{2}}  \notag \\
&\leq &C_{3}^{d}\cdot \left( \prod_{i=1}^{p}\max \{1,L_{i}(a)_{v}\}\right)
^{dC_{2}}\cdot e^{s_{v}(1+dC_{1})}  \notag
\end{eqnarray}%
with $C_{3}=1$ if $v$ is non Archimedean. Let $L_{v}=\prod_{i=1}^{p}\max
\{1,L_{i}(a)_{v}\}.$ We get
\begin{equation}
\log \left\Vert Ad(b^{t_{v}})\right\Vert _{v}\leq d\log C_{3}+dC_{2}\cdot
\log L_{v}+(1+dC_{1})\cdot s_{v}  \label{1st}
\end{equation}%
Now assume $v$ is Archimedean. According to Proposition \ref{upper2} we have
constants $D_{i}>0$ such that if $s_{v}\leq 1$ then
\begin{eqnarray}
\log \left\Vert Ad(b^{t_{v}})\right\Vert _{v} &\leq &s_{v}+d\log \left\Vert
Ad(u_{v})\right\Vert _{v}\leq s_{v}+dD_{3}L_{v}^{D_{2}}\cdot s_{v}^{D_{1}}
\label{2nd} \\
&\leq &D_{4}^{\prime }L_{v}^{D_{2}}\cdot s_{v}^{D_{1}}  \notag
\end{eqnarray}%
where $D_{4}^{\prime }=dD_{3}+1$ and where we have chosen $D_{1}\leq 1$ as
we may so that $s_{v}\leq s_{v}^{D_{1}}$. Since $|f(b^{j})|_{v}\leq
\left\Vert Ad(b^{t_{v}})\right\Vert _{v}^{j}$ for each $j\in \lbrack 1,n]$
and $v,$ we have obtained:
\begin{equation*}
\frac{h(f(b^{j}))}{j}\leq \frac{1}{[K:\mathbb{Q}]}\sum_{v\in
V_{K}}n_{v}\cdot \log \left\Vert Ad(b^{t_{v}})\right\Vert _{v}
\end{equation*}%
In order to prove Proposition \ref{heightbound0}, we will decompose this sum
into four parts. Let $\kappa =\min_{i}|\alpha _{i}|_{\infty }$ (see Lemma %
\ref{norm} and the remark following it for the definition of $|\alpha
_{i}|_{\infty }$). We split the set of places $v$ into four parts: $v\in
V_{\infty }$, $s_{v}\leq \varepsilon $ and $L_{v}\geq A/\kappa $ (this gives
$H_{\leq }^{+}$), $v\in V_{\infty }$, $s_{v}\leq \varepsilon $ and $%
L_{v}<A/\kappa $ (this gives $H_{\leq }^{-}$), $v\in V_{\infty }$ and $%
s_{v}>\varepsilon $ (this gives $H_{\geq }$) and finally $v\in V_{f}$ (this
gives $H_{f}$). So
\begin{equation*}
\frac{h(f(b^{j}))}{j}\leq H_{\leq }^{-}+H_{\leq }^{+}+H_{\geq }+H_{f}
\end{equation*}%
Making use of the bound (\ref{1st}) for $H_{\leq }^{+},$ $H_{\geq }$ and $%
H_{f}$ and the bound (\ref{2nd}) for $H_{\leq }^{-}$ respectively, we obtain
the following estimates as soon as $A$ is large enough ($\log A>\log
A_{0}:=1+dC_{1}+d\log C_{3}+\log |\kappa |,$ we also set $C_{4}=4dC_{2}$) :

\begin{equation*}
H_{f}\leq (1+dC_{1})\frac{1}{[K:\mathbb{Q}]}\sum_{v\in V_{f}}n_{v}\cdot
s_{v}+(dC_{2})\frac{1}{[K:\mathbb{Q}]}\sum_{v\in V_{f}}n_{v}\cdot \log L_{v}
\end{equation*}
\begin{eqnarray*}
H_{\geq } &\leq &\left( \frac{d\log C_{3}}{\varepsilon }+(1+dC_{1})\right)
\cdot \frac{1}{[K:\mathbb{Q}]}\sum_{v\in V_{\infty },s_{v}\geq \varepsilon
}n_{v}\cdot s_{v}+\frac{C_{4}}{4}\frac{1}{[K:\mathbb{Q}]}\sum_{v\in
V_{\infty },s_{v}\geq \varepsilon }n_{v}\cdot \log L_{v} \\
&\leq &\frac{4\log A}{\varepsilon }\cdot \frac{1}{[K:\mathbb{Q}]}\sum_{v\in
V_{\infty },s_{v}>\varepsilon }n_{v}\cdot s_{v}+\frac{C_{4}}{4}\frac{1}{[K:%
\mathbb{Q}]}\sum_{v\in V_{\infty },s_{v}>\varepsilon ,L_{v}\geq A/\kappa
}n_{v}\cdot \log L_{v}
\end{eqnarray*}
\begin{eqnarray*}
H_{\leq }^{+} &\leq &(2dC_{2})\cdot \frac{1}{[K:\mathbb{Q}]}\sum_{v\in
V_{\infty },s_{v}\leq \varepsilon ,L_{v}\geq A/\kappa }n_{v}\cdot \log L_{v}
\\
H_{\leq }^{-} &\leq &\frac{1}{[K:\mathbb{Q}]}\sum_{v\in V_{\infty
},s_{v}\leq \varepsilon ,L_{v}<A/\kappa }n_{v}\cdot D_{4}^{\prime
}L_{v}^{D_{2}}\cdot s_{v}^{D_{1}}\leq 2\frac{D_{4}^{\prime }}{\kappa ^{D_{2}}%
}A^{D_{2}}\varepsilon ^{D_{1}}\leq D_{4}A^{D_{2}}\varepsilon ^{D_{1}}
\end{eqnarray*}
for $D_{4}=2\frac{D_{4}^{\prime }}{\kappa ^{D_{2}}},$ since $n_{v}\leq 2$
for $v\in V_{\infty }.$

Note that $e(F)=\frac{1}{[K:\mathbb{Q}]}\sum_{v\in V_{K}}n_{v}\cdot s_{v},$
so the above bounds give :
\begin{equation}
\frac{h(f(b^{j}))}{j}\leq 4\frac{\log A}{\varepsilon }e(F)+\frac{C_{4}}{2}%
\frac{1}{[K:\mathbb{Q}]}\left( \sum_{v\in V_{\infty },L_{v}\geq A/\kappa
}n_{v}\cdot \log L_{v}+\sum_{v\in V_{f}}n_{v}\cdot \log L_{v}\right)
+D_{4}A^{D_{2}}\varepsilon ^{D_{1}}  \label{u}
\end{equation}
On the other hand $\log L_{v}\leq \sum_{1\leq i\leq p}\log ^{+}L_{i}(a)_{v}$
where $L_{i}(a)_{v}=|\delta _{i}^{-1}|_{v}/|\alpha _{i}|_{v}$ and $\delta
_{i}=1-\alpha _{i}(a).$

Clearly if $L_{i}(a)_{v}\geq A/\kappa \geq \kappa ^{-2}$ then $|\delta
_{i}^{-1}|_{v}\geq A$ and $L_{i}(a)_{v}\leq |\delta _{i}^{-1}|_{v}^{2}$ $.$
We get:
\begin{eqnarray}
\sum_{v\in V_{\infty },L_{v}\geq A/\kappa }n_{v}\cdot \log L_{v} &\leq
&\sum_{1\leq i\leq p}\sum_{v\in V_{\infty },L_{i}(a)\geq A/\kappa
}n_{v}\cdot \log ^{+}L_{i}(a)_{v}  \label{u'} \\
&\leq &2\cdot \sum_{1\leq i\leq p}\sum_{v\in V_{\infty },|\delta
_{i}|_{v}\leq A^{-1}}n_{v}\cdot \log ^{+}|\delta _{i}^{-1}|_{v}  \notag
\end{eqnarray}%
Now note that for $v\in V_{f}$ we have $|\alpha _{i}|_{v}=1$ according to
Lemma \ref{norm}. It follows that
\begin{equation}
\sum_{v\in V_{f}}n_{v}\cdot \log ^{+}L_{i}(a)_{v}=\sum_{v\in
V_{f}}n_{v}\cdot \log ^{+}|\delta _{i}^{-1}|_{v}=[K:\mathbb{Q}]\cdot
h_{f}(\delta _{i}^{-1})  \label{v}
\end{equation}%
Hence combining $(\ref{u})$ with $(\ref{u'}),$ $(\ref{v})$ we obtain $(\ref%
{y})$ and this ends the proof of Proposition \ref{heightbound0}.

\subsection{Bilu's equidistribution theorem}

We are now going to apply Bilu's equidistribution theorem to show that the
last term in estimate $(\ref{y})$ becomes very small when both $A$ is large
and $e(F)$ is small.

\begin{theorem}
(Bilu's Equidistribution of Small Points \cite{Bilu})\label{EQ} Suppose $%
(\lambda _{n})_{n\geq 1}$ is a sequence of algebraic numbers (i.e. in $%
\overline{\mathbb{Q}}$) such that $h(\lambda _{n})\rightarrow 0$ and $\deg
(\lambda _{n})\rightarrow +\infty $ as $n\rightarrow +\infty .$ Let $%
\mathcal{O}(\lambda _{n})$ be the Galois orbit of $\lambda _{n}$. Then we
have the following weak-$*$ convergence of probability measures on $\mathbb{C%
}$,
\begin{equation}
\frac{1}{\#\mathcal{O}(\lambda _{n})}\sum_{x\in \mathcal{O}(\lambda
_{n})}\delta _{x}\underset{n\rightarrow +\infty }{\rightarrow }d\theta
\label{equi}
\end{equation}
where $d\theta $ is the normalized Lebesgue measure on the unit circle $%
\{z\in \mathbb{C}$, $|z|=1\}.$
\end{theorem}

Let us first draw two consequences of this equidistribution statement:

\begin{lemma}
\label{lillem0}For every $\alpha >0$ there is $A_{1}>0$, $\eta _{1}>0$ and $%
\varepsilon _{1}>0$ with the following property. If $\lambda \in \overline{%
\mathbb{Q}}$ is such that $h(\lambda )\leq \eta _{1}$ and $\deg (\lambda
)>A_{1} $ then
\begin{equation}
h_{\infty }^{\varepsilon _{1}^{-1}}(\frac{1}{1-\lambda })\leq \alpha
\label{toprove}
\end{equation}
where $h_{\infty }^{\varepsilon _{1}^{-1}}$ was defined in $(\ref{hinfty})$.
\end{lemma}

\proof%
We have
\begin{equation*}
h_{\infty }(\frac{1}{1-\lambda })\leq h(\frac{1}{1-\lambda })=h(1-\lambda
)\leq h_{f}(\lambda )+h_{\infty }(1-\lambda )\leq h(\lambda )+h_{\infty
}(1-\lambda )
\end{equation*}
Hence
\begin{equation*}
\frac{1}{\deg (\lambda )}\sum_{x\in \mathcal{O}(\lambda )}\log \frac{1}{|1-x|%
}=h_{\infty }(\frac{1}{1-\lambda })-h_{\infty }(1-\lambda )\leq h(\lambda )
\end{equation*}
and
\begin{equation}
h_{\infty }^{\varepsilon _{1}^{-1}}(\frac{1}{1-\lambda })=\frac{1}{\deg
(\lambda )}\sum_{|1-x|\leq \varepsilon _{1}}\log \frac{1}{|1-x|}\leq
h(\lambda )+\frac{1}{\deg (\lambda )}\sum_{|1-x|>\varepsilon _{1}}\log |1-x|
\label{in}
\end{equation}
Consider the function $f_{\varepsilon _{1}}(z)=\mathbf{1}_{|z-1|>\varepsilon
_{1}}\log |1-z|$. It is locally bounded on $\mathbb{C}$. By Theorem \ref{EQ}%
, for every $\varepsilon _{1}>0$, there must exists $\eta _{1}>0$ and $%
A_{1}>0$ such that, if $h(\lambda )\leq \eta _{1},$ and $d=\deg (\lambda
)>A_{1}$, then
\begin{equation*}
\left| \frac{1}{\deg (\lambda )}\sum_{x}f_{\varepsilon
_{1}}(x)-\int_{0}^{1}f_{\varepsilon _{1}}(e^{2\pi i\theta })d\theta \right|
\leq \frac{\alpha }{3}
\end{equation*}
On the other hand we verify that $\int_{0}^{1}\log |1-e^{2\pi i\theta
}|d\theta =0.$ Hence we can choose $\varepsilon _{1}>0$ small enough so that
$\left| \int_{0}^{1}f_{\varepsilon _{1}}(e^{2\pi i\theta })d\theta \right|
\leq \frac{\alpha }{3}.$ Combining these inequalities with $(\ref{in})$ and
choosing $\eta _{1}\leq \frac{\alpha }{3},$ we get $(\ref{toprove}).$

\endproof%

\begin{lemma}
\label{lillem}For every $\alpha >0$ there exists $\eta >0$ and $A_{1}>0$
such that for any $\lambda \in \overline{\mathbb{Q}},$ if $h(\lambda )\leq
\eta $ and $d=\deg (\lambda )>A_{1}$, then
\begin{equation*}
\left| \frac{1}{\deg (\lambda )}\sum_{v\in V_{\infty }}n_{v}\cdot \log
|1-\lambda |_{v}\right| \leq \alpha
\end{equation*}
\end{lemma}

\proof%
The previous lemma shows that the convergence $(\ref{equi})$ not only holds
for compactly supported functions on $\mathbb{C}$, but also for functions
with logarithmic singularities at $1.$ In particular it holds for the
function $f(z)=\log |1-z|,$ which is exactly what we need, since we check
easily that $\int_{0}^{1}f(e^{2\pi i\theta })d\theta =0.$
\endproof%

As a consequence we obtain:

\begin{lemma}
\label{lillem2}For every $\alpha >0$ there exists $\eta _{0}>0$ and $A_{1}>0$
such that for any $\lambda \in \overline{\mathbb{Q}},$ if $h(\lambda )\leq
\eta _{0}$ and $d=\deg (\lambda )>A_{1}$, then
\begin{equation*}
h_{f}(\frac{1}{1-\lambda })\leq 2\alpha
\end{equation*}
\end{lemma}

\proof%
We apply the product formula to $\delta =1-\lambda $, which takes the form $%
h(\delta )=h(\delta ^{-1}),$ hence
\begin{equation*}
h_{f}(\delta ^{-1})=h_{\infty }(\delta )-h_{\infty }(\delta
^{-1})+h_{f}(\delta )
\end{equation*}
But $h_{f}(\delta )=h_{f}(1-\lambda )\leq h_{f}(\lambda )\leq \eta _{0}$ and
$h_{\infty }(\delta )-h_{\infty }(\delta ^{-1})=\frac{1}{[K:\mathbb{Q}]}%
\sum_{v\in V_{\infty }}n_{v}\cdot \log |\delta |_{v}$, which is bounded by $%
\alpha $ according to Lemma \ref{lillem}. We are done.
\endproof%

The outcome of all this is that each of the terms $h_{f}(\delta
_{i}^{-1})+h_{\infty }^{A}(\delta _{i}^{-1})$ in $(\ref{y})$ becomes small
as soon as $e(F)$ (hence $h(\alpha _{i}(a))$) becomes small and $A$ becomes
large.

\subsection{Proof of Proposition \protect\ref{heightbound}}

Let $n\in \mathbb{N}$ and $\alpha >0$ be arbitrary. Let $j\in \lbrack 1,n]$
an integer and $F=\{a,b\}\subset \mathbb{G}(\overline{\mathbb{Q}})$ with $%
a\in T(\overline{\mathbb{Q}}).$ Then\ for any $\varepsilon >0$ and $A>0$
large enough we obtained the upper bound $(\ref{y})$ above$.$ On the other
hand we had $h(\alpha _{i}(a))\leq e(F)$ for each positive root $\alpha _{i}$
and $\delta _{i}=1-\alpha _{i}(a).$ Let $\varepsilon _{1},$ $A_{1}$ and $%
\eta _{0}$ be the quantities obtained in the previous paragraph in Lemmas %
\ref{lillem0} and \ref{lillem2}. Choose $A$ so that $A^{-1}<\varepsilon _{1}$
and $A\geq A_{0}$ and consider $(\ref{y}).$ Assume that for each $i\in
\{1,...,p\}$ $\deg (\alpha _{i}(a))>A_{1}$. Then Lemmas \ref{lillem0} and %
\ref{lillem2} will hold with $\lambda =\alpha _{i}(a)$ as soon as $e(F)<\eta
_{0}$. Hence for each $i=1,...,p$
\begin{equation*}
\left\vert h_{f}(\delta _{i}^{-1})+h_{\infty }^{A}(\delta
_{i}^{-1})\right\vert \leq 2\alpha
\end{equation*}%
and
\begin{equation*}
\frac{h(f(b^{j}))}{j}\leq \frac{4\log A}{\varepsilon }e(\{a,b%
\})+D_{4}A^{D_{2}}\varepsilon ^{D_{1}}+2p(4dC_{2})\alpha
\end{equation*}%
Now choose $\varepsilon >0$ so that $2D_{4}A^{D_{2}}\varepsilon
^{D_{1}}<\alpha .$ Then choose $\eta >0$ so that $4\frac{\log A}{\varepsilon
}\eta <\alpha $ and $\eta <\eta _{0}.$ From $(\ref{y})$, we then obtain that
if $e(F)<\eta $ and $j\in \mathbb{N}$%
\begin{equation*}
\frac{1}{j}h(f(b^{j}))\leq (2+p(4dC_{2}))\alpha
\end{equation*}%
Since $\alpha $ was arbitrary we obtain the desired bound.

\section{Proof of the statements of Section 3}

\subsection{Proof of Theorem \protect\ref{main}}

Before beginning the proof of Theorem \ref{main}, we recall Zhang's theorem
on small points of algebraic tori. Let $\mathbb{G}_{m}$ be the
multiplicative group and $n\in \mathbb{N}$. On the $\overline{\mathbb{Q}}$%
-points of the torus $\mathbb{G}_{m}^{n}$ we define a notion of height in
the following natural way. If $\mathbf{x}=(x_{1},...,x_{n})\in \mathbb{G}%
_{m}^{n}$ then $h(\mathbf{x}):=h(x_{1})+...+h(x_{n})$ where $h(x_{i})$ is
the standard logarithmic Weil height we have been using so far.

\begin{theorem}
\label{Zh}(Zhang \cite{Zhang}) Let $V$ be a proper closed algebraic
subvariety of $\mathbb{G}_{m}^{n}$ defined over $\overline{\mathbb{Q}}.$
Then there is $\varepsilon >0$ such that the Zariski closure $V_{\varepsilon
}$ of the set $\{\mathbf{x}\in V$, $h(\mathbf{x})<\varepsilon \}$ consists
of a finite union of torsion coset tori, i.e. subsets of the forms $\mathbf{%
\zeta }H$, where $\mathbf{\zeta }=(\zeta _{1},...,\zeta _{n})$ is a torsion
point and $H$ is a subtorus of $\mathbb{G}_{m}^{n}.$
\end{theorem}

We will need the following lemma, where $\mathbb{G}$ is a semisimple
algebraic group over an algebraically closed field, $T$ a maximal torus
together with a choice of simple roots $\Pi $, and $f$ is the regular
function defined at the beginning of the last section:

\begin{lemma}
\label{multindep}For every $k\in \mathbb{N}$, the regular functions $%
f_{1},...,f_{k}$ defined on $\mathbb{G}$ by $f_{i}(g)=f(g^{i})$ are
multiplicatively independent. Namely, if for each $i,$ $n_{i}$ and $m_{i}$
are non-negative integers and the $f_{i}$'s satisfy an equation of the form $%
f_{1}^{n_{1}}\cdot ...\cdot f_{k}^{n_{k}}=f_{1}^{m_{1}}\cdot ...\cdot
f_{k}^{m_{k}}$ then $n_{i}=m_{i}$ for each $i$.
\end{lemma}

\proof%
To prove this lemma it is enough to show that for each $i$ one can find a
group element $g\in \mathbb{G}$ such that $f_{i}(g)=0$ while all other $%
f_{j}(g)$'s are non zero. Let $H$ be the copy of $PGL_{2}$ corresponding to
the roots $\alpha =\alpha _{d}$ and $-\alpha =\alpha _{1}$ with Lie algebra $%
\mathfrak{h}$ generated by $X_{\alpha },$ $X_{-\alpha }$ and $H_{\alpha }.$
Clearly it is enough to prove the lemma for the restriction of $f$ to $H$.
Therefore without loss of generality we may assume that $\mathbb{G}=PGL_{2}$%
, hence $f(g)=a^{2}$ if $g=\left(
\begin{array}{ll}
a & b \\
c & d%
\end{array}%
\right) \in PGL_{2}.$ Let for instance $D_{\lambda }=\left(
\begin{array}{ll}
\lambda & 0 \\
0 & \lambda ^{-1}%
\end{array}%
\right) \in PGL_{2}$ and $P=\left(
\begin{array}{ll}
1 & 1 \\
1 & 2%
\end{array}%
\right) .$ Set $g_{\lambda }=PD_{\lambda }P^{-1}.$ Then compute $%
f(g_{\lambda })=2\lambda -\lambda ^{-1}$ and $f_{i}(g_{\lambda
})=f(g_{\lambda ^{i}}).$ Hence $f_{i}(g_{\lambda })=0$ if and only if $%
2\lambda ^{2i}=1.$ These conditions are mutually exclusive for distinct
values of $i$. So we are done.
\endproof%

We now conclude this subsection with the proof of Theorem \ref{main}.
According to the reductions made in Section \ref{reduc} we may assume that $%
F\subset \mathbb{G}(\overline{\mathbb{Q}})$ where $\mathbb{G}$ is a
connected absolutely almost simple algebraic group $\mathbb{G}$ of adjoint
type (viewed as embedded in $GL(\mathfrak{g})$ via the adjoint
representation), and that the group $\left\langle F\right\rangle $ is
Zariski dense in $\mathbb{G}$. Let $T$ be a maximal torus in $\mathbb{G}$
and $\Phi $ be the corresponding set of roots with set of simple roots $\Pi $
and let $\alpha _{1}=-\alpha _{d}$ be the highest root. The function $f\in
\overline{\mathbb{Q}}[\mathbb{G}]$ was defined at the beginning of Section %
\ref{global} by $f(g)=g_{dd}$ where $\{g_{ij}\}_{1\leq i,j\leq d}$ is the
matrix of $Ad(g)$ in the Chevalley basis $(Y_{1},...,Y_{d})$. Let $%
f_{i}(g)=f(g^{i})$ and let $\Omega $ be the Zariski open subset of $\mathbb{G}$ defined by $\{g,~f_{i}(g)\neq 0$ for each $i\leq d+1\}$. Let $\mathbf{f}$
be the regular map $\mathbf{f}(g):=(f_{1}(g),...,f_{d+1}(g))$ $:$ $\Omega
\rightarrow \mathbb{G}_{m}^{d+1}$. Since $d=\dim \mathbb{G}$, $Im(\mathbf{f})$ is not Zariski dense in $\mathbb{G}_{m}^{d+1}.$ Let $V$ be its
Zariski closure. According to the above theorem of Zhang, there is $\mu >0$
such that the Zariski closure $V_{\mu }$ of $\{\mathbf{x}
=(x_{1},...,x_{d+1})\in V \textnormal{ such that } h(\mathbf{x})<\mu \}$ is a finite
union of torsion coset tori. On the other hand, Lemma \ref{multindep} and
the Zariski connectedness of $\mathbb{G}$ shows that $V$ cannot be equal to
a finite union of torsion coset tori. Hence $V_{\mu }$ is a proper Zariski
closed subset of $V$. Let $Z_{\mu }=\Omega ^{c}\cup \mathbf{f}^{-1}\{V_{\mu
}\}.$ Then $Z_{\mu }$ is a proper Zariski-closed subset of $\mathbb{G}$.
Note that since $f$ is invariant under conjugation by $T,$ $Z_{\mu }$ also
is invariant under conjugation by $T.$ Let $\widehat{Z}_{\mu }$ the Zariski
closure of the set $\{(gag^{-1},gbg^{-1})\in \mathbb{G}^{2}$ with $g\in
\mathbb{G}$, $a\in T$ and $b\in Z_{\mu },$ or $a\in Z_{\mu }$ and $b\in T\}$
in $\mathbb{G\times G}.$ It is a proper Zariski closed subset, since $\dim
\widehat{Z}\leq 2\dim \mathbb{G}-1.$ Take $n=d+1$ and $\alpha =\mu /n$ in
Proposition \ref{heightbound}, which gives us an $A_{1}>0$ and an $\eta >0.$
According to Proposition \ref{twoelements} there is a number $c=c(\mathbb{G}%
,Z_{\mu },A_{1})>0$ such that $F^{c}$ contains two elements $a$ and $b$
which are $A_{1}$-regular semisimple elements, generate a Zariski-dense
subgroup of $\mathbb{G}$ and satisfy $(a,b)\notin \widehat{Z}_{\mu }.$ Now
let $\varepsilon =\eta /c$ and assume that $e(F)<\varepsilon .$ Then $%
e(\{a,b\})<\eta $. For some $g\in \mathbb{G}(\overline{\mathbb{Q}})$, $%
gag^{-1}\in T$, and since $e(\cdot )$ is invariant under conjugation by
elements from $\mathbb{G}(\overline{\mathbb{Q}}),$ we have $%
e(\{gag^{-1},gbg^{-1}\})<\eta .$ We can now apply Proposition \ref%
{heightbound} to see that we must have $h(\mathbf{f}(gbg^{-1}))<\mu $,
therefore $gbg^{-1}\in Z_{\mu }$ and hence $(gag^{-1},gbg^{-1})\in \widehat{Z%
}_{\mu }$. which gives the desired contradiction. Hence $e(F)>\varepsilon $
and we are done.

\subsection{Proof of Proposition \protect\ref{minheight}.\newline
}

\bigskip

\bigskip

\textit{Reduction to the adjoint representation.}

We first reduce to proving the statement of Proposition \ref{minheight} for
the adjoint representation and the \textquotedblleft Killing
height\textquotedblright\ $h_{Kill}$. Changing $\mathbb{G}$ into its image
in $SL(V)$ via $\rho ,$ we may assume that $\rho $ is non trivial on each
simple factor of $\mathbb{G}$. Let $Ad,\mathfrak{g}$ be the adjoint
representation of $\mathbb{G}$ and let $h_{Kill}$ be the \textquotedblleft
Killing height\textquotedblright\ introduced in Paragraph \ref{reducadj}.
According to Proposition \ref{welldefheight} and its proof there exists a
constant $C_{\rho }\geq 1$ and a basis of $V$ giving rise to an associated
height function $h$ on $End(V),$ such that $\frac{1}{C_{\rho }}\cdot
h_{Kill}(F)-C_{\rho }^{\prime }\leq h(\rho (F))\leq C_{\rho }\cdot
h_{Kill}(F)+C_{\rho }^{\prime }$ \ for all $F$ (as mentioned in Paragraph %
\ref{reducadj} $h_{Kill}$ and the height associated to a Chevalley basis of $%
\mathfrak{g}$ only differ by an additive constant). Granting the conclusion
of Proposition \ref{minheight} for the adjoint representation, we obtain $%
g\in \mathbb{G}(\overline{\mathbb{Q}})$ such that $h(\rho (gFg^{-1}))\leq
CC_{\rho }^{2}\cdot \widehat{h}(\rho (F))+C_{\rho }^{\prime }.$ But by the
main Theorem \ref{main}, since $F$ generates a non virtually solvable group,
we have $C_{\rho }^{\prime }\leq C^{K}\cdot \widehat{h}(\rho (F))$ and $%
CC_{\rho }^{2}\leq C^{K}$ for some $K=K(d)\in \mathbb{N}$ independent of $F.$
Hence $h(\rho (gFg^{-1}))\leq 2C^{K}\cdot \widehat{h}(\rho (F)).$ The
remaining inequalities are clear or follow from the basic properties of
heights explained in Section 2.

\bigskip

\textit{Proof of Proposition \ref{minheight} for the adjoint representation.}

We thus assume that $\rho =Ad$ and $h=h_{Kill},$ while $\mathbb{G}$ is
semisimple of adjoint type and $\left\langle F\right\rangle $ is Zariski
dense in $\mathbb{G}$. Again let $T$ be a maximal torus in $\mathbb{G}$ and
pick a corresponding basis of $\mathfrak{g}_{\mathbb{Z}}$ made of weight
vectors say $(Y_{1},...,Y_{d})$ as in Section \ref{local}. Since $\mathbb{G}$
is of adjoint type, it is a direct product of its simple factors. Looking at
the projection of $F$ to each simple factors, it is straightforward to
verify that, when proving Proposition \ref{minheight}, we can reduce to the
case when $\mathbb{G}$ is absolutely simple. So we assume $\mathbb{G}$
absolutely simple.

Clearly, if we prove the statement for a bounded power of $F$ instead, then
this will prove the statement for $F$. Hence making use of escape (i.e.
applying Proposition \ref{twoelements}), and after possibly conjugating $F$
by an element of $\mathbb{G}(\overline{\mathbb{Q}}),$ we may assume that $F$
contains two elements $a,b$ which generate a subgroup acting irreducibly on $%
\mathfrak{g}$ and such that $a$ is a regular semisimple element in $T$ and $%
b $ is generic with respect to $T,$ i.e. such that the matrix coefficient $%
B_{ij}$ of $Ad(b)$ in the basis $(Y_{1},...,Y_{d})$ is non zero for any
indices $i,j$. We thus write $F=\{a,b,b_{1},...,b_{M}\}.$

Let $S\subset \lbrack 1,d]$ be the set of indices corresponding to the
simple roots. So $|S|=rk(\mathbb{G}).$ Let $I_{r}\subset \lbrack 1,d]$ be
the set of indices corresponding to the $Y_{i}$'s that belong to $\mathfrak{t%
}=Lie(T).$ For each $j\in S$, let us choose some $i_{j}\in I_{r}$. We have $%
B_{i_{j}j}B_{ji_{j}}\neq 0.$ Then one can choose a unique point $t\in T(%
\overline{\mathbb{Q}}\mathbb{)}$ such that $\alpha _{j}(t)^{2}=\frac{%
B_{i_{j}j}}{B_{ji_{j}}}$ for each $j\in S$. As we may, we change $F$ into $%
tFt^{-1}.$ Then $B_{i_{j}j}=B_{ji_{j}}$ for every $j\in S$. Moreover we know
from $(\ref{Bbound})$ that for any place $v$ and any real number $%
s_{v}>E_{v}^{Ad}(F),$ there exists $t_{v}\in T(\overline{\mathbb{Q}_{v}})$
such that
\begin{equation*}
\left\Vert Ad(b^{t_{v}})\right\Vert _{v}\leq C_{v}^{d}\cdot \left(
\prod_{k=1}^{p}\max \{1,L_{k}(a)_{v}\}\right) ^{dC_{2}}\cdot
e^{s_{v}(1+dC_{1})}
\end{equation*}%
where $C_{1},C_{2},C_{\infty }$ are positive constants independent of $v$
and $C_{v}=1$ if $v$ is non archimedean, while $C_{v}=C_{\infty }$ if $v$ is
archimedean. Since every matrix coefficient of $Ad(b^{t_{v}})$ is bounded by
$\left\Vert Ad(b^{t_{v}})\right\Vert _{v}$ if $v$ is non archimedean and by
a constant multiple of this norm if $v$ is archimedean, up to enlarging $%
C_{\infty }$ if necessary we get that the same bound holds for all matrix
coefficients of $Ad(b^{t_{v}})$, i.e.
\begin{equation}
\log ^{+}|\alpha _{i}\alpha _{j}{}^{-1}(t_{v})B_{ij}|_{v}\leq d\log
C_{v}+dC_{2}\sum_{k=1}^{p}\log ^{+}L_{k}(a)_{v}+(1+dC_{1})s_{v}=:r_{v}(a)
\label{rel}
\end{equation}%
Specializing this for $B_{ij}=B_{ji}$ when $j\in S$ and $i=i_{j}$ and
adding, we obtain
\begin{equation*}
2\log ^{+}|B_{ij}|_{v}=\log ^{+}|B_{ij}B_{ji}|_{v}\leq 2r_{v}(a)
\end{equation*}%
On the other hand
\begin{eqnarray}
\frac{1}{[K:\mathbb{Q}]}\sum_{v\in V_{K}}n_{v}\cdot r_{v}(a) &\leq &d\log
C_{\infty }+dC_{2}\sum_{k=1}^{p}(h(\delta _{k}^{-1})+\log ^{+}\frac{1}{%
\kappa })+(1+dC_{1})e(F)  \notag \\
&\leq &C_{\infty }^{\prime }+(1+dC_{1}+dpC_{2})e(F)  \label{anotherline}
\end{eqnarray}%
where $C_{\infty }^{\prime }$ is another positive constant, $\delta
_{k}=1-\alpha _{k}(a)$ for $k\in S$, $\kappa =\min_{k\in S}|\alpha
_{k}|_{\infty }$ as in \S \ref{prelimupper} above, and where we have used $%
h(\delta _{k}^{-1})=h(\delta _{k})\leq h(\alpha _{k}(a))+\log 2\leq
e(F)+\log 2$. Hence for $j\in S$ and $i=i_{j},$%
\begin{equation}
h(B_{ij})\leq C_{\infty }^{\prime }+(1+dC_{1}+dpC_{2})e(F)
\label{anotherline2}
\end{equation}%
On the other hand, since $i\in I_{r}$ $\alpha _{i}=1$ and $(\ref{rel})$
gives
\begin{equation*}
\log ^{+}|\alpha _{j}{}^{\pm 1}(t_{v})B_{ij}|_{v}\leq r_{v}(a)
\end{equation*}%
\begin{equation*}
\log ^{+}|\alpha _{j}{}^{\pm 1}(t_{v})|_{v}\leq r_{v}(a)+\log ^{+}\left\vert
\frac{1}{B_{ij}}\right\vert _{v}
\end{equation*}%
Taking the weighted sum over all places, we get
\begin{equation*}
h(\alpha _{j}(t_{v})_{v}),h(\alpha _{j}{}^{-1}(t_{v})_{v})\leq h(\frac{1}{%
B_{ij}})+\frac{1}{[K:\mathbb{Q}]}\sum_{v\in V_{K}}n_{v}\cdot r_{v}(a)
\end{equation*}%
which, as $h(B_{ij}^{-1})=h(B_{ij})$ gives from $(\ref{anotherline})$ and $(%
\ref{anotherline2})$%
\begin{equation}
h(\alpha _{j}(t_{v})_{v}),h(\alpha _{j}{}^{-1}(t_{v})_{v})\leq 2C_{\infty
}^{\prime }+2(1+dC_{1}+dpC_{2})e(F)  \label{rootbound}
\end{equation}%
Now let $\alpha $ be an arbitrary root, i.e. $\alpha =\prod_{j\in S}\alpha
_{j}{}^{n_{j}}$ for some integers $n_{j}\in \mathbb{Z}$. Since there are
only finitely many possibilities for the $n_{j}$'s given $\mathbb{G}$, there
is a bound, say $N,$ for the possible sums $\sum |n_{j}|.$ Hence $(\ref%
{rootbound})$ gives
\begin{equation*}
h(\alpha (t_{v})_{v})\leq 2NC_{\infty }^{\prime }+2N(1+dC_{1}+dpC_{2})e(F)
\end{equation*}%
for every root $\alpha .$ Finally, if $i$ and $j$ are arbitrary indices this
time, we get from $(\ref{rel})$ and $(\ref{anotherline})$
\begin{eqnarray*}
h(B_{ij}) &\leq &\frac{1}{[K:\mathbb{Q}]}\sum_{v\in V_{K}}n_{v}\cdot
r_{v}(a)+h(\alpha _{i}{}^{-1}(t_{v})_{v})+h(\alpha _{j}(t_{v})_{v}) \\
&\leq &(4N+1)C_{\infty }^{\prime }+(4N+1)(1+dC_{1}+dpC_{2})e(F)
\end{eqnarray*}%
Since $A_{ij}=0$ for $i\neq j$ while $h(A_{ii})\leq e(F)$ by Proposition $(%
\ref{propbis})$ (c), we finally get $h_{Kill}(A)+h_{Kill}(B)\leq
O_{d}(1)\cdot (\sum_{ij}h(A_{ij})+h(B_{ij}))\leq C+C\cdot e(F)$.

Now recall that $a$ and $b$ were chosen so that they generate a subgroup
which acts irreducibly on $\mathfrak{g}(\overline{\mathbb{Q}}).$ By
Burnside's theorem, this means that $Ad(a)$ and $Ad(b)$ generate $End(%
\mathfrak{g})$ as an associative $\overline{\mathbb{Q}}$-algebra. In
particular, one can find $d^{2}$ elements, say $u_{1},...,u_{d^{2}},$ in $%
\{Id,Ad(a),Ad(b)\}^{d^{2}}$ which form a basis of $End(\mathfrak{g})$ over $%
\overline{\mathbb{Q}}.$ Clearly $h_{Kill}(u_{i})\leq d^{2}(C+Ce(F))$ for
each $i=1,...,d^{2}.$ Let $E_{ij}$ be the elementary matrices associated to
our basis $(Y_{1},...,Y_{d})$ of $\mathfrak{g}$. We may write $u_{i}$ as a
linear combination $\sum U_{kl}^{(i)}E_{kl}$ with $U_{kl}^{(i)}\in \overline{%
\mathbb{Q}}.$ By definition of the height $h=h_{Kill}$ on $End(\mathfrak{g})$%
, it differs from the height associated to the basis $(Y_{1},...,Y_{d})$
only by an additive constant $C_{\infty }$ due to the fact that the $Y_{i}$%
's are not necessarily orthogonal at infinite places. Thus each height $%
h(U_{kl}^{(i)})$ is at most $h(u_{i})+C_{\infty }$. In particular the height
of the determinant of $(u_{1},...,u_{d^{2}})$ in the basis of the $E_{ij}$
is bounded in terms of the $h(u_{i})$ hence in terms of $e(F)$ only. As a
result, if we write each $E_{ij}$ as a linear combination $\sum
x_{k}^{(ij)}u_{k}$ with $x_{k}^{(ij)}\in \overline{\mathbb{Q}},$ then the
height $h(x_{k}^{(ij)})$ is bounded in terms of $e(F)$ (and $d$) only, i.e. $%
\leq C_{\infty }^{\prime \prime }+O_{d}(1)\cdot e(F)$ for some other
constant $C_{\infty }^{\prime \prime }>0$ depending on $d$ only.

Let $c$ be any element of $F=\{a,b,b_{1},...,b_{M}\}$. Then we may write $%
C=Ad(c)=\sum C_{ij}E_{ij}$ and $C_{ij}=(E_{ji}C)_{jj}=\sum
x_{k}^{(ij)}(u_{k}C)_{jj}.$ Now observe that we may apply $(\ref{Bbound})$
to the two matrices $\{Ad(a),u_{k}C\}$ and get as in $(\ref{rel})$ for each
place $v$ and all $j=1,...,d$%
\begin{equation*}
\log ^{+}|(u_{k}C)_{jj}|_{v}\leq d\log C_{v}+dC_{2}\sum_{k=1}^{p}\log
^{+}L_{k}(a)_{v}+(1+dC_{1})s_{v}=r_{v}(a).
\end{equation*}%
We may now estimate $\log ^{+}||F||_{v}.$ First if $v$ is non archimedean
one gets $\log ^{+}||F||_{v}\leq \log
^{+}\max_{k,j,c}|(u_{k}C)_{jj}|_{v}+\log ^{+}\max_{k,i,j}|x_{k}^{(ij)}|_{v}$
and%
\begin{equation*}
\log ^{+}||F||_{v}\leq r_{v}(a)+\sum_{k,i,j}\log ^{+}|x_{k}^{(ij)}|_{v}
\end{equation*}%
while if $v$ is archimedean we get the same estimate plus an additive error.
Summing over the places as in $(\ref{anotherline2})$ we have%
\begin{equation*}
h_{Kill}(F)\leq C_{\infty }^{\prime \prime }+\sum_{k,i,j}h(x_{k}^{(ij)})+%
\frac{1}{[K:\mathbb{Q}]}\sum_{v\in V_{K}}n_{v}r_{v}(a)
\end{equation*}%
And thus $h_{Kill}(F)\leq O_{d}(1)(1+e(F)).$ Using Theorem \ref{main}, this
upper bound can be replaced by $h_{Kill}(F)\leq O_{d}(1)\cdot e(F),$ and
Proposition \ref{minheight} is proved.%
\endproof%

\begin{remark}
\label{poscharadjoint}In positive characteristic $p$ with $p$ not $2$ nor $3$
and $\mathbb{G}$ not of type $A_{n},$ the adjoint representation is
irreducible and the above proof continues to hold verbatim without having to
appeal to Theorem \ref{main} at the end because no additive constant appears
in the upper bound (since all places are non archimedean). In the cases
where the adjoint representation is not irreducible, one can modify the
above proof to make it work for every irreducible rational representation
instead of $Ad$. One has to take a set of linearly independent weights $\chi
_{j}$ in place of the simple roots in order to define the conjugating
element $t\in T$, and then modify $(\ref{Bbound})$ accordingly. Details are
left to the reader.
\end{remark}

\textit{Proof of Proposition \ref{minh} from the introduction. }

Let $\mathbb{G}$ be the Zariski closure of $F$ in $GL_{d}.$ Since we are in
characteristic $0$, $\mathbb{G}$ is completely reducible when acting on $%
\overline{\mathbb{Q}}^{d}.$ Since there are only finitely many isomorphism
classes of semisimple algebraic subgroups of $GL_{d}$ and finitely many
isomorphism classes of irreducible representations of $\mathbb{G}$ of
dimension at most $d,$ we may consider the maximum of all constants $C\geq 1$
appearing in Proposition \ref{minheight} for the various semisimple groups $%
\mathbb{G}$ and representations that can arise. Thus Proposition \ref%
{minheight} gives a basis of $V$ with height $h_{0}$ and $g_{0}\in \mathbb{G}%
(\overline{\mathbb{Q}})$ such that $h_{0}(g_{0}Fg_{0}^{-1})\leq C\widehat{%
h_{0}}(F).$ But there is $g\in GL_{d}(\overline{\mathbb{Q}})$ such that $%
h(\cdot )=h_{0}(g\cdot g^{-1})$ and $\widehat{h}=\widehat{h_{0}},$ so we are
done.%
\endproof%

\bigskip

\subsection{Proof of Corollaries \protect\ref{eig} and \protect\ref{eig1}}

First we assume that $F$ generates a non virtually solvable group. From
Lemma \ref{CompLem}, we have for any set $F$ containing $1$, $\sum_{a\in
F^{d^{2}}}e(\{a\})\geq e(F)-|\log c|$. In particular
\begin{equation}
\max \{e(\{a\}),a\in F^{nd^{2}}\}\geq \frac{1}{|F|^{nd^{2}}}(n\widehat{h}%
(F)-|\log c|)  \label{anotherineq}
\end{equation}%
and for every $n\in \mathbb{N}$. Now by Theorem \ref{main}, we have $%
\widehat{h}(F)>\varepsilon =\varepsilon (d)>0$. Hence for some $%
n_{0}=n_{0}(d)\in \mathbb{N}$,
\begin{equation*}
\max \{e(\{a\}),a\in F^{n_{0}}\}\geq \frac{d}{|F|^{n_{0}}}\cdot \widehat{h}%
(F).
\end{equation*}%
On the other hand, we clearly have $e(\{a\})\leq \sum h(\lambda )$ where the
sum is over the $d$ eigenvalues of $a.$ Hence the assertion of Corollary \ref%
{eig}.

Now assume that $F$ generates a virtually solvable subgroup. It is well
known (see \cite{Werf} 3.6 and 10.10) that there is an integer $%
n_{0}=n_{0}(d)\in \mathbb{N}$ such that any virtually solvable subgroup of $%
GL_{d}(\mathbb{C})$ contains a subgroup of index at most $n_{0}$ which can
be conjugated inside the upper-triangular matrices. Applying Lemma \ref%
{finiteindex} (and its proof), we may find $F_{1}\subset F^{2n_{0}-1}$ such
that $F^{n}\cap B\subset (F_{1}\cup F_{1}^{-1})^{2n}$ for all $n,$ where $%
B=T_{d}(\mathbb{C})$ is the subgroup of upper-triangular matrices. But $%
F^{n}=\cup (F^{n}\cap f_{i}^{-1}B)$ for at most $n_{0}$ elements $f_{i}$ in $%
F^{n_{0}}.$ Hence $F^{n}\subset \cup f_{i}^{-1}(F^{n+n_{0}}\cap B)$ and $%
R_{v}(F)\leq \lim \inf ||F^{n}\cap B||^{1/n}\leq R_{v}(F_{1}\cup
F_{1}^{-1})^{2}.$ However, since $F_{1}\subset B,$ it is straightforward to
observe that $R_{v}(F_{1}\cup F_{1}^{-1})=\Lambda _{v}(F_{1}\cup
F_{1}^{-1}). $ Summing over all places, we obtain $\widehat{h}(F)\leq
2\sum_{a\in F_{1}}e(\{a\})+e(\{a^{-1}\})\leq 2|F|^{2n_{0}}\max \{\sum
h(\lambda )+h(\lambda ^{-1}),\lambda $ eigenvalue of $a\in F_{1}\}.$ Since $%
h(\lambda )=h(\lambda ^{-1}),$ we get the desired result.

Now we turn to Corollary \ref{eig1}. By the remark above on the bound $n_{0}$
of the index of a triangular subgroup in any virtually solvable subgroup, it
is easy to see that the set of pairs $(A,B)$ in $GL_{d}\times GL_{d}$ that
generate a virtually solvable subgroup is a closed subvariety. Since every
connected simple algebraic group can be topologically generated (for the
Zariski topology) by two elements (see Proposition\ \ref{2elem}), we can
apply the escape from subvarieties lemma (Lemma \ref{Bezout}) and conclude
that there is a pair $\{A,B\}$ in $F^{c(d)}$ which generates a non virtually
solvable subgroup of $\left\langle F\right\rangle .$ Then apply Corollary %
\ref{eig} to $\{Id,A,B\}.$

\subsection{Proof of Corollaries \protect\ref{torsion}, \protect\ref{corobv}
and \protect\ref{geomint}}

\proof[Proof of Corollary \ref{torsion}]%
Let $k$ be the algebraic closure of $K$ and $\Gamma $ the subgroup generated
by $F$. First assume that $\Gamma \leq GL(W)$ acts absolutely irreducibly on
$W=k^{d}$. According to Burnside's theorem the $k$-subalgebra generated by
the elements of $\Gamma $ is the full algebra $End_{k}(W).$ Since $D=\dim
End_{k}(W)=(\dim W)^{2}\leq d^{2}$, there exists a linear basis, say $%
w_{1},...,w_{D}$ of $End_{k}(W)$ in $F^{d^{2}}$ (start with $w_{1}=1,$ then
multiply by the elements of $F$ one after the other). Since $\{x\mapsto
tr(zx)\}_{z\in End_{k}(W)}$ account for all linear forms on $End_{k}(W),$
the linear forms $x\mapsto tr(w_{i}x)$ must be linearly independent, and the
matrix $\{tr(w_{i}w_{j})\}_{1\leq i,j\leq D}$ is invertible. Let $L$ be the
field generated by the eigenvalues of all elements of $F^{2d^{2}+1}.$ Note
that $L$ contains $tr(w_{i}w_{j})$ and $tr(fw_{i}w_{j})$ for $f\in F$ and
all $i,j.$ We claim that $\Gamma \leq \bigoplus\limits_{1\leq i\leq
D}Lw_{i}\leq End_{k}(W).$ Indeed for each $i,$ and each $f\in F,$ write $%
fw_{i}=\sum a_{ij}w_{j}$ for some $a_{ij}\in k$. Then as $%
\{tr(w_{i}w_{j})\}_{1\leq i,j\leq D}$ is invertible, the $a_{ij}$ must
belong to $L.$ Since $w_{1}=1,$ we see that positive words in $F$ lie all in
$\bigoplus\limits_{1\leq i\leq D}Lw_{i}.$ On the other hand, the
Cayley-Hamilton theorem implies that $f^{-1}\in L[f].$ Finally $\Gamma \leq
\bigoplus\limits_{1\leq i\leq D}Lw_{i}$ as claimed. The left regular
representation of $\Gamma $ on $\bigoplus\limits_{1\leq i\leq D}Lw_{i}$
gives us a faithful representation of $\Gamma $ in $GL_{D}(L).$ If $%
F^{2d^{2}+1}$ consists only of torsion elements, the field $L,$ is generated
over its prime field by finitely many roots of unity. If $char(K)>0$ this
already implies that $L$ is finite and thus that $\Gamma $ is finite, a
contradiction. If $char(K)=0,$ then $L$ belongs to $\overline{\mathbb{Q}}$
and we are thus reduced to the case when $\Gamma $ lies in $GL_{D}(\overline{%
\mathbb{Q}})$. Then, by the combination of Corollary \ref{eig} with Theorem %
\ref{main} we are done unless $\Gamma $ is virtually solvable.

If $\Gamma $ does not act irreducibly of $k^{d},$ let $\{0\}\leq V_{1}\leq
...\leq V_{k}=k^{d}$ be a composition series for $\Gamma $ and let $%
W=V_{i_{0}}/V_{i_{0}+1}$ be an (irreducible) composition factor. If $%
char(K)>0,$ by the above, the image of $\Gamma $ is $GL(W)$ is finite. It
follows that $\Gamma $ is virtually unipotent and hence finite, because
finitely generated unipotent subgroups in positive characteristic are finite.

If $char(K)=0,$ then the image of $\Gamma $ on each composition factor is
virtually solvable, and hence $\Gamma $ itself is virtually solvable. Recall
that there is an integer $n_{0}=n_{0}(d)\in \mathbb{N}$ such that any
virtually solvable subgroup of $GL_{d}(\mathbb{C})$ contains a subgroup of
index at most $n_{0}$ which can be conjugated inside the upper-triangular
matrices (see \cite{Werf} 3.6 and 10.10). Applying Lemma \ref{finiteindex},
we may assume without loss of generality that $F$ is made of
upper-triangular matrices. Then for every $a,b\in F$, the commutator $[a,b]$
is a unipotent matrix in $SL_{d}(\mathbb{C})$, hence is either trivial or of
infinite order. If one of them has infinite order, we are done. Otherwise
this means that the matrices in $F$ commute. But a finitely generated
abelian group generated by torsion elements is finite. We are done.

The argument above works verbatim without the need to take inverses until
the point in the last paragraph when $F$ is assumed to consist of
upper-triangular matrices. Note that if the elements of $F$ are torsion,
then their eigenvalues are roots of unity, hence the group generated by $F$
is virtually nilpotent. This completes the proof of the corollary.
\endproof%

\proof[Proof of Corollary \ref{corobv} from the Introduction]%
If $\gamma $ has a transcendental eigenvalue for some $\gamma \in
F^{2d^{2}+1},$ then the second alternative obviously holds. If no $\gamma
\in F^{2d^{2}+1}$ has a transcendental eigenvalue, then the argument given
in the proof of Corollary \ref{torsion} shows that $\Gamma $ has a faithful
representation in $GL_{d^{2}}(\overline{\mathbb{Q}})$. So we are reduced to
this situation and the claim is clear by Corollary \ref{corobv}.%
\endproof%

\proof[Proof of Corollary \ref{geomint} from the Introduction]%
If $F$ fixes a point in the Bruhat-Tits building $X_{k}$ of $SL_{d}$ over a $%
p$-adic field $k$, then $F$ fixes a vertex of $X_{k}$ (it fixes the vertices
of the smallest simplex containing the fixed point). But vertices of $X_{k}$
are permuted transitively by the action of $GL_{d}(k)$. If follows from
Lemma \ref{displacement} that $E_{k}(F)=1.$ Hence if $F$ fixes a point on
each $X_{k}$ for $k$ non archimedean, then $e_{f}(F)=0.$ Hence by Theorem %
\ref{main} we must have $e_{\infty }(F)>\varepsilon .$ Thus there exists an
embedding $\sigma $ of $K$ in $\mathbb{C}$ such that $\log E_{\mathbb{C}%
}(\sigma (F))>\varepsilon .$ Then by Lemma \ref{displacement}, every point
of $X_{\mathbb{C}}$ must be moved by at least $\varepsilon $ by some element
of $F$.
\endproof%

\begin{Ack}
This paper grew out of my attempts to improve the results of [21] and [14]. I wish to thank T. Gelander for our past collaboration and joint works which naturally led me to study the questions addressed in this paper.

I thank the referee for his careful reading of the paper and many comments that helped improve the exposition. I am grateful to P. Sarnak and A. Yafaev from whom I learned about Bilu's theorem. I thank E. Bombieri, J-B. Bost and A. Chambert-Loir for their insights about diophantine geometry and P. E. Caprace for his invaluable help with buildings. I also thank G. Chenevier, E. Lindenstrauss, G. Prasad, A. Salehi-Golsefidy and V. Talamanca for stimulating conversations. Finally I acknowledge the generous support of the European Research Council through Grant GADA-208091.
\end{Ack}


\begin{thebibliography}{99}
\bibitem{AmoD} Amoroso F. and David S., \textit{Le probleme de Lehmer en
dimension superieure}, J. Reine Angew. Math. 513 (1999), 145--179.

\bibitem{AmDv} Amoroso F., Dvornicich R.,\textit{\ A lower bound for the
height in abelian extensions}, J. Number Theory 80 (2000), no.2. 260--272.

\bibitem{AM} Atiyah M., Macdonald I., \textit{Introduction to commutative
algebra}, Addison-Wesley series in Mathematics, (1969).

\bibitem{BakRum} Baker M., Rumely R., \textit{Equidistribution of small
points, rational dynamics, and potential theory}, Ann. Inst. Fourier 56
(2006), no 3., 625--688.

\bibitem{BarCor} Bartholdi L., Cornulier Y., \textit{Infinite groups with
large balls of torsion elements and small entropy}, Archiv der Mathematik
87(2), 104-112, (2006).

\bibitem{Bilu} Bilu, Y, \textit{Limit distribution of small points on
algebraic tori}, Duke Math. J. \textbf{89} (1997), no. 3, 465--476.

\bibitem{Bom} Bombieri, E., Gubler, W.,\textit{\ Heights in Diophantine
geometry}, New Mathematical Monographs, 4. Cambridge University Press,
Cambridge, (2006).

\bibitem{Bor} Borel, A., \textit{Linear algebraic groups}, Notes taken by
Hyman Bass W. A. Benjamin, Inc., New York-Amsterdam 1969

\bibitem{Bourb} Bourbaki, N. \textit{Groupes et Alg\`{e}bres de Lie},
Chapitres 4-5-6 and 7-8, Hermann ed.

\bibitem{BreSol} Breuillard, E., \textit{On uniform exponential growth for
solvable groups}, Pure and Applied Math. Quart. 3, no 4, Margulis Volume
Part 1, 949--967, (2007).

\bibitem{uti} Breuillard, E, Gelander, T., \textit{Uniform independence in
linear groups}, Invent. Math. \textbf{173}, no 2, 225--263, (2008).

\bibitem{Bre} Breuillard, E., \textit{Heights on }$\mathit{GL}_{2}$\textit{\
and free subgroups}, in Geometry, Rigidity and Group Actions, Zimmer's
Festschrift, B. Farb and D. Fisher eds., Chicago Univ. Press. (2011).

\bibitem{B} Breuillard, E., \textit{A strong Tits alternative}, preprint
April 2008.

\bibitem{BreEffec} Breuillard, E. \textit{Effective estimates for the
spectral radius of a finite set of matrices}, preprint.

\bibitem{BH} Bridson M., Haefliger A., \textit{Metric spaces of non-positive
curvature}, Springer-Verlag , (1999), vii, 643 p.

\bibitem{BT} Bruhat F., Tits J., \textit{Groupes r\'{e}ductifs sur un corps
local}, Publ. Math. IHES \textbf{41}\ (1972) 5-252.

\bibitem{Cap} Caprace P.E., personal communication.

\bibitem{Cham} Chambert-Loir A., \textit{Mesures et equidistribution sur les
espaces de Berkovich}, J. Reine. Angew. Math. 585 (2006), 215--235.

\bibitem{CurR} Curtis C.W., Reiner I., \textit{Representation Theory of
Finite Groups and Associative Algebras}, (Interscience, New York) (1962).

\bibitem{Eber} Eberlein P., \textit{Geometry of nonpositively curved
manifolds}, Chicago Lectures in Math. (1996).

\bibitem{EMO} Eskin A., Mozes S., Oh H., \textit{On uniform exponential
growth for linear groups}, Invent. Math. \textbf{160} (2005), no. 1, 1--30

\bibitem{FL} Favre C. and Rivera-Letelier J., \textit{Equidistribution
quantitative des points de petite hauteur sur la droite projective}, Math.
Ann. 335 (2006), no.2., 311-361.

\bibitem{Gel} Gelander, T. \textit{Homotopy type and volume of symmetric
spaces}, Duke Math. J. 2004.

\bibitem{IwaMat} Iwahori, N., Matsumoto, H., \textit{On some Bruhat
decomposition and the structure of the Hecke rings of }$p$\textit{-adic
Chevalley groups}, Inst. Hautes \'{E}tudes Sci. Publ. Math. No. \textbf{25}
(1965) 5--48.

\bibitem{Jab} Jacobson, N.,\textit{\ Lie algebras}, Interscience Dover
(1962).

\bibitem{KM} Kazhdan, D., Margulis, G., \textit{A proof of Selberg's
hypothesis}, Mat. Sb. (N.S.) \textbf{75} (117) 1968 163--168

\bibitem{Land} Landvogt, E., \textit{Some functorial properties of the
Bruhat-Tits building}, J. Reine Angew. Math. \textbf{518} (2000), 213--241.

\bibitem{Lang0} Lang, S., \textit{Algebra}, Revised 3rd edition, GTM\ 211,
Springer-Verlag, New York, (2002).

\bibitem{Lang} Lang, S., \textit{Fundamentals of Diophantine geometry},
Springer-Verlag, New York, (1983).

\bibitem{MW} Masser, Wustholz, \textit{Fields of large transcendence degree
generated by values of elliptic functions}, Invent. Math. (1983).

\bibitem{Mos} Mostow, G. D., \textit{Self-adjoint groups}, Ann. of Math. (2)
\textbf{62}, (1955). 44--55.

\bibitem{OnVin} Onishchik, A. L.; Vinberg, \`{E}. B. \textit{Lie groups and
algebraic groups}, Translated from the Russian and with a preface by D. A.
Leites. Springer Series in Soviet Mathematics. Springer-Verlag, Berlin,
(1990).

\bibitem{Pi} Pineiro J., Szpiro L., and Tucker T., \textit{Mahler measure
for dynamical systems on and intersection theory on a singular arithmetic
surface}, in Geometric methods in algebra and number theory, 219--250,
Progr. Math., 235, Birkhauser Boston, 2005.

\bibitem{Rag} Raghunathan M.S., \textit{Discrete Subgroups of Lie Groups, }%
Ergebnisse der Mathematik und Ihrer Grenzgebiete. Band \textbf{68} (1972).

\bibitem{Sch} Schinzel, A. \textit{Polynomials with special regard to
reducibility}, With an appendix by Umberto Zannier. Encyclopedia of
Mathematics and its Applications, 77. Cambridge University Press, Cambridge,
(2000).

\bibitem{Schur} Schur I., \textit{Uber Gruppen periodischer Substitutionen},
Sitzber. Preuss. Akad. Wiss. (1911), 619-627.

\bibitem{Sh} Shalom Y., \textit{Explicit Kazhdan constants for
representations of semisimple and arithmetic groups}, Ann. Inst. Fourier,
\textbf{50} (2000), no. 3, 833--863.

\bibitem{Sm} Smyth, C. \textit{The Mahler measure of algebraic numbers: a
survey}, in Number theory and polynomias, 322--349, London Math. Soc. Lecure
Note Ser., 352, Cambridge Univ. Press, Cambridge, (2008).

\bibitem{SUZ} Szpiro L., Ullmo E., Zhang S., \textit{Equir\'{e}partition des
petits points}, Invent. Math. \textbf{127} (1997), 337--347.

\bibitem{Stein} Steinberg R., \textit{Lectures on Chevalley groups}, Notes
prepared by John Faulkner and Robert Wilson. Yale University, New Haven,
Conn., (1968).

\bibitem{Tala} Talamanca V., \textit{A Gelfand-Beurling type formula for
heights on endomorphism rings}, J. Number Theory \textbf{83} (2000), no. 1,
91--105.

\bibitem{Tits} Tits J., \textit{Free subgroups of Linear groups}, Journal of
Algebra \textbf{20}\textit{\ }(1972), 250-270.

\bibitem{Thu} Thurston, W, \textit{Three-dimensional geometry and topology},
Vol. 1. Edited by Silvio Levy. Princeton Mathematical Series, \textbf{35}.
Princeton University Press, (1997).

\bibitem{Ull} Ullmo, E. \textit{Positivit\'{e} et discr\'{e}tion des points
alg\'{e}briques des courbes}, Ann. of Math. (2) \textbf{147} (1998), no. 1,
167--179.

\bibitem{Wang} Wang, H. C., \textit{Topics on totally discontinuous groups},
in Symmetric spaces (Short Courses, Washington Univ., St. Louis, Mo.,
1969--1970), pp. 459--487. Pure and Appl. Math., Vol. 8, Dekker, (1972).

\bibitem{Werf} Wehrfritz, B., \textit{Infinite linear groups. An account of
the group-theoretic properties of infinite groups of matrices}, Ergeb. Mat.
Grenz., \textbf{76}, Springer-Verlag, (1973).

\bibitem{Zhang} Zhang, S-W., \textit{Small points and adelic metrics}, J.
Algebraic Geom. 4 (1995), no. \textbf{2}, 281--300

\bibitem{Zhang0} Zhang, S-W., \textit{Equidistribution of small points on
abelian varieties}, Ann. of Math. (2) \textbf{147} (1998), no. 1, 159--165.
\end{thebibliography}
\end{document}